\documentclass{amsart}
\usepackage{amssymb}
\usepackage{pdflscape}
\usepackage{longtable}
\usepackage[table]{xcolor}
\usepackage[all, frame, curve]{xy}
\usepackage[
bookmarksopen=true,
pdftitle={Tau Signatures and Characters of Weyl Groups},
pdfauthor={Thomas Folz-Donahue, Steven Glenn Jackson, Todor Milev and Alfred G. No\"el},
pdfstartview={FitBH},
]{hyperref}

\newcommand{\hideme}[1]{#1}
\newcommand{\hidemeFinal}[1]{#1}
\renewcommand{\hidemeFinal}[1]{}

\newcommand{\extendedVersion}[1]{#1}

\newtheorem{theorem}{Theorem}[section]
\newtheorem{lemma}[theorem]{Lemma} 
\newtheorem{proposition}[theorem]{Proposition}
\newtheorem{corollary}[theorem]{Corollary}

\theoremstyle{definition}
\newtheorem{definition}[theorem]{Definition}
\newtheorem{example}[theorem]{Example}
\newtheorem{remark}[theorem]{Remark}

\theoremstyle{alg}

\numberwithin{equation}{section} 

\newcommand{\abs}[1]{\lvert#1\rvert}

\newcommand{\comps}[1]{\left\langle#1\right\rangle}
\DeclareMathOperator{\Ind}{Ind}

\DeclareMathOperator{\diag}{diag}

\DeclareMathOperator{\Hom}{Hom}
\DeclareMathOperator{\signFunctionSymbol}{sgn}
\DeclareMathOperator{\signSignature}{SignSig}
\DeclareMathOperator{\signSignatureGeneralized}{SignSigGen}

\DeclareMathOperator{\signRepSymbol}{Sign}
\newcommand{\signRep}[1]{\signRepSymbol #1 }
\newcommand{\signFunction}{\signFunctionSymbol}
\DeclareMathOperator{\signMult}{sgnMult}
\DeclareMathOperator{\tr}{tr}
\DeclareMathOperator{\id}{id}
\DeclareMathOperator{\linspan}{linspan}

\begin{document}
\title[Tau Signatures and Characters of  Weyl Groups]
 {Tau Signatures and Characters of  Weyl Groups}
 
\author[Folz-Donahue]{Thomas Folz-Donahue}
\author[Jackson]{Steven Glenn Jackson}
\author[Milev]{Todor Milev}
\author[No\"el]{Alfred G. No\"el}

\address{Department of Mathematics\\
         University of Massachusetts\\
         100 Morrissey Boulevard\\
         Boston, MA 02125-3393}

\email{eigenlambda@gmail.com}
\email{jackson@math.umb.edu}
\email{todor.milev@gmail.com}
\email{anoel@math.umb.edu}

\thanks{ Alfred No\"el is also Directeur du Centre de Rcherche Math\'ematique de l'Institut des Sciences, des Technologies, et des Etudes Avanc\'ees d' Ha\"iti (ISTEAH)}
\thanks{Jackson and No\"el were partially supported by NSF grant \#DMS 0554278.}
\thanks{They wish to thank David Vogan of the Massachusetts Institute of Technology for many stimulating discussions.}
\thanks{They also wish to thank Birne Binegar of Oklahoma State University for supplying Atlas data with clear explanation.}

\begin{abstract}
Let $G_{\mathbb R}$ be the set of real points of a complex linear reductive group and $\hat G_\lambda$ its classes of irreducible admissible representations with infinitesimal integral regular character $\lambda$.
In this case each cell of representations is associated to a \emph{special} nilpotent orbit.
This helps organize the corresponding set of irreducible Harish-Chandra modules. The goal of this paper is to is to describe algorithms for
identifying the special nilpotent orbit attached to a cell in terms
of descent sets appearing in the cell. 
\end{abstract}
\maketitle
\section{Introduction}\label{s:intro}\

Let ${\mathfrak g}$ be a complex reductive Lie algebra with adjoint group $G$ and Weyl group $ W$.
This paper describes a simple algorithm by which one can read off the complex nilpotent orbit associated with a cell representation of $ W$ - provided that $ W$ is of classical type.

Given any representation $V$ of $ W$ we define the \emph{sign signature} of $V$
\[
\signSignature V 
\]
to be the set of all parabolic subgroups $P\subseteq W$ (relative to a fixed simple basis) such that $V_{|P}$ contains a copy of the sign representation of $P$. The sign signature depends only on the conjugacy class of $P$.

For $W$ of classical type, we show that the irreducible representations of $W$ are determined by their sign signatures, and we give a simple algorithm by which one can use the sign signature to find the partition or partition-pair indexing a given irreducible representation.

Moreover, the \emph{$\tau$-signature}  of a cell representation $\mathcal{C}$ of $ W$ coincides with that of its unique special subrepresentation, and can also be identified with the collection of all parabolics $P$ such that the simple roots of $P$ are contained in some \emph{$\tau$-invariant} of the cell.
Combining this with the Springer correspondence, we obtain a simple method which computes the nilpotent orbit associated with the cell directly from the $\tau$-invariants.


Our algorithms for recovering irreducible representations from sign signatures are summarized in Theorems \ref{thm:TypeARecoverIrrepsAlgorithm}, \ref{thm:TypeBRecoverIrrepsAlgorithm}, \ref{thm:TypeDRecoverIrrepsAlgorithm}. We hope these algorithms will help to facilitate the practical implementation of the theory.


The parametrization of irreducible representations using sign signatures fails for the exceptional Weyl groups.
In $G_2$, it is possible to find two different irreducible representations of $ W$ with the same sign signature.
In order to distinguish all irreducible representations, we propose the notion of extended sign signature.
Given any representation $V$ of $ W$ we define the \emph{extended sign signature} of $V$ to be the set of all subgroups $ G\subseteq W$ generated by additively closed root subsystems such that $V|_{G}$ contains a copy of the sign representation of $G$.
We show that the extended sign signature determines uniquely the irreducible representations of $ W$ (Theorem \ref{thm:classicalAndExceptional}).

\section{Weak $ W$-Graphs and $\tau$-invariants}
In this section we fix notation and collect some basic facts.
Let $ W$ be a Weyl group, let $\Delta$ be a set of simple roots.
For any root $\alpha$ of $ W$, denote by $s_\alpha$ the reflection in the hyperplane orthogonal to $\alpha$.

\begin{definition}
\label{defn:wgraph}
A \emph{Weak $\mathbf W$-graph} is a triple $\Gamma = (V, m, \tau)$, where $V$ is a finite set, $m$ is a map from $V\times V$ into ${\mathbb C}$, and $\tau$ is a map from $V$ into the power set of $\Delta$, such that the linear transformations $s_{\alpha}:{\mathbb C} V\rightarrow {\mathbb C} V$ given by the formula 
\[
s_{\alpha} (v) = \begin{cases} -v & \text{if $\alpha\in\tau(v)$} \\\displaystyle v - \sum_{\alpha\in\tau(u)} m(u,v)u & \text{if $\alpha\not\in\tau(v)$} \end{cases}
\] define a representation of $ W$ on ${\mathbb C} V$.
The \emph{$\tau$-invariant} of $\Gamma$ is the image of the function $\tau$, i.e.\ the set $\{\tau(v)|v\in V\}$.
\end{definition}

\begin{remark} Note that our definition of weak $\mathbf W$-graph is not quite the same as that of $\mathbf W$-graph as it appears in Kazhdan and Lusztig 1979 paper on Representations of Coxeter groups and Hecke algebras, Ivent. Math 53 , 165-184.

\end{remark}
\begin{example}
The Coxeter graph of $W$, labeled tautologically, is a weak  $\mathbf W$-graph.
More precisely, take $V = \Delta$ and put $\tau(\alpha) = \{\alpha\}$, where $m(\alpha, \beta)$ is the element Cartan matrix $m(\alpha,\beta) = 2\frac{(\alpha,  \beta)}{(\alpha,\alpha)}$ of $W$.
Then $\Gamma = (V, m, \tau)$ is a weak $\mathbf W$-graph whose associated representation is isomorphic to the defining representation of $ W$.
\end{example}

\begin{remark}
The name \emph{$\tau$-invariant} will be justified by Corollary \ref{cor:inv}.
\end{remark}

Now let $A$ be any subset of $\Delta$, and define $ W(A)$ to be the subgroup generated by $\{s_{\alpha}|\alpha\in A\}$.
The elements of the group algebra ${\mathbb C}[ W]$ defined by 
\[
\begin{array}{rcl} 
Q(A) &=& \displaystyle \frac{1}{\abs{ W(A)} }\sum_{g\in W(A)} g \\
R(A) &=& \displaystyle \frac{1}{\abs{ W(A)} }\sum_{g\in W(A)} (\signFunction(g))g 
\end{array}
\]
act on any finite-dimensional $ W$-module by projection onto the sign and trivial components, respectively, of the restriction to $ W(A)$.

For a given weak $\mathbf W$-graph $\Gamma=(V,m,\tau)$, we decompose $V$ as the disjoint union of the following three subsets:
\[
\begin{array}{rcl}
    V(A,-) &=& \{v\in V|A\subseteq\tau(v)\}. \\
    V(A,0) &=& \{v\in V|A\cap\tau(v)\neq\emptyset\text{ and } A\not\subseteq\tau(v)\} \\
    V(A,+) &=& \{v\in V|A\cap\tau(v) =\emptyset\}
\end{array}
\]

\begin{proposition}
\label{prop:flag}
The subspaces of ${\mathbb C} V$ spanned by $V(A,-)$ and by $V(A,-)\cup V(A,0)$ are $ W(A)$-stable.
\end{proposition}

\begin{proof}
It follows immediately from Definition \ref{defn:wgraph} that $ W(A)$ acts by the sign character on the span of $V(A,-)$.
To see that the span of the union is $ W(A)$-stable, note that the union is $\{v|A\cap\tau(v)\neq\emptyset\}$ and appeal once more to Definition \ref{defn:wgraph}.
\end{proof}

\begin{proposition}
\label{prop:kernel}
The kernel of $Q(A)$ on ${\mathbb C} V$ is precisely the span of $V(A,-)\cup V(A,0)$.
\end{proposition}

\begin{proof}
Suppose $v$ belongs to the union and choose $\alpha\in A\cap\tau(v)$.
Let $X$ be a set of representatives for the coset space $ W(A)/\comps{s_\alpha}$, so that $ W(A)$ is the disjoint union of $X$ and $Xs_\alpha$.
Then we have
\[
\begin{array}{rcl}
Q(A)(v) &=&\displaystyle \frac{1}{\abs{ W(A)}}\sum_{g\in  W(A)} g(v) \\ 
&=&\displaystyle \frac{1}{\abs{ W(A)}}\left(\sum_{g\in X} \left(g(v)+ gs_\alpha(v)\right)\right) \\ 
&=& 0.
\end{array}
\]
where the last equality follows because $\alpha\in\tau(v)$, hence $s_\alpha(v)=-v$.

On the other hand, if $v\in V(A,+)$ then it follows from Definition \ref{defn:wgraph} and Proposition \ref{prop:flag} that $Q(A)(v) = v + u$ where $u$ lies in the span of $V(A,-)\cup V(A,0)$.
Consequently no non-zero element of the span of $V(A,+)$ can map to zero under $Q(A)$.
\end{proof}

\begin{corollary}
Let $\Gamma=(V,m,\tau)$ be a weak $\mathbf W$-graph.
Then ${\mathbb C} V$ contains a trivial representation for $ W(A)$ if and only if $A\cap\tau(v)=\emptyset$ for some $v\in V$.
\end{corollary}
\begin{proof}
It follows immediately from Proposition \ref{prop:kernel} that the multiplicity of the trivial representation equals the cardinality of $V(A,+)$.
\end{proof}

\begin{proposition}
\label{prop:image}
The image of $R(A)$ in ${\mathbb C} V$ is precisely the span of $V(A,-)$.
\end{proposition}

\begin{proof}
If $v\in V(A,-)$ then, by Definition \ref{defn:wgraph}, $R(A)(v)=v$.
Hence the span of $V(A,-)$ is contained in the image of $R(A)$.
But $R(A)$ is idempotent, so for equality it suffices to show that $\tr_V(R(A))$ equals the cardinality of $V(A,-)$.

Suppose that $v\not\in V(A,-)$, and choose $\alpha\in A-\tau(v)$.
Choose a set $X$ of coset representatives for $\comps{ s_\alpha} \backslash W(A)$, so that $ W(A)$ is the disjoint union of $X$ and $s_\alpha X$.
For any $g\in X$, write $g(v)=\sum_{w\in V} c_{g, w} w$.
Then
\[
\begin{array}{rcl}
\displaystyle R(A)(v) &=&\displaystyle \frac{1}{\abs{ W(A)}}\sum_{g\in W(A) } \signFunction(g) g(v) \\
&=& \displaystyle \frac{1}{\abs{ W(A)}} \sum_{ g\in X}(\signFunction(g)g(v) + \signFunction(s_\alpha g)s_\alpha(g(v))) \\
&=& \displaystyle \frac{1}{\abs{ W(A)}} \sum_{g\in X}\sum_{w\in V}c_{g, w} \signFunction(g)( w- s_\alpha(w)) \\
&=&\displaystyle \frac{1}{\abs{ W(A)} }\sum_{ g\in X}\signFunction(g)\left( \sum_{ \alpha \in \tau(w)}2c_{g,w}w + \sum_{ \alpha \not \in \tau(w)} c_{g,w} \sum_{ \alpha \in \tau ( u)}m(u,w)u\right)
\end{array}
\]
which lies in the span of $\{w\in V|\alpha\in\tau(w)\}$.
Since $v$ lies outside this set, it makes no contribution to the trace of $R(A)$.
Combined with the observations of the first paragraph, this proves the result.
\end{proof}

\begin{corollary}
\label{cor:sign}
Let $\Gamma=(V,m,\tau)$ be a weak $\mathbf W$-graph.
Then ${\mathbb C} V$ contains a sign representation for $ W(A)$ if and only if $A\subseteq\tau(v)$ for some $v\in V$.
\end{corollary}

\begin{proof}
It follows immediately from Proposition \ref{prop:image} that the multiplicity of the sign representation equals the cardinality of $V(A,-)$.
\end{proof}

\begin{theorem}
The following are equivalent:
\begin{enumerate}
    \item There is some $v\in V$ for which $A=\tau(v)$,
    \item The operator $Q(\Delta-A)R(A)$ is nonzero on ${\mathbb C} V$,
    \item The trace of $Q(\Delta-A)R(A)$ is nonzero on ${\mathbb C} V$,
    \item The trace of $R(A)Q(\Delta-A)$ is nonzero on ${\mathbb C} V$.
\end{enumerate}
\end{theorem}

\begin{proof}
By \ref{prop:image}, the image of $R(A)$ is the span of the set $\{v\in V|A\subseteq \tau(v)\}$.
But by Proposition \ref{prop:kernel}, $Q(\Delta - A)$ is non-cacheable-zero only on $V(\Delta-A,+) = \{v\in V|\tau(v)\subseteq A\}$.
Hence the composition $Q(\Delta-A)R(A)$ is nonzero if and only if $\tau(v)=A$ for some $v$, and the formula for $Q(\Delta-A)$ given in the proof of Proposition \ref{prop:kernel} shows that the trace of $Q(\Delta-A)R(A)$ is precisely the number of such $v$.
Finally note that $\tr_V(R(A)Q(\Delta-A)) =\tr_V(Q(\Delta-A)R(A))$.
\end{proof}

\begin{corollary}
\label{cor:inv}
If $\Gamma_1=(V_1,m_1,\tau_1)$ and $\Gamma_2=(V_2,m_2,\tau_2)$ are such that ${\mathbb C} V_1$ and ${\mathbb C} V_2$ are isomorphic representations of $ W$, then $\Gamma_1$ and $\Gamma_2$ have the same $\tau$-invariant.
\end{corollary}

\begin{proof}
The statement follows from the fact that elements of the group algebra ${\mathbb C}[ W]$ act with the same trace on isomorphic representations.
\end{proof}

The $\tau$-invariants are therefore invariants of representations.
Corollaries \ref{cor:sign} and \ref{cor:inv} suggest the investigation of invariant objects simpler than the $\tau$-invariants. 


\begin{theorem}
\label{thm:Classical} 
Let $ W$ be a Weyl group of classical type.
Then the irreducible representations of $ W$ are in a one-to-one correspondence with their sign signatures.
\end{theorem}

The proof of the theorem is done case by case in Sections \ref{sec:typeA}, \ref{sec:typeB} and \ref{sec:typeD}.
More precisely, in Propositions \ref{prop:typeA}, \ref{prop:typeB} and \ref{prop:typeD} we show how to compute the multiplicity of the sign component of an irreducible representation of $ W$ when restricted to a parabolic subgroup.
In each case we show how to reverse the procedure, i.e., how to recover an irreducible representation of $W$ from its sign signature.

Theorem \ref{thm:Classical} fails for the exceptional Weyl groups.
However, if we allow extended sign signatures (with respect to additively closed root subsystems) instead of sign signatures, Theorem \ref{thm:Classical} has a direct generalization valid for arbitrary Weyl groups.
We carry out this generalization in Theorem \ref{thm:classicalAndExceptional}.

\section{Type $A$}
\label{sec:typeA}
This case seems to have been treated by A.\ Young.
A modern account is found in I.\ G.\ Macdonald \cite{Macdonald:SommeIrrepsWeylGroups}.
We give a proof only to make Theorem \ref{thm:Classical} self-contained. 

From the standpoint of recovering irreducible representations from sign signatures, the results in this section can be summarized as follows.
\begin{theorem}\label{thm:TypeARecoverIrrepsAlgorithm}
Let $V=V_{\lambda^*}$ be an irreducible representation of $A_{n-1}$, where $\lambda^*$ is the partition dual to the partition $\lambda$.
Then $\lambda$ is the lexicographically largest partition such that $P_\lambda$ belongs to the sign signature of $V$, where $P_\lambda$ is the parabolic subgroup corresponding to $\lambda$.
\end{theorem}
This is a direct corollary of Proposition \ref{prop:typeA:RecoverPartitionFromSignSignature} below. 
Details on the notation used in Theorem \ref{thm:TypeARecoverIrrepsAlgorithm} can be found in the rest of this section.

\subsection{Preliminaries and notation}

\subsubsection{Preliminaries: Young tableaux.}
In this subsection, we recall some facts  and notations about Young tableaux \cite{Fulton:YoungTableaux}.

We denote tuples of non-negative numbers by
\[
\lambda=(\lambda_1, \dots, \lambda_k).
\]
Given tuples of non-negative numbers $\lambda, \mu$, denote by $\lambda\oplus\mu$ the concatenation  of the two tuples, i.e., the tuple obtained by writing the entries of $\lambda$ first, followed by the entries of $\mu$. As usual, define the sum $\alpha+\beta$ as $(\lambda_1+\mu_1, \dots, \lambda_k+\mu_k)$, where $\lambda_i, \mu_i$ denote the $i^{th}$ entries or $0$ if the corresponding tuple has fewer than $i$ entries. 

Denote by $|\lambda| $ the sum of the entries of $\lambda$, i.e., 
\[
|\lambda|=\sum_{j=1}^{k}\lambda_j .
\]
We say that a tuple
\[
\lambda=[\lambda_1, \lambda_2,\dots, \lambda_k],
\] 
is a partition if $\lambda_1\geq \lambda_2\geq \dots \geq \lambda_k\geq 0$. 
We adopt the convention that the use of $[]$-brackets around the entries of $\lambda$ implies that $\lambda$ is a partition. 
We say that $\lambda $ is a partition of $n$ if $|\lambda|=n$.

A partition $\lambda$ is usually depicted as a ``shape'' obtained by stacking a row of $\lambda_k$ boxes in the bottom, $\lambda_{k-1}$ boxes on top it, and so on, as indicated in Table \ref{tableShapesSkewShapes}.

\begin{table}[h!] 
\caption{\label{tableShapesSkewShapes} Shapes and skew shapes.}
\begin{tabular}{|p{3.5cm}|p{3cm}|p{4cm}|}\hline
Shape $\lambda$& Shape $\mu$& Shape $\lambda/\mu$\\\hline
$\begin{array}[t]{|l|l|l|l|l|lll}
\cline{1-5}
&&& &&\lambda_1\text{ boxes}\\ \cline{1-5}
&&& \multicolumn{2 }{c}{} & \lambda_2 \text{ boxes}\\\cline{1-3}
\multicolumn{5 }{l}{\vdots} & \vdots\\\cline{1-2}
&&\multicolumn{3}{c}{} & \lambda_{l}\text{ boxes}\\\cline{1-2}
\multicolumn{5}{l}{\vdots}&\vdots\\\cline{1-1}
&\multicolumn{4}{c}{} & \lambda_{k}\text{ boxes}\\\cline{1-1}
\end{array}
$
&
$\begin{array}[t]{|l|l|l|llll}
\cline{1-3}
&&& \mu_1\text{ boxes}\\ \cline{1-3}
&& \multicolumn{1 }{c}{} & \mu_2 \text{ boxes}\\\cline{1-2}
\multicolumn{3 }{l}{\vdots} & \vdots\\\cline{1-1}
&\multicolumn{2}{c}{} & \mu_l\text{ boxes}\\\cline{1-1}
\end{array}
$
&
$\begin{array}[t]{|l|l|l|l|l|lll}
\cline{1-5}
\cellcolor{black}&\cellcolor{black}&\cellcolor{black}& &&\lambda_1-\mu_1\text{ boxes}\\ \cline{1-5}
\cellcolor{black}&\cellcolor{black}&& \multicolumn{2 }{c}{} & \lambda_2-\mu_2 \text{ boxes}\\\cline{1-3}
\multicolumn{5 }{l}{\vdots} & \vdots\\\cline{1-2}
\cellcolor{black}&&\multicolumn{3}{c}{} & \lambda_{l}-\mu_l\text{ boxes}\\\cline{1-2}
\multicolumn{5}{l}{\vdots}&\vdots\\\cline{1-1}
&\multicolumn{4}{c}{} & \lambda_{k}\text{ boxes}\\\cline{1-1}
\end{array}
$
\\
\hline
\end{tabular}
\end{table}

Let $\lambda^* $ be the partition obtained by transposing the boxes of the partition $\lambda$.
In other words, if $\lambda= [\lambda_1, \dots, \lambda_k]$, then 
\[
\lambda^*=[ \underbrace{k,\dots, k}_{\lambda_k \text{ times}}, \underbrace{k-1, \dots, k-1}_{\lambda_{k-1}- \lambda_{k} \text{ times}}, \dots, \underbrace{1, \dots, 1}_{\lambda_1-\lambda_{2}\text{ times}} ] .
\]

A labeling of boxes in the shape of $\lambda$ with the numbers $v_1, \dots, v_n$, $n=|\lambda|$ as indicated in Table \ref{tableNumbering} is called a \emph{numbering of shape $\lambda$}.
We say that a numbering has content $(p_1, \dots, p_k)$ with $\sum p_k=n$ if the labels consist of $p_1$ $1$'s, $p_2$ $2$'s, and so on.

A numbering of shape $\lambda$ is a \emph{semi-standard Young tableau} if it has the following property.
\begin{itemize}
\item The numbers appearing in each row are non-decreasing and the numbers appearing in each column are strictly increasing.
\end{itemize}

\begin{table}[h!]
\caption{\label{tableNumbering} $(v_1, \dots, v_{n})$-numbering of shape $\lambda$}

$\begin{array}[t]{|l|l|l|l|l|lll}
\cline{1-5}
v_1& \multicolumn{3}{c|}{ ~~~~~~~~~~~~ \dots ~~~~~~~~~~~~} &v_{\lambda_1} &&\lambda_1 \text{ boxes}\\ \cline{1-5}
\multicolumn{6 }{l}{\vdots} & \vdots\\\cline{1-3}
\multicolumn{2}{|c|}{~~~~~~~~~~~~ \dots ~~~~~~~~~~~~ }& v_{q} & \multicolumn{1}{c}{ ~~~~~~}&\multicolumn{2}{c}{} & \lambda_{n}\text{ boxes}\\\cline{1-3}
\end{array}
$
\end{table}

Let $\lambda$ be a partition with $k$ entries and $\mu$ be a partition with $l$ entries.
Suppose in addition $k\geq l$ and $\lambda_j\geq \mu_j$ for all allowed indices $j$.
Define the skew shape $\lambda/\mu $ as the $k$-tuple $(\lambda_1- \mu_1, \dots, \lambda_l-\mu_l, \lambda_{l+1}, \dots, \lambda_k)$.
Skew shapes are depicted as shown in Table \ref{tableShapesSkewShapes}. 

If the skew shape $\lambda/\mu$ is defined, we recall that a labeling of boxes in the skew shape of $\lambda/\mu$ with the numbers $v_1, \dots, v_b$, $b=|\lambda|-|\mu|$ as illustrated in Table \ref{tableSkewShapeNumbering}, is called a \emph{numbering of a skew shape $\lambda/\mu$}.
Just as with non-skew shapes, we say that a numbering of a skew shape $\lambda/\mu$ has content $(p_1, \dots, p_k)$ with $\sum p_k=|\lambda|-|\mu|$ if the labels consist of $p_1$ $1$'s, $p_2$ $2$'s, and so on.

In analogy with semi-standard Young tableaux, we say that a numbering of a skew shape is a \emph{skew-tableau} of shape $\lambda/\mu $ if it has the following property.

\begin{itemize}
\item The numbers in the same row are non-decreasing and all numbers in the same column are strictly increasing and all numbers appearing .
\end{itemize}

\begin{table}[h!]
\caption{\label{tableSkewShapeNumbering} $(v_1, \dots, v_b)$-numbering of skew shape $\lambda/\mu$.}
$\begin{array}[t]{|l|l|l|l|l|lll}
\cline{1-5}
\cellcolor{black}&\cellcolor{black}&\cellcolor{black}&v_{b-1} &v_{b} &\lambda_1 -\mu_1\text{ boxes}\\ \cline{1-5}
\cellcolor{black}&\cellcolor{black}&v_{b-2}& \multicolumn{2 }{c}{} & \lambda_2-\mu_2 \text{ boxes}\\\cline{1-3}
\multicolumn{5 }{l}{\vdots} & \vdots\\\cline{1-2}
\cellcolor{black}&v_{j}&\multicolumn{3}{c}{} & \lambda_{l}-\mu_l\text{ boxes}\\\cline{1-2}
\multicolumn{5}{l}{\vdots}&\vdots\\\cline{1-1}
v_1&\multicolumn{4}{c}{} & \lambda_{k}\text{ boxes}\\\cline{1-1}
\end{array}
$
\end{table}

We say that skew-tableau of shape $\lambda/\mu$ is a \emph{Littlewood-Richardson skew-tableau} if it has in addition the following property. 
\begin{itemize}
\item The $b$-tuple $(v_1, \dots, v_b)$ is a Yamanuchi word, i.e., for any index $s$, the sequence $v_b, v_{b-1},\dots, v_s$ contains at least as many $1$'s as it does $2$'s, at least as many $2's$ as it does $3$'s, and so on for all positive integers. 
\end{itemize}

\subsubsection{Preliminaries: branching laws and tableaux.} \label{sec:typeA:preliminaries:branchinglaws}
In this section, we set ${\mathfrak g} = \mathfrak{gl}_n$, i.e., the set of $n\times n$ matrices.
Then the Weyl group $W$ is the symmetric group $S_n$.
The parabolic subgroups of $ W$ are then given by 
\begin{equation}\label{eq:typeA:parabolicDef}
P_{(p_1, \dots, p_k)}=    S_{p_{1}} \times  \dots \times  S_{p_k}  .
\end{equation}
where $p_j\geq 1$, $\displaystyle \sum\limits_{j=1}^k p_j=n$, and each factor-group $S_{p_j} $ is embedded in $S_n$ as the permutations of the consecutive indices $\left\{a+1, a+2, \dots, a+p_j \right\}$, where $\displaystyle a=\sum\limits_{i=1}^{j-1} p_i $.

It is well-known that the irreducible representations of $ W=S_n$ are in a one-to-one correspondence to the partitions of $n$.
We denote by $V_{\lambda}$ the irreducible representation corresponding to the partition $\lambda=[p_1, \dots, p_k]$ and by $\signRep {S_n}$ the representation
\[
\signRep {S_n}=V_{[1, \dots, 1]}.
\]
As the notation suggests, $\signRep {S_n}$ is the sign (determinant) representation of $S_n$.
The objective of this section is to prove that an irreducible representation $V$ of $S_n$ is determined by the collection of those parabolic subgroups $P$ for which the restriction of $V$ down to $P$ contains a sign representation.
We will later need the following well-known observation, which we recall without proof.
\begin{lemma}\label{le:typeA:tensoringbySignrepstarsLambda}
\[
V_\lambda\otimes \signRep {S_n}\simeq V_{\lambda^*}
\]
\end{lemma}
Let 
\[
P=P_{(a,b)}= S_a\times S_b\subset S_{n},
\] 
$a+b=n$, $a,b\geq 1$ be a parabolic subgroup of $S_n$ realized as in \eqref{eq:typeA:parabolicDef}.
The Littlewood-Richardson coefficient $c_{\mu, \nu}^{\lambda}$ is defined as the multiplicity of the $P$-module $V_{\mu}\otimes V_{\nu}$ in the restriction of $ V_\lambda$ down to a $P$-module.
In other words, $c_{\mu, \nu}^\lambda$ is defined via the $P$-module isomorphism 
\begin{equation}\label{eqLittlewoodRichardsonCoeffDef}
{V_{\lambda}} _{|P}\simeq \bigoplus_{\substack{ |\mu|=a\\|\nu|=b}} c_{\mu, \nu}^{\lambda} V_{\mu}\otimes V_{\nu}  . 
\end{equation}

The main technical tool used in this section is the following well-know theorem.
\begin{theorem}[{\cite[{Theorem 4.9.4}]{Sagan:SymmetricGroupRepsAlgorithms}}] \label{thLittlewoodRichardsonCoeff}
The Littlewood-Richardson coefficients $c_{\mu,\nu}^\lambda$ defined by \eqref{eqLittlewoodRichardsonCoeffDef} are computed via
\begin{equation}\label{eqLittlewoodRichardsonCoeffComputation}
c^{\lambda}_{\mu, \nu}= \left\{ \begin{array}{ll} \begin{array}{@{}l}\text{number of Littlewood-Richardson tableaux } \\
\text{of shape }\lambda/\mu \text{ and content }\nu\end{array} &\text{if } \lambda/\mu \text{ is defined} \\
0&\text{otherwise.}\end{array}\right.
\end{equation}
\end{theorem}

This theorem implies the following.

\begin{corollary}[Pieri rule for $S_n$]\label{corPieriRuleSignRep}
 
Let $\pi=[\underbrace{1,\dots, 1}_{b \text{ times}}]$. 
Then $c^\lambda_{\mu, \pi}$ equals $0$ or $1$ with $1$ achieved if and only if the skew shape $\lambda/\mu$ has at most one box per row.
\end{corollary}
\begin{proof}
A skew-tableau $(v_1,\dots, v_b)$ of skew shape $\lambda/\mu$ has content $\pi=[1,\dots, 1]$ if the set $\{v_1,\dots, v_b\}$ consists of the numbers $1,2,3,\dots, b$.
The Yamanuchi word requirement for a Littlewood-Richardson skew tableau implies that a Littlewood-Richardson tableau of skew shape $\lambda/\mu$ and content $\pi=[1,\dots, 1]$ is given by $(v_1,\dots, v_b)=(1,2,\dots, b)$.
Finally the requirement that a Littlewood-Richardson tableau is in fact a tableau, i.e., has increasing entries in each row, forces the skew-tableau to have at most one box per row. 

Conversely, if a skew shape has $b$ boxes with one box per row, then a Littlewood-Richardson tableau of that shape with content $[1,\dots, 1]$ can be constructed by setting $(v_1, \dots, v_b)=(1,2, \dots, b)$, where the $v_i$'s are the tableau entries from Table \ref{tableSkewShapeNumbering}. 
\end{proof}
\subsection{Sign representation multiplicities when restricted to parabolic subgroups in type $A$}
Let $P=P_{\left(p_1, \dots, p_k\right)}=S_{p_1}\times\dots \times S_{p_k} \subset S_n $ be a parabolic subgroup embedded as described in \eqref{eq:typeA:parabolicDef}.
Let $\lambda=|n|$ and  $V_\lambda$ the corresponding irreducible representation of $S_n$.
For an arbitrary tuple of numbers $(p_1, \dots, p_k)$,  the Kostka number $K_{\lambda, (p_1,\dots, p_s)}$ is defined as
\[
K_{\lambda, (p_1,\dots, p_s)}=\begin{array}{l}\text{number of semi-standard tableaux} \\\text{of shape } \lambda \text{ and content } (p_1,\dots, p_s).\end{array}
\]

\begin{proposition}\label{prop:KostkaNumberInterpretation}
$K_{\lambda, (p_1,\dots, p_k)}$ is the dimension of the weight space of weight $(p_1, \dots, p_k)$ in the $GL(n)$-module of highest weight $\lambda$.
\end{proposition}

The following results are needed:

\begin{corollary}\label{cor:KostkaNumberInvarintUnderSn}
Permutation of the content entries $(p_1, \dots, p_n)$ does not change the Kostka number, i.e.,
\[
K_{\lambda, (p_1,\dots, p_k)} = K_{\lambda, (p_{\sigma(1)},\dots, p_{\sigma(k)})},
\] 
for an arbitrary permutation $\sigma $. 
\end{corollary}

\begin{lemma}\label{leBoxRowColumnIdentity}
Let $p$ be a positive number and let $\pi$ be the partition $\pi= [ \underbrace{1, \dots, 1}_{p\text{ times}}]$. Let $\mu$ be an arbitrary partition and $(p_1, \dots, p_s)$ be an arbitrary $s$-tuple of numbers. Then 
\begin{equation}\label{eqBoxRowColumnIdentity}
K_{\mu^*, (p_1,\dots, p_s, p)}= \sum_{ |\sigma|+p=|\mu|} c_{\sigma, \pi}^\mu  K_{\sigma^*, (p_1, \dots, p_s)} 
\end{equation}
\end{lemma}
\begin{proof}
A semi-standard tableau of content $\left(p_1, \dots, p_{s}\right)$ and shape $\sigma$ can be extended to a semi-standard tableau of content $\left(p_1,\dots, p_{s}, p\right)$ in as many ways as we can add $p$ boxes to $\sigma$ (while keeping the result a shape) without adding two boxes in the same column.
Indeed, this follows from the fact that all new boxes must be labeled with $s+1$ and therefore cannot be added to the same column, and there are no further restrictions as $s+1$ is the largest number appearing in the newly constructed tableau.
Therefore 
\[
\begin{array}{@{}r@{}c@{}l@{}l@{}|l@{}}
&&\displaystyle K_{\mu^*,\left(p_1, \dots, p_k\right)}=\\
&=& \displaystyle \sum_{\substack { \text{content}(\sigma^*)\\ =\left(p_1, \dots, p_{k-1}\right)}} K_{\sigma^*, \left(p_1, \dots, p_{k-1}\right)} \left(\begin{array}{l}
\text{\# ways to extend }\sigma^{*} \text{ to }\mu^* \\
\text{using at most one extra} \\
\text{box per column}
\end{array}
\right)\\
&=&\displaystyle \sum_{\substack { \text{content}(\sigma^*)\\ =\left(p_1, \dots, p_{k-1}\right)}} K_{\sigma^*, \left(p_1, \dots, p_{k-1}\right)} \left(\begin{array}{l}
\text{\# ways to extend }\sigma \text{ to }\mu \\
\text{using at most one extra} \\
\text{box per row}
\end{array}
\right) \\
&=&\displaystyle \sum_{\substack { \text{content}(\sigma^*)\\ =\left(p_1, \dots, p_{k-1}\right)}} K_{\sigma^*, \left(p_1, \dots, p_{k-1} \right)} c_{\sigma, \pi}^{\mu}. && \text{Cor. } \ref{corPieriRuleSignRep}\\
\end{array}
\]
\end{proof}

Given a parabolic subgroup $R\subset S_n $, let $\signRep {R} $ denote the sign representation of $R$. Recall that in the notation of this section, $P=P_{(p_1, \dots, p_k)}=S_{p_1}\times\dots\times S_{p_k} $.

\begin{proposition}
\label{prop:typeA} The multiplicity $\dim \Hom_P (\signRep {P}, V_\lambda)$ of the sign representation of $P$ in $V_\lambda$ is given by
\[
\dim \Hom_P (\signRep {P}, V_\lambda)=K_{\lambda^*, (p_1, \dots, p_k )},
\]
where $(p_1, \dots, p_k)$ is a tuple of positive integers.
\end{proposition}

\begin{proof}

The proof will be carried out by induction on $k$, i.e., on the number of factors of $P$.

In the base case, we have $k=1$, $P=S_n$ and $\left(p_1,\dots\right)=(n)$.
The Kostka number $K_{\lambda^*, (n)} $ now measures the number of semi-standard Young tableaux numberings of the shape $\lambda^* $ that only use the number $1$.
A semistandard Young tableaux has strictly increasing numbers in each column and therefore $ K_{\lambda^*, (n)}$ can be non-zero only if $\lambda^*$ consists of a single row.
Now the base case follows from the fact that the sign representation of $P=S_n$ has non-zero multiplicity only when the original irreducible module is the sign representation of $S_n$, i.e. $\lambda =[1,\dots, 1]$. 

Suppose we have shown our statement by induction for up to $k-1$ factors.
Set \[
Q=S_{p_1}\times \dots\times  S_{p_{k-1}}   .
\] 
The product $Q$ embeds with an embedding of the form \eqref{eq:typeA:parabolicDef} in $S_{n-p_k}$.
Furthermore $S_{n-p_k} \times S_{p_k}$ is again a parabolic subgroup of the form \eqref{eq:typeA:parabolicDef}, i.e., we have the chain of embeddings 
\[
Q\hookrightarrow S_{n-p_k} \hookrightarrow S_{n-p_k}\times S_{p_k}\hookrightarrow S_n .
\]

The rest of the proof is a computation that combines the induction hypothesis with the combinatorics in Lemma \ref{leBoxRowColumnIdentity}.

\medskip 
$
\begin{array}{@{}r@{}c@{}l@{}l@{\!\!\!\!\!\!\!\!\!}|l@{}}
\dim \Hom_P \left(\signRep {P}, V_\lambda\right)&= & \displaystyle  \dim\Hom_{P}\left(\signRep {P}, \bigoplus_{ \substack{|\sigma|=n-p_k\\ |\nu|=p_k} } c_{\sigma,\nu}^{\lambda} V_\sigma\otimes V_{\nu}\right)&& \begin{array}{l}\text{use}\\ \eqref{eqLittlewoodRichardsonCoeffDef}\end{array} \\
&=&\displaystyle \!\!\!\!\sum\limits_{ \substack{ |\sigma| =n-p_k\\ |\nu|=p_k}}  \!\!\!\!c_{\sigma,\nu}^{\lambda} \dim \Hom_{P} \left(\signRep{P}, V_\sigma\otimes V_\nu\right) \\
&=&\displaystyle \!\!\!\!\sum\limits_{ \substack{ |\sigma| =n-p_k\\ |\nu|=p_k}}  \!\!\!\!c_{\sigma,\nu}^{\lambda} \dim \Hom_{Q\times S_{p_k}} \left(\signRep{Q}\otimes \signRep{S_{p_k}},\right.\\ && \quad \quad \quad \quad \quad \quad \quad \quad \quad \quad \quad \left. V_\sigma\otimes V_\nu\right) \\~\\~\\
&=&\displaystyle \!\!\!\!\sum\limits_{ \substack{ |\sigma| =n-p_k\\ |\nu|=p_k}}  \!\!\!\!c_{\sigma,\nu}^{\lambda} \dim \Hom_{Q} \left(\signRep{Q} , V_\sigma\right) \\
&&\quad \quad \dim \Hom_{S_{p_k}} \left(\signRep{S_{p_k}}, V_\nu\right) \\~\\~\\
&=&\displaystyle \!\!\!\!\sum\limits_{|\sigma| =n-p_k} c_{\sigma,\nu}^{\lambda} K_{\sigma^*, \left(p_1, \dots, p_{k-1}\right)}  &&\text{ind. hyp.} \\
&=&\displaystyle K_{\lambda^*,\left(p_1, \dots, p_k\right)}, &&\begin{array}{l}\text{use}\\ \eqref{eqBoxRowColumnIdentity}\end{array}\\
\end{array}
$
\medskip

\noindent which completes our induction step and the proof of the theorem.
\end{proof}

\subsection{Recovering representations from sign multiplicities in type $A$}
\begin{theorem}
Let $V=V_{\lambda^*}$ be an irreducible representation of $A_{n-1}$.
Then $\lambda$ is the lexicographically largest partition such that $P_\lambda$ belongs to the sign signature of $V$.
\end{theorem}
\begin{proof}
This is Proposition~\ref{prop:typeA:RecoverPartitionFromSignSignature}.
\end{proof}

\medskip
The sign signature of a representation $V$ is the set of parabolic subgroups $P $ such that $ \dim \Hom_{P}(\signRep{ P}, V)>0$. 

\begin{proposition}\label{cor:typeA:differentIrrepsDifferentSignSignatures}
For $\lambda\neq \mu$ the $S_n$- representations $V_{\lambda} $ and $V_\mu$ have different sign signatures. 
\end{proposition}

\begin{proof}
Suppose the first different entry of $\lambda $ and $\mu$ is larger for $\lambda$. 
Then $K_{\lambda^*, \lambda^*} $ equals $1$ and $K_{\mu^*, \lambda^*}$ equals $0$.
Proposition \ref{prop:typeA} implies that parabolic subgroup given by the entries of $\lambda^* $ is in the sign signature of $V_\lambda $ but not in the sign signature of $V_\mu$. 
\end{proof}
Let $ > $ be the lexicographic order on partitions. 
Then $\lambda>\mu$ if $\lambda\neq \mu$ and the first different entry of $\lambda$ and $\mu$ is larger for $\lambda$. 
We compare partitions with different number of rows by adding zero entries to the partitions with smaller number of rows. Let 
\begin{equation*}
\label{eq:typeA:irrepOrder}
\lambda_1^{*}>\lambda_2^*> \dots >\lambda_N^* 
\end{equation*}
be the partitions of $n$, ordered with respect to the lexicographic order on their transposes where $N$ is the number of partitions of $n$. 

Set $P_i=P_{\lambda_i} $, and let $\Pi$ be the set of the $P_i$'s:
\[
\Pi=\{ P_{1}, \dots, P_{N}\}.
\]
Every parabolic subgroup $P$ is conjugate to one of the $P_i$'s and so the sign signature of a representation is determined unambiguously when restricted to $\Pi$.

Recall that $V_\lambda$ stands for an irreducible representation of $S_n$. 
The proof of Proposition \ref{cor:typeA:differentIrrepsDifferentSignSignatures} suggests the following.
\begin{proposition}\label{prop:typeA:RecoverPartitionFromSignSignature}~
\begin{enumerate}
\item $\lambda$ is the partition dual to the lexicographically maximal partition in $\signSignature (V_\lambda)$, in other words, 
\[\lambda^* = \max_{>} \{\mu^*|\mu\in \signSignature (V_\lambda)\}.\]
\item The sign signature  $\signSignature V_{\lambda_i}$ contains $P_i$ and does not contain $P_{1},\dots, P_{i-1}$.
\end{enumerate}
\end{proposition}

To recover $\lambda $ from $\signSignature (V_\lambda)$, we take the dual partitions $S=\{\mu^*|P_\mu\in\signSignature (V_\lambda)\}$, find the lexicographically maximal element of $S$, and take its dual. 

Let $\mathbb C[\Pi]$ be the formal vector space generated freely over $\mathbb C$ by the elements of $\Pi$. 
In other words, an element of $\mathbb C[\Pi]$ is a formal sum of the form $\displaystyle \sum_{\sigma} a_\sigma P_\sigma $, where $a_\sigma$ are arbitrary complex numbers and $\sigma$ runs over the partitions of $n$.
\begin{definition}\label{def:typeA:SignMult}
Given a representation $V$, define $\signMult (V)\in \mathbb C[\Pi] $ as the formal sum 
\[
\signMult(V)=\sum_{|\sigma|=n} \dim\Hom_{P_\sigma} \left(\signRep{P_{\sigma}}, V \right) P_\sigma\quad ,
\]
where $\sigma$ runs over the partitions of $n$.
\end{definition} 
In the terminology of Definition \ref{def:typeA:SignMult}, Proposition \ref{prop:typeA:RecoverPartitionFromSignSignature} can be restated as follows.  Order the representations $V_{\lambda}$ in the order $V_{\lambda_1}, V_{\lambda_2},\dots$ given by $\lambda_1^*< \lambda_2^*$. Then the matrix formed by the coordinates $a_{\nu}$ of the irreducible representations $V_{\lambda}$, form a triangular matrix with non-zero diagonal entries.
Linear algebra implies the following.

\begin{corollary}\label{cor:typeA:recoverArbitraryRepFromSignSignature}~
\begin{enumerate}
\item The vectors $\signMult(V_\lambda) $ obtained by letting $V_\lambda$ run over the set of irreducible representations of $S_n$ are linearly independent.
\item $\signMult(V)=\signMult(U)$ if and only if $V\simeq U$.
\item \label{cor:typeA:MatrixSignMult} The matrix $ \left(\dim\Hom_{P_\sigma}\left( \signRep{P_\sigma}, V_{\lambda}\right) \right)_{\sigma, \lambda }$ is square and has non-zero determinant.
\end{enumerate}
\end{corollary}

Hence, up to isomorphism, an arbitrary representation $V$ can be identified  with the tuple $ \left(\dim \Hom_{P_{\sigma}} \left( \signRep{P_{\sigma}}, V\right), \dots\right)_{\sigma} $, where $\sigma$ runs over the partitions of $n$. 
To recover the multiplicities of each irreducible constituent of $V$ from $\signMult V$, one multiplies the vector-column $ \left(\dim \Hom_{P_{\sigma}} \left( \signRep{P_{\sigma}}, V\right), \dots\right)_{\sigma} $ on the left by the inverse of the matrix given in Corollary \ref{cor:typeA:recoverArbitraryRepFromSignSignature}.\ref{cor:typeA:MatrixSignMult}).
\section{Types $B$ and $C$}\label{sec:typeB}
In type $A$, we characterized an arbitrary representation $V$ via the multiplicities of the sign representation of the parabolic subgroups. 
For an irreducible representations in type $B$, we give a similar characterization in Corollary \ref{cor:typeB:irreps}. 
This in turn implies an algorithm for recovering irreducible representations from sign signatures given in  Corollary \ref{cor:typeB:recoveringIrrepAlgorithm}. 
A short rephrasing of this corollary is given in the following.

\begin{theorem}\label{thm:TypeBRecoverIrrepsAlgorithm}
Let $V=V_{\lambda^*,\mu^*}$ be an irreducible representation of $B_n$.
Then $\lambda$ and $\mu$ can be computed from the sign signature of $V$ as follows:
\begin{enumerate}
\item Let $\alpha=(\alpha_1,\dots,\alpha_k)$ be the lexicographically largest partition such that $P_{\alpha,\emptyset}$ belongs to the sign signature of $V$.
\item For each $i$ between $1$ and $k$, let $\beta^{(i)}=\left(\beta^{(i)}_1,\dots,\beta^{(i)}_{k_i}\right)$ be the lexicographically largest partition such that $P_{\left(\beta^{(i)}_1,\dots,\beta^{(i)}_{i-1},\beta^{(i)}_{i+1},\dots,\beta^{(i)}_{k_i}\right),\left(\beta^{(i)}_i\right)}$ lies in the sign signature of $V$.
\item Then $\lambda=\left(\beta^{(1)}_1,\dots,\beta^{(k)}_k\right)$ and $\mu=\alpha-\lambda$.
\end{enumerate}
\end{theorem}
Details on the notation used in Theorem \ref{thm:TypeARecoverIrrepsAlgorithm} can be found in the rest of this section.

This characterization of irreducibles cannot be extended to all representations in type $B$ where the number of conjugacy types of parabolic subgroups is smaller than that of irreducible representations.

However, if we replace the set of parabolic subgroups with the set of generalized parabolic subgroups - i.e, sums of type $A$ and possibly more than one type $B$ factors - we get a characterization of all representations in type $B$. 
Indeed, in Corollary \ref{cor:typeB:recoverArbitraryRepFromGeneralizedSignSignature} we show that the sign multiplicities over the generalized parabolic subgroups determine up to isomorphism the representations in type $B$. 
Of course, this result is not surprising in view of the fact that generalized parabolics give rise to Macdonald polynomials and therefore can be used to construct all irreducible representations of $B_n$.

We construct general branching rules from $B_n$ over subgroups of the forms $S_n$ and  $B_{i}\times B_{n-i} $. 
Then we extract branching rules of  irreducible representations of $B_n$ down to generalized parabolic subgroups, and  
we produce formulas for the multiplicities of the sign representations of such subgroups.
Finally, we use these formulas to characterize representation of $B_n$ via the sign multiplicities. 
\subsection{Preliminaries and notation}
Throughout this section we use the tableaux notation and terminology fixed in Section \ref{sec:typeA}. 

The Weyl groups of type $B$ and $C$ coincide and are denoted by $B_n$, the semidirect product,
\begin{equation}\label{eqBnDefinition}
B_n = S_n \ltimes \mathbb Z_2^n   .
\end{equation}
The structure of the semidirect product can be described by the natural representation of $B_n$. 
This is obtained by realizing $S_n$ via its natural representation (the $n\times n$ permutation matrices) and $\mathbb Z_2^n$ as the matrices $\{ \diag (\pm 1, \dots, \pm 1) \}$. 
Also we can think of $B_n $ as the automorphism group of the following graph.

\begin{equation} \label{eq:graphTypeBC}
\xygraph {
!{<0cm, 0cm>;<1cm, 0cm>:<0cm, 1cm>: :}
!{ ( 0, 0) } *+{\substack{a_1 \\ \bullet}}="a1"
!{ ( 0,-1) } *+{\substack{a_2 \\ \bullet}}="a2"
!{ ( 0,-3) } *+{\dots}
!{ ( 0,-5) } *+{\substack{a_{n-1} \\ \bullet}}="a3"
!{ ( 0,-6) } *+{\substack{a_{n} \\ \bullet}}="a4"
!{ (-1, 0) } *+{\substack{\bullet \\ b_1}}="b1"
!{ (-1,-1) } *+{\substack{\bullet \\ b_2}}="b2"
!{ (-1,-3) } *+{\dots}
!{ (-1,-5) } *+{\substack{\bullet \\ b_{n-1}}}="b3"
!{ (-1,-6) } *+{\substack{\bullet \\ b_{n}}}="b4"
"a1"-"b1" "a2"-"b2"  "a3"-"b3" "a4"-"b4"
}
\end{equation}
In this model for $B_n$, the group $\mathbb Z_2^n$ is the group consisting of elements that preserve edges but swap their endpoints.

The parabolic subgroups of $B_n$ are of two types:

\begin{equation} \label{eqParabolicsTypeAInBn} 
P_{(p_1,\dots, p_k), \emptyset} \simeq  S_{p_1} \times \dots \times  S_{p_k} 
\end{equation}
and
\begin{equation}\label{eqParabolicsMixedTypeInBn} 
P_{(p_1,\dots,p_{k}), (p_{k+1})} =S_{p_1} \times \dots \times  S_{p_{k}} \times  B_{p_{k+1}} 
\end{equation}
where $\sum\limits_{j=1}^{k+1} p_j = n$. 
We extend our notation to 
\begin{equation}\label{eqParabolicsGeneralizedBn} 
P_{(p_1,\dots,p_{k}), (p_{k+1},\dots, p_{l})} =S_{p_1} \times \dots \times  S_{p_{k}} \times  B_{p_{k+1}}\times \dots \times B_{p_l}.
\end{equation}
Here, each factor-group $S_{p_j} $ is embedded in $S_n$ as the permutation of the consecutive indices $ \{a+1, a+2, \dots, a+p_j\}$, where $a=\sum\limits_{i=1}^{j-1} p_i$, and similarly $B_{p_j}$ is embedded as the subgroup generated by $S_{p_{j}}$ and the diagonal matrices of the form $\diag( \underbrace{1,\dots, 1}_{a\text{ copies}}, \underbrace{\pm 1, \dots,\pm 1}_{p_j \text{ copies}}, 1,\dots, 1) $. 
We call subgroups of the form \eqref{eqParabolicsGeneralizedBn} \emph{generalized parabolic subgroups}. 
Generalized parabolic subgroups with more than one type $B$ factor are not parabolic - indeed, they are the reflection groups generated by non-additively closed root subsystems (see Definition \ref{def:additivelyClosedRootSubsystem} below). 
Consequently generalized parabolic subgroups do not in general arise as the Weyl groups of the Levi parts of parabolic Lie subalgebras. 

\subsubsection{Preliminaries: the little groups method of Wigner and Mackey}\label{sec:LittleGroups}
The irreducible representations of a semidirect sum of a finite group with an abelian normal group are constructed from the irreducible representations of the starting finite group via the ``little groups'' method of Wigner and Mackey. 
For the reader's convenience, we recall the method (with minor modifications in notation) from \cite[\S 8.2]{Serre:FiniteGroupsRepresentations}.

Let $W$ be a group of the form $ G\ltimes H$ with $H$ abelian and 

\[
X= \Hom(H, \mathbb C^*)  .
\]
the group of characters of $H$.
Then $G$ acts $X$ via
\[
g(\chi)(h)= \chi(g^{-1}h g)\quad \text{for } h\in H, \chi \in X, g\in G .
\]
Let $\{\chi_i\}_{\chi_i\in X/G}$ be a set of representatives of the orbits of the characters under the action of $G$. 
For each $\chi_i$ denote by $G_i$ the stabilizer of $\chi_i$ in $X $ and let $W_i$ be the subgroup generated by $G_i$ and $H$. 
Thus
\begin{equation}\label{eq:DefAbelianCharacterStabilizer}
\begin{array}{rcl}
G_i &=& \{g| g(\chi_i)=\chi_i \} \\
W_i &=& G_i\ltimes H.
\end{array}
\end{equation}
We extend the function $\chi_i$ to $W_i$ by setting 
\[
\chi_i(h g)=\chi_i(h) \quad \text{ for } h \in H, g\in G_i.
\]

If $A$ is a subgroup of a group $B$, and a representation of $V$ of $A$, then $\Ind_{A}^B V $ denotes the representation of $B$ induced from $V$.

\begin{theorem} (\cite[Proposition 25, {\S}8.2]{Serre:FiniteGroupsRepresentations}, Wigner and Mackey's little groups method) \label{theoremLittleGroupsMethod}
\begin{enumerate}
\item \label{theoremLittleGroupsMethodPart1} Let $V$ be an irreducible representation of $G$. Then 
\[
\Ind_{W_i}^W \left(V\otimes \chi_i \right)
\]
is an irreducible representation of $W$.

\item The irreducible representations given in \ref{theoremLittleGroupsMethodPart1}) are pairwise distinct and exhaust the irreducible representations of $W$.
\end{enumerate}

\end{theorem}

\subsubsection{Preliminaries: the irreducible representations of $B_n$}
The irreducible representations of $B_n$ are well known (see for example \cite[page 375]{Carter:FiniteGroupsLieType}).
However, it is in our interest to reinterpret these irreducible representations via the little groups method of Wigner and Mackey (Section \ref{sec:LittleGroups}). 
Assume $G= S_n$ and $H=\mathbb Z_2^n$ and let $\epsilon_{1}$  and $\epsilon_{-1}$ stand respectively for the trivial and the non-trivial character of $\mathbb Z_2$. 
Then the characters $X$ of $\mathbb Z_2^n$ are of the form 
\[ 
\epsilon_{\pm1} \otimes  \dots  \otimes  \epsilon_{\pm1}.
\]
A set $X/S_n$ of orbit representatives of $X$ under the action of $S_n$ is then given by 
\begin{equation}\label{eqZ2characterTypeB}
\{\chi_i|\chi_i= \underbrace{\epsilon_{1}\otimes \dots \otimes \epsilon_1}_{i \text{ copies}}\otimes \underbrace{\epsilon_{-1}\otimes \dots \otimes \epsilon_{-1}}_{n-i \text{ copies}}, 0\leq i\leq n  \}.
\end{equation}
The stabilizer $G_i$ of $\chi_i$ under the action of $S_n$ is the group $S_{i}\times S_{n-i}$. 
Under the natural matrix realization of $W=B_n$, $S_n$ is the permutation group of the indices $\{1,\dots, n\}$ and so the factors of $S_{i}\times S_{n-i}$ are realized  as the permutations of the indices $\{1,\dots, i\}$ and $\{i+1, \dots, n\}$ giving the the following chain of embeddings
\[
G_i=S_i\times S_{n-i}\hookrightarrow S_n\hookrightarrow S_n\ltimes \mathbb Z_2^n =B_n.
\]

Also we have:
\begin{equation}\label{eqBitimesBn-1def}
\begin{array}{rcl}
\displaystyle W_i&=&\displaystyle G_i\ltimes \mathbb Z_2^n \\
&=&\displaystyle \left(S_i\times S_{n-i}\right)\ltimes \mathbb Z_2^n \\
&\simeq&\displaystyle  B_{i}\times B_{n-i}.
\end{array}
\end{equation}
The set of irreducible representations of $S_i\times S_{n-i}$ consists (up to isomorphism) of the representations $V_\lambda \otimes V_\mu$ with  $ |\lambda|=i$ and $|\mu|=n-i$. 
Here $V_\lambda$ and $V_\mu$ are the type $A$ irreducible representations fixed in Section \ref{sec:typeA:preliminaries:branchinglaws}. 
Theorem \ref{theoremLittleGroupsMethod} now yields the following well-known characterization of the irreducible representations of $B_n$.
\begin{corollary}\label{cor:typeB:irreps}
The set of irreducible representations of $W=B_n$ consists (up to isomorphism) of the (pairwise non-isomorphic) representations
\begin{equation}\label{eqIrrepsTypeB}
V_{\lambda, \mu} =\Ind_{B_{i}\times B_{n-i} }^{B_n}\left( V_\lambda\otimes V_\mu \otimes \chi_i\right),
\end{equation}
where $i=|\lambda|$ and $|\lambda|+|\mu|=n$.
\end{corollary}
In the above notation, the sign (determinant) representation in type $B_n$ is quickly seen to be
\begin{equation}\label{eqSignRepTypeB}
\signRep{B_n}=V_{\emptyset, \pi}
\end{equation}
where as usual $\pi=[\underbrace{1,\dots, 1}_{n}]$ is the partition that corresponds to the sign representation of $S_{n}$, and $\emptyset$ denotes the empty partition.

When dealing with the type $D$ case, we will need the following.
\begin{corollary}
\[V_{\lambda,\mu}\otimes \signRep {B_n}\simeq V_{\mu^*, \lambda^*} \]
\end{corollary}
\begin{proof}
In the notation of Corollary \ref{cor:typeB:irreps}, let $|\lambda|=i$ and $|\mu |=n-i$. 
By \eqref{eqBitimesBn-1def} $\signRep {B_n}$ restricts to the $B_i\times B_{n-i}\simeq  S_i\times S_{n-i}\ltimes \mathbb Z_2^n$-representation, $V$, given by 
\[
V=\signRep S_i\otimes \signRep S_{n-i} \otimes \chi_n .
\] 
This implies the $ S_i\times S_{n-i}\ltimes \mathbb Z_2^n$-module isomorphism 
\[
\begin{array}{@{}r@{}c@{}ll|l}
\left(V_{\lambda}\otimes V_{\mu}\otimes \chi_i \right) \otimes V
&=&\left(V_\lambda\otimes \signRep {S_i}\right) \otimes \left(V_\lambda\otimes \signRep {S_{n-i}}\right)\otimes \left(\chi_i\otimes \chi_n \right) \\
&=&V_{\lambda^*}\otimes V_{\mu^*} \otimes \bar \chi_{n-i}, &&\text{Lemma } \ref{le:typeA:tensoringbySignrepstarsLambda}
\end{array}
\]
where by $\bar \chi_{n-i}$ we denote the $\mathbb  Z_2^n$-character that acts by $-1$ on the first $i$ direct multiplicands of $\mathbb Z_2^n$ and by $1$ on the remaining.

The representation above can also be regarded as the $B_{n-i} \times B_{i}\simeq S_{n-i}\times S_{i}\ltimes \mathbb Z_2^n $-module $V_{\mu^*}\otimes V_{\lambda^*} \otimes \chi_{n-i}$. 
By Corollary \ref{cor:typeB:irreps} the module 
\[
V_{\mu^*, \lambda^*}\simeq  \Ind_{B_{n-i}\times B_i} ^{B_n} \left( V_{\mu^*}\otimes V_{\lambda^*}\otimes \chi_{n-i}\right) 
\] is irreducible and, by the preceding considerations, is isomorphic to
\[
\Ind_{B_i\times B_{n-i}}^{B_n}\left(  \left(V_\lambda \otimes V_{\mu}\otimes \chi_i\right)\otimes V \right).
\]
On the other hand the latter module is contained in $V_{\lambda,\mu}\otimes \signRep {B_n} $ and consequently $V_{\lambda,\mu}\otimes \signRep {B_n} $ contains  a sub-module isomorphic to $V_{\mu^*, \lambda^*} $; since $V_{\lambda,\mu}\otimes \signRep {B_n} $ is irreducible (because $\signRep {B_n}$ is one-dimensional) this can hold only if $V_{\lambda,\mu}\otimes \signRep {B_n} \simeq V_{\mu^*, \lambda^*}$.
\end{proof}

\subsubsection{Preliminaries: restrictions of induced representations}\label{secRestrictionsInducedReps}

Let $P\backslash W/A$ of $W$ stands for the double coset of $W$, i.e., 
\[
P\backslash W/A =\{PgA| g\in G  \}.
\]
Let $S$ be a set of representatives for $P\backslash W /A$. For $s\in S$, set 
\begin{equation}\label{eqAs-DoubleCosetGroup}
A_s=sAs^{-1}\cap P\quad .
\end{equation}
An $A$-module $V$ can be equipped with a structure of an $A_s$-module by setting 
\begin{equation}\label{eqVs-DoubleCosetRep}
a\cdot v = s^{-1} a s(v) \quad \text{ for } v\in V, a\in A_s, s\in S.
\end{equation}
We denote the so obtained $A_s$-module by $V_s$. 
\begin{proposition}\cite[Proposition 22, \S 7.3]{Serre:FiniteGroupsRepresentations} \label{propRestrictionInducedModules} 
Let $V$ be an $A$-module. 
Then we have the $P$-module isomorphism
\[\left(\Ind_{A}^{W} V\right)_{|P} = \bigoplus_{s\in S} \Ind_{A_s}^P ( V_s) .
\]
\end{proposition}
The following proposition is known as Frobenius reciprocity.
\begin{proposition}(\cite[proof of Theorem 13, \S 7.2]{Serre:FiniteGroupsRepresentations})
Let $B<P$ and let $V$ be a $B$-module and $U$ a $P$-module. 
Then
\begin{equation}\label{eqFrobeniusReciprocity}
\begin{array}{rcl}
\displaystyle \dim (\Hom_P(  U, \Ind_{B}^P V))&=& \displaystyle  \dim (\Hom_P( \Ind_{B}^P V, U)) \\
&=& \dim (\Hom_B (V,U_{|B} )). 
\end{array}
\end{equation}
\end{proposition}

\subsubsection{Preliminaries: branching laws in type $B$}

Let $B_i\times B_{n-i}$ be the subgroup of $B_n$ as given in \eqref{eqBitimesBn-1def}.
Let $V_{\lambda, \mu}$ be the irreducible module of $B_n$ defined in \eqref{eqIrrepsTypeB},  $V_{\lambda}$ stands for irreducible module of $S_n$, 
and $c_{\alpha, \beta}^\gamma$ (defined in \eqref{eqLittlewoodRichardsonCoeffDef}) the Littlewood-Richardson coefficients (in type $A$). 
The branching of the irreducible representations of $B_n$ over $B_k\times B_{n-k}$ and $S_n$ are known, see for example \cite[\S 2.C]{Stembridge:GuideToWeylGroupRepsEmphasisOnBranching}. 
Nevertheless, we provide more details.
\begin{proposition}
\label{prop:typeB:branchingOverLargeSubgroups}
\begin{eqnarray}
\displaystyle {V_{\lambda, \mu}}_{|B_k\times B_{n-k}} &\simeq& \displaystyle  \bigoplus_{\begin{array}{@{}r@{}c@{}l} |\nu| + |\sigma| &=&k\\ |\xi| +|\zeta|&=&n-k\end{array}} c_{\nu,\xi }^\lambda c_{\sigma,\zeta}^{\mu} V_{\nu, \sigma}\otimes V_{\xi, \zeta} \label{eqBranchingB_nToB_i+B_(n-i)}\\
{V_{\lambda, \mu}}_{|S_n}&\simeq &\bigoplus_{|\lambda|+|\mu|=|\nu|} c_{\lambda, \mu}^\nu V_\nu\label{eqBranchingB_nToS_n},
\end{eqnarray}
where $\lambda, \nu, \xi, \sigma, \zeta $ denote partitions. 
\end{proposition}
\begin{proof}
Both formulas will be proved by applying Proposition \ref{propRestrictionInducedModules} to the modules $V_{\lambda, \mu}$ given in \eqref{eqIrrepsTypeB}. 

Throughout this proof we will use the notation of Section \ref{secRestrictionsInducedReps}.

Set $i=|\lambda|$ and $A= W_i = B_{i}\times B_{n-i} $. 
From the graph realization of $B_n$ given in \ref{eq:graphTypeBC} we intepret $A$ as the direct product of the automorphism groups of the two subgraphs indicated in \ref{eq:GraphB_iB_n-i}.
\begin{equation}  \label{eq:GraphB_iB_n-i}
\xygraph {
!{<0cm, 0cm>;<1cm, 0cm>:<0cm, 1cm>: :}
!{ ( 0, 0) } *+{\substack{a_1 \\ \bullet}}="a1"
!{ ( 0,-1) } *+{\dots}
!{ ( 0,-2) } *+{\substack{a_i \\ \bullet}}="a2"
!{ ( 0,-3) } *+{\substack{a_{i+1} \\ \bullet}}="a3"
!{ ( 0,-4) } *+{\dots}
!{ ( 0,-5) } *+{\substack{a_{n} \\ \bullet}}="a4"
!{ (-1, 0) } *+{\substack{\bullet \\ b_1}}="b1"
!{ (-1,-1) } *+{\dots}
!{ (-1,-2) } *+{\substack{\bullet \\ b_i}}="b2"
!{ (-1,-3) } *+{\substack{\bullet \\ b_{i+1}}}="b3"
!{ (-1,-4) } *+{\dots}
!{ (-1,-5) } *+{\substack{\bullet \\ b_{n}}}="b4"
!{(-0.5, -1)} *=(1.8,-2.8){}*\frm{.}
!{(-0.5, -4)} *=(1.8,-2.8){}*\frm{.}
!{(-0.5, 1)}*+{B_i}
!{(-0.5, -6)}*+{B_{n-i}}
"a1"-"b1" "a2"-"b2"  "a3"-"b3" "a4"-"b4"
}
\end{equation}

\noindent\emph{Proof of \eqref{eqBranchingB_nToS_n}}. 
Set $P=S_n$. 
Since $B_n = S_n\ltimes \mathbb Z_2^n = P\ltimes \mathbb Z_2^n$ and $\mathbb Z_2^n\subset A$ it follows that there is only one double coset $B_n= P\backslash B_n /A$. 
Therefore we can choose the set of double coset representatives to be $ S=\{\id\}$. 
Then \eqref{eqAs-DoubleCosetGroup} becomes 
\[
A_{\id}= \id A \id^{-1}\cap P= \left( B_{i} \times B_{n-i}\right)\cap S_n = S_i\times S_{n-i}. 
\]
The rest of our statement for formula \eqref{eqBranchingB_nToS_n} follows from Proposition \ref{propRestrictionInducedModules}:
\[
\begin{array}{@{}r@{}c@{~}l@{}l@{}|l}
\displaystyle {V_{\lambda, \mu}}_{|S_n}&=&\displaystyle \Ind_{A}^W \left(V_\lambda\otimes V_\mu\otimes \chi_i\right)_{|S_n}  \\
&= &\displaystyle \Ind_{A_{\id}}^{S_n}\left( \left(V_\lambda\otimes V_\mu\otimes \chi_i \right)_{\id} \right) &&\text{Proposition \ref{propRestrictionInducedModules}}\\
&=&\displaystyle \Ind_{S_i\times S_{n-i}}^{S_n}\left( V_\lambda\otimes V_\mu\right) && \text{see }\eqref{eqVs-DoubleCosetRep}\\
&=&\displaystyle \bigoplus_{|\nu|= |\lambda|+|\mu|}\left( \dim \Hom_{S_n} (V_\nu, \Ind_{S_i\times S_{n-i}}^{S_n} \left(V_\lambda\otimes V_\mu \right)) \right) V_\nu\\
&=&\displaystyle  \bigoplus_{|\nu|=|\lambda|+|\mu| }\dim \Hom_{S_{i}\times S_{n-i}}(V_\lambda\otimes V_\mu, V_\nu) V_\nu&& 
\begin{array}{@{}l}
\text{Frobenius} \\
\text{reciprocity}\\
\eqref{eqFrobeniusReciprocity}
\end{array}\\
&=& \displaystyle  \bigoplus_{|\nu|=|\lambda|+|\mu| }c_{\lambda, \mu}^{\nu} V_\nu, &&
\begin{array}{@{}l}
\text{Littlewood-}\\
\text{Richardson}\\
\text{coeff. } \eqref{eqLittlewoodRichardsonCoeffDef}
\end{array}
\end{array}
\]
as desired. 

\noindent\emph{Proof of  \eqref{eqBranchingB_nToB_i+B_(n-i)}.} Set $P=B_{k}\times B_{n-k}$. 

Let $s'''\in B_n$ be an arbitrary element. 
We will now choose a set of representatives $S$ for the double coset $P\backslash B_n/A$ by modifying $s'''$ via the actions of $P$ and $A$ to a form unique to each double coset. 
Since $\mathbb Z_2^n\subset A$, using multiplication on the right via elements of $A$ we can transform $s'''$ to an element $s''$ in the same double coset that maps the $a_j$-vertices (and consequently also the $b_j$-vertices) onto themselves. 
Among the vertices $a_1, \dots, a_i$, let $l$ of them be mapped (non-surjectively) into the vertices $a_1, \dots, a_k $ by the action $s''$. 
It follows that $k-l$ of the vertices $a_{i+1}, \dots, a_{n}$ are mapped into the vertices $a_{1},\dots, a_k$ by the action of $s''$. 
By multiplying on the right by elements of $S_{i}\times S_{n-i}\subset A$ we can transform $s''$ to an element $s'$ in the same double coset with the following property: $s'$ maps the vertices $a_1, \dots, a_l$ into the vertices $a_{1},\dots, a_{k}$ and the vertices $a_{i+1}, \dots, a_{i+k-l}$ into the vertices $a_{k+1}, \dots, a_{n}$. 
Finally, by multiplying on the left by elements of $P=S_{k}\times S_{n-k}$ we can transform $s'$ to the element $s_l$ that lies in the same double coset and is defined via \eqref{eqCosetRepP=B_kB_(n-k)}.

\begin{equation}  \label{eqCosetRepP=B_kB_(n-k)}
\xygraph {
!{<0cm, 0cm>;<0.85cm, 0cm>:<0cm, 0.85cm>: :}
!{ ( 0, 0 ) } *+{\color{red} \substack{a_1 \\ \bullet}}="a1"
!{ ( 0,-1.5 ) } *+{\color{red} \dots}
!{ ( 0,-3 ) } *+{\color{red} \substack{a_{l} \\ \bullet}}="al"
!{ ( 0,-4 ) } *+{\color{red} \substack{a_{l+1} \\ \bullet}}="al+1"
!{ ( 0,-6 ) } *+{\color{red} \dots}
!{ ( 0,-8 ) } *+{\color{red} \substack{a_{i} \\ \bullet}}="ai"
!{ ( 0,-9 ) } *+{\substack{a_{i+1} \\ \bullet}}="ai+1"
!{ ( 0,-10) } *+{\dots}
!{ ( 0,-11) } *+{\substack{\!\!\!\!\!\!\!a_{i+k-l} \\ \bullet}}="ai+k-l"
!{ ( 0,-12) } *+{\substack{~~~~~~~a_{i+k-l+1} \\ \bullet}}="ai+k-l+1"
!{ ( 0,-13) } *+{\dots}
!{ ( 0,-14) } *+{\substack{a_{n} \\ \bullet}}="an"
!{ (-2, 0 ) } *+{\color{red} \substack{ \bullet\\ a_1}}="A1"
!{ (-2,-1.5)} *+{\dots}
!{ (-2,-3 ) } *+{\color{red} \substack{ \bullet \\ a_{l}}}="Al"
!{ (-2,-4 ) } *+{\substack{ \bullet \\ a_{l+1}}}="Al+1"
!{ (-2,-5 ) } *+{\dots}
!{ (-2,-6 ) } *+{\substack{ \bullet \\ a_{k}}}="Ak"
!{ (-2,-7 ) } *+{\color{red} \substack{ \bullet \\ a_{k+1}}}="Ak+1"
!{ (-2,-9 ) } *+{\dots}
!{ (-2,-11 )} *+{\color{red} \substack{\bullet\\ \!\!\!\!\!\!\!a_{k+i-l} }}="Ak+i-l"
!{ (-2,-12 )} *+{\substack{\bullet\\ ~~~~~~~a_{i+k-l+1} }}="Ai+k-l+1"
!{ (-2,-13) } *+{\dots}
!{ (-2,-14) } *+{\substack{\bullet\\ a_{n}}}="An"
!{(0, -4)} *=(0.6,-8.5){}*\frm{.}
!{(0, -11.5)} *=(0.6,-6){}*\frm{.}
!{(-2, -3)} *=(0.6,-6.5){}*\frm{.}
!{(-2, -10.5)} *=(0.6,-8){}*\frm{.}
!{(-6, -6)} *+[c]{
\begin{array}{rcl}
s_l(a_1)&=&a_1\\
&\vdots \\
s_l(a_l)&=& a_l\\\hline
s_l(a_{l+1})&=&a_{k+1}\\
&\vdots\\
s_l(a_{i})&=&a_{k+i-l}\\\hline
s_l(a_{i+1})&=&a_{l+1}\\
&\vdots\\
s_l(a_{i+k-l})&=&a_{k}\\\hline
s_l(a_{i+k-l+1})&=&a_{k+i+1}\\
&\vdots \\
s_l(a_n)&=&a_n
\end{array}
}
!{(-6, -11)} *+[r]{
\begin{array}{rcl}
0&\leq & l\\
i-n+k&\leq& l\\
l&\leq&i \\ 
l&\leq& k\\
\end{array}
}
"a1"-@{>}"A1" "al"-@{>}"Al" "al+1"-@{>}"Ak+1" "ai"-@{>}"Ak+i-l" "ai+1"-@{>}"Al+1" "ai+k-l"-@{>}"Ak" "ai+k-l+1"-@{>}"Ai+k-l+1" "an"-@{>}"An"
}
\end{equation}
In the formulas \eqref{eqCosetRepP=B_kB_(n-k)} for $s_l$, the ``corner'' cases such as $l=0$ are to be interpreted in the natural way. 
For different numbers $l$, the elements $s_l$ lie in different cosets, and therefore a set of double coset representatives is given by
\[
S=\{s_l| \max(0, i+k-n) \leq l \leq \min(i,k) \} .
\]
Recall that $A_{s_l}=s_l A s_{l}^{-1} \cap P$ and let 

\begin{equation}\label{eqBslDef}
B_{s_l}= s_{l}^{-1} A s_{l}. 
\end{equation}
Using \eqref{eqCosetRepP=B_kB_(n-k)} it follows that
\begin{equation}\label{eqAslBranchingBnoverB_kB_(n-k)}
\begin{array}{rcl}
A_{s_l}&=&s_l A s_l^{-1}\cap P=\left( S_{l}\times S_{k-l}\times S_{i-l}\times S_{n-i-k+l}\right)\ltimes Z_{2}^n\\
&=& B_{l}\times B_{k-l}\times B_{i-l}\times B_{n-i-k+l}\\
B_{s_l}&=&s_{l}^{-1} A s_{l}= S_{l}\times S_{i-l}\times S_{k-l}\times S_{n-i-k+l}\ltimes \mathbb Z_2^n\\
&=& B_{l}\times B_{i-l}\times B_{k-l}\times B_{n-i-k+l}.
\end{array}
\end{equation}
Here, the products of type $A$ factors are realized as the permutation groups of four groups of consecutive indices. 
In particular the four factors  $S_{l}\times S_{k-l}\times S_{i-l}\times S_{n-i-k+l}$ corresponding to $A_{s_l}$ permute the four groups of indices appearing in the second row of figure \eqref{eqCosetRepP=B_kB_(n-k)}, and the four factors $S_{l}\times S_{i-l}\times S_{k-l}\times S_{n-i-k+l}$ corresponding to $B_{s_l}$ permute the four group of indices appearing in the second row of figure \eqref{eqCosetRepP=B_kB_(n-k)}.

Let $\chi_{p}^{q}$ denote the $\mathbb Z_2^q$-character $\underbrace{\epsilon_1\otimes \dots \otimes \epsilon_1}_{p\text{ times}} \otimes \underbrace{ \epsilon_{-1}\otimes \dots\otimes \epsilon_{-1}}_{q-p\text{ times}}$ (in this notation, $\chi_i=\chi_i^n$).  
Given a $A=B_{i}\times B_{n-i}$-module $ V_\lambda\otimes V_\mu\otimes \chi_i$, we have an $B_{s_l}$-module isomorphism:
\begin{equation}\label{eqSplitVlambdatimesVmutimeschioverBsl}
\renewcommand{\arraystretch}{1.6}
\begin{array}{@{}r@{}c@{}l}
\displaystyle \left(V_\lambda \otimes V_\mu\otimes \chi_i \right)_{|B_{s_l}}&=&\displaystyle  \left(V_\lambda \otimes V_\mu\otimes\chi_i \right)_{|\left(B_{l}\times B_{i - l}\right)\times \left(B_{k-l}\times B_{n-i-k+l}\right)}\\
&\simeq &\displaystyle  \left(V_\lambda \otimes \chi_l^i \right)_{B_{l}\times B_{i - l}} \otimes \left(  V_\mu \otimes \chi_{k-l}^{n-i}\right)_{| B_{k-l}\times B_{n-i-k+l}}\\
&\simeq &\displaystyle \left( V_\lambda  \right)_{| S_{l}\times S_{i-l}}\otimes \chi_l^{i} \otimes  \left( V_\mu \right)_{|S_{k-l}\times S_{n-i-k+l}}\otimes \chi_{k-l}^{n-i}\\
&\simeq &\displaystyle \left(\bigoplus_{\substack{|\nu|=l\\|\xi|=i-l }} c_{\nu, \xi}^\lambda V_\nu\otimes V_\xi \right)\otimes \chi_{l}^i \otimes \left( \bigoplus_{\substack{|\sigma|=k-l\\|\zeta|=n-i-k+l } }  c_{\sigma, \zeta}^\mu V_\sigma\otimes V_\zeta \right)\otimes  \chi_{k-l}^{n-i}.
\end{array}
\end{equation}

Given an $A=B_{i}\times B_{n-i}$-module $V$,  \eqref{eqVs-DoubleCosetRep} equips it with a structure of an $A_{s_l}$-module denoted by $V_{s_l}$.
The structure of the module $V_{s_l}$ can be recovered only from the structure of $V$ as a $B_{s_l}$-module.
Now the definition \eqref{eqBslDef} of $B_{s_l}$ and the definition \eqref{eqVs-DoubleCosetRep} of $V_{s_l}$ imply that 
\begin{equation}\label{eqfrom-Bsl-module-V-to-Asl-module-Vsl}
\left(\left( V_\nu\otimes V_\xi\otimes \chi_l^i\right) \otimes\left( V_\sigma \otimes V_\zeta\otimes \chi_{k-l}^{n-i}\right)\right)_{s_l} \simeq \left( V_\nu\otimes V_{\sigma}  \otimes \chi^k_l\right)\otimes \left( V_\xi\otimes V_{\zeta}  \otimes \chi^{n-k}_{i-l}\right),
\end{equation}
where $|\nu|=l, |\xi|=i-l, |\sigma|=k-l, |\zeta|=n-i-k+l$.
Finally we can compute the $A_{s_l}$-module isomorphism:
\begin{equation}\label{eqSplitVlambdatimesVmutimeschisloverAsl}
\renewcommand{\arraystretch}{1.6}
\begin{array}{@{}r@{}c@{}l@{}l|@{}l}
\displaystyle \left(V_\lambda \otimes V_\mu\otimes \chi_i \right)_{s_l}& \simeq &\displaystyle  \left(\left(\bigoplus_{|\nu|=l } c_{\nu, \xi }^\lambda V_\nu\otimes V_\xi \right)\otimes \chi_{l}^i \otimes \left( \bigoplus_{|\sigma|=k-l }  c_{\sigma, \zeta}^\mu V_\sigma\otimes V_\zeta \right)\otimes  \chi_{k-l}^{n-i} \right)_{s_l}&& \begin{array}{l} \text{use} \\ \eqref{eqSplitVlambdatimesVmutimeschioverBsl} \end{array} \\
&\simeq&\displaystyle \bigoplus_{\substack{|\sigma|=k-l\\ |\zeta|= n-i -k+ l\\ |\nu|=l\\|\xi|=i-l }} c_{\nu, \xi}^\lambda c_{\sigma, \zeta}^\mu \left(\left( V_\nu\otimes V_\xi\otimes \chi_l^i\right) \otimes\left( V_\sigma \otimes V_\zeta\otimes \chi_{k-l}^{n-i}\right)\right)_{s_l} \\
&\simeq &\displaystyle   \bigoplus_{\substack{|\sigma|=k-l\\ |\zeta|= n-i -k+ l\\ |\nu|=l\\|\xi|=i-l }} c_{\nu, \xi}^\lambda c_{\sigma, \zeta}^\mu \left( V_\nu\otimes V_{\sigma}  \otimes \chi^k_l \right) \otimes \left( V_\xi\otimes V_{\zeta}  \otimes \chi^{n-k}_{i-l}\right) . &&\begin{array}{l}\text{use}\\\eqref{eqfrom-Bsl-module-V-to-Asl-module-Vsl}\end{array} \\
\end{array}
\end{equation}

The rest of the statement now follows from Proposition  \ref{propRestrictionInducedModules}:
\[
\begin{array}{@{}r@{~}c@{~}l@{}l@{}|l}
\displaystyle {V_{\lambda, \mu}}_{|P}&=&\displaystyle  \Ind_{A}^{B_n} \left(V_\lambda\otimes V_\mu\otimes \chi_i \right)_{|P}\\
&=& \displaystyle  \bigoplus_{s_l\in S} \Ind_{A_{s_l}}^{B_k\times B_{n-k}}\left(  \left(V_\lambda\otimes V_\mu\otimes \chi_i \right)_{s_l} \right) &&\text{Prop. } \ref{propRestrictionInducedModules}\\
&=&\displaystyle \bigoplus_{\substack{s_l\in S\\|\sigma|=k-l\\ |\zeta|= n-i -k+ l\\ |\nu|=l\\|\xi|=i-l }} c_{\nu, \xi}^\lambda c_{\sigma, \zeta}^\mu \left(\begin{array}{l} \Ind_{A_{s_l}}^{B_{k}\times B_{n-k}} \left( \left(V_{\nu}\otimes V_{\sigma}\otimes \chi_{l}^k \right)\right. \\ \quad \quad \quad \quad \quad \quad  \left. \otimes\left( V_\xi\otimes V_{\zeta}  \otimes \chi^{n-k}_{i-l}\right) \right) \end{array} \right) &&\begin{array}{l} \text{use}\\ \eqref{eqSplitVlambdatimesVmutimeschisloverAsl} \end{array} \\
&\simeq &\displaystyle \bigoplus_{\substack{s_l\in S\\|\sigma|=k-l\\ |\zeta|= n-i -k+ l\\ |\nu|=l\\|\xi|=i-l }}c_{\nu, \xi}^\lambda c_{\sigma, \zeta}^\mu\left( \begin{array}{l} \Ind_{B_{l}\times B_{k-l}}^{B_{k}} \left(V_{\nu}\otimes V_{\sigma}\otimes \chi_{l}^k \right) \\ \otimes  \Ind_{B_{i-l}\times B_{n-k-i+l}}^{B_{n-k}}  \left( V_\xi\otimes V_{\zeta}  \otimes \chi^{n- k}_{i -l}\right) \end{array} \right) \\
&\simeq&\displaystyle \bigoplus_{\substack{s_l\in S\\|\sigma|=k-l\\ |\zeta|= n-i -k+ l\\ |\nu|=l\\|\xi|=i-l }} c_{\nu, \xi}^\lambda c_{\sigma, \zeta}^\mu V_{\nu, \sigma}\otimes V_{\xi, \zeta} &&\text{Cor. }\ref{cor:typeB:irreps}\\
&=&\displaystyle \bigoplus_{\substack{|\sigma|+\nu=k\\ |\zeta|+|\xi|=n-k }} c_{\nu, \xi}^\lambda c_{\sigma, \zeta}^\mu V_{\nu, \sigma}\otimes V_{\xi, \zeta},
\end{array}
\]
which completes the proof of the proposition.
\end{proof}
As a side result, in Corollaries 
\ref{cor:typeB:BranchingOverTypeBGeneralizedParabolics} and \ref{cor:typeB:BranchingOverGeneralizedParabolics} we state branching laws of an arbitrary irreducible $B_n$-representation over generalized parabolic subalgebras. 
We note that these branching laws present an alternative way - which we did not follow as it appears to be more laborious - to prove the results of Section \ref{sec:typeB:multSignRep}.

\begin{corollary}\label{cor:typeB:BranchingOverTypeBGeneralizedParabolics}
\label{thm:branchbprod}
Let $P=P_{\emptyset, (p_1, \dots, p_l)}= B_{p_1}\times\dots\times B_{p_l}$, where $\abs{\lambda}+\abs{\mu}=p_1+\dots+p_l=n$ and let $\nu_1, \dots, \nu_k$ and $\sigma_1, \dots, \sigma_k$ be partitions with $|\nu_i|+|\sigma_i|=p_i$.
Then the multiplicity of the $P$-module $V_{\nu_1,\sigma_1}\otimes\dots\otimes V_{\nu_k,\sigma_k}$ in ${V_{\lambda, \mu}}_{|P}$ is given by:
\[
\begin{array}{rcl}
&&\displaystyle \dim \Hom_{P}\left(V_{ \nu_1,\sigma_1 }\otimes\dots\otimes V_{\nu_k,\sigma_k}, V_{\lambda, \mu}\right) \\
&=&\displaystyle \sum_{\substack{
\xi_1,\dots,\xi_{k-1} \\ \zeta_1,\dots,\zeta_{k-1} }} c^{\lambda }_{\nu_1, \xi_1}c^{\xi_1}_{\nu_2,\xi_2}\dots c^{\xi_{k-2}}_{\nu_{k-1}, \xi_{k-1}}  c^{\xi_{ k-1}}_{\nu_k,\emptyset}\, c^{\mu}_{ \sigma_1, \zeta_1} c^{\zeta_1}_{\sigma_2,\zeta_2}\dots c^{\zeta_{k-2}}_{\sigma_{k-1}, \zeta_{k-1}} c^{\zeta_{k -1} }_{\sigma_k, \emptyset}.
\end{array}
\]
\end{corollary}
\begin{proof}
The proof consists in consecutively applying the branching law \eqref{eqBranchingB_nToB_i+B_(n-i)} to the embeddings 
\[
\begin{array}{rcl}
B_{p_1}\times B_{n-p_1}&\hookrightarrow&  B_{n}\\
B_{p_2}\times B_{n-p_1-p_2}&\hookrightarrow& B_{n-p_1}\\
&\vdots& \\
B_{p_{k-1}}\times B_{p_k}&\hookrightarrow& B_{p_{k-1}+p_{k}};
\end{array}
\]
we omit the details of the computation.
\end{proof}

Proposition \ref{prop:typeB:branchingOverLargeSubgroups} and Corollary \ref{cor:typeB:BranchingOverTypeBGeneralizedParabolics} imply the following.

\begin{corollary}\label{cor:typeB:BranchingOverGeneralizedParabolics}
Let $\abs{\lambda}+\abs{\mu}=p_1+\dots+p_l=n$ and $1\leq k\leq l$, let $P=P_{(p_1, \dots, p_k), (p_{k+1},\dots, p_l)}= S_{p_1}\times \dots\times S_{p_k}\times B_{p_{k+1}} \times\dots\times B_{p_l}$. Let $\rho_1,\dots, \rho_k$ be partitions with $|\rho_i|=p_1,\dots, |\rho_k|=p_k$
and let $\nu_{k+1}, \dots, \nu_l$ and $\sigma_{k+1}, \dots, \sigma_l$ be partitions such that $|\nu_i|+|\sigma_i|=p_i$ for all $i>k$. 
Then the multiplicity of the $P$-module $V_{\rho_1}\otimes\dots\otimes V_{\rho_k}\otimes V_{\nu_{k+1},\sigma_{k+1}}\otimes\dots\otimes V_{\nu_l,\sigma_l}$ in ${V_{\lambda,\mu}}_{|P}$ is given by 
\[
\begin{array}{@{}l}
\displaystyle \dim\Hom_P\left(V_{\rho_1}\otimes \dots\otimes V_{\rho_k}\otimes V_{\nu_{k+1}, \sigma_{k+1}} \otimes\dots \otimes V_{\nu_l, \sigma_l}, {V_{\lambda, \mu}}\right)=\\ 
\displaystyle \sum_{\substack{
\nu_1,\dots,\nu_k \\ \sigma_1,\dots,\sigma_k \\ 
\abs{\nu_i} + \abs{\sigma_i} = p_i, i\leq k\\
\xi_1,\dots,\xi_{l-1} \\ \zeta_1, \dots, \zeta_{l-1} \\
}} \!\!\!\!\!\!\!\!\!\!\! c^{\rho_1}_{\nu_1, \sigma_1}\dots c^{\rho_k }_{\nu_k,\sigma_k}\, c^{\lambda }_{\nu_1, \xi_1} c^{ \xi_1}_{\nu_2,\xi_2}\dots c^{\xi_{l-2}}_{\nu_{l-1}, \xi_{l-1}} c^{ \xi_{l-1}}_{\nu_l,\emptyset}\, c^{\mu }_{ \sigma_1, \zeta_1} c^{\zeta_1}_{\sigma_2, \zeta_2}\dots c^{\zeta_{l-2}}_{\sigma_{l-1}, \zeta_{l-1} } c^{\zeta_{l-1}}_{\sigma_l, \emptyset}.
\end{array}
\]
\end{corollary}
\begin{proof}
Use the result of Corollary \ref{thm:branchbprod} to branch $B_n$ over $B_{p_1}\times \dots \times  B_{p_l}$ and then apply the second rule in Proposition \ref{prop:typeB:branchingOverLargeSubgroups} to branch the first $k$-tuple of groups $B_{p_1}\times \dots\times  B_{p_k}$ over $S_{p_1}\times \dots \times S_{p_k}$.
\end{proof}

\subsection{Sign representation multiplicities when restricted to parabolic subgroups in type $B$}\label{sec:typeB:multSignRep}
In the present section we obtain formulas for the multiplicity of the sign representation of a subgroup of the form $S_{p_1}\times \dots \times S_{p_k}\times B_{p_1}\times \dots\times B_{p_l}$ (see \eqref{eqParabolicsGeneralizedBn}) in an arbitrary irreducible $B_n$-representation $V_{\lambda, \mu}$.

We need the following technical statement.
\begin{proposition}
\begin{equation}\label{eqShurPolyMultIdentity}
\sum_{\alpha+\beta=\gamma} K_{\lambda, \alpha}K_{\mu, \beta}= \sum_{\nu} c_{\lambda, \mu}^\nu K_{\nu, \gamma}\quad .
\end{equation}
where $\alpha, \beta, \gamma$ stand for arbitrary tuples of non-negative integers (not necessarily partitions) and $\lambda,\mu, \nu $ denote partitions. 
\end{proposition}
This proposition follows from comparing coefficients in the Schur polynomial multiplication identities  \cite[page 66, (4)]{Fulton:YoungTableaux} and we omit its proof. An alternative proof can also be extracted directly from Proposition \ref{prop:KostkaNumberInterpretation}.

Recall from \eqref{eqParabolicsMixedTypeInBn}, \eqref{eqParabolicsTypeAInBn} and \eqref{eqParabolicsGeneralizedBn} that $P_{(p_1, \dots, p_k), \emptyset }$ denotes a parabolic subgroup that is a sum of type $A$ subgroups, $ P_{(p_1, \dots, p_k), (p_{k+1})}$ denotes a parabolic subgroup whose last component is of type $B_{p_{k+1}}$, and $P_{(p_1, \dots, p_k), (p_{k+1}, \dots, p_l)} = S_{p_1}\times \dots \times S_{p_k}\times B_{p_{k+1}}\times \dots \times B_{p_l}$ stands for the natural generalization of a parabolic subgroup in type $B$. 
Just as in type $A$, for a subgroup $P$ of $B_n$ generated by reflections, let $\signRep{P}$ denote the sign representation of $P$.

For a tuple of non-negative integers $ \alpha=(\alpha_1, \dots, \alpha_k )$, let $d(\alpha)$ stand for the number of entries of $\alpha$ excluding the possible trailing zeroes at the end, that is, 
\[
d(\alpha)=\text{largest } j \text{ for which } t>j \text{ implies } a_t=0.
\] 

\begin{proposition}\label{prop:typeB}
The multiplicity  $ \dim\Hom_{P}\left(\signRep{P}, V_{\lambda, \mu}\right) $ of the sign representation of  $P=P_{(p_1, \dots, p_k), (p_{k+1}, \dots, p_l)} $  in $V_{\lambda, \mu}$ is given by
\begin{equation}\label{eqMultSignoverGeneralizedParabolicsTypeB}
\sum_{\substack{\alpha + \beta = (p_1,\dots,p_l) \\ d(\alpha) \leq k}} K_{\lambda^*,\alpha}K_{\mu^*,\beta},
\end{equation}
where $\alpha, \beta$ stand for tuples of non-negative integers. \end{proposition}
We note that when we specialize the above proposition to $k=l $ and $k=l+1$ we get the multiplicities of $V_{\lambda, \mu}$ over the (non-generalized) parabolic subgroups in type $ B_n$.
We state those in the following.
\begin{corollary}~\label{cor:typeBsignMultNonGeneralizedParabolics}
\begin{enumerate}
\item The multiplicity $\dim \Hom_{P} \left(\signRep{P}, V_{\lambda, \mu} \right)$ of the sign representation of $P=P_{(p_1,\dots, p_k ), \emptyset}$ in $V_{\lambda, \mu}$ is given by
\begin{equation}\label{eq:typeB:MultSignParabolicNoTypeB}
\displaystyle \dim \Hom_{P} \left(\signRep{P}, V_{\lambda, \mu} \right)\\
=\displaystyle  \sum_{\alpha+\beta= (p_1,\dots, p_k)} K_{\lambda^*, \alpha} K_{\mu^*, \beta},
\end{equation}
where $\alpha, \beta$ stand for tuples of non-negative integers (not necessarily partitions).
\item The multiplicity $\dim \Hom_{ P}\left(\signRep{P}, V_{\lambda, \mu}\right)$ of the sign representation of $ P=  P_{(p_1, \dots, p_k), (p_{k+1})}$ in $V_{\lambda, \mu}$ is given by
\begin{equation}\label{eq:typeB:MultSignParabolicWithTypeB}
\dim \Hom_{ P} \left(\signRep{P}, V_{\lambda,\mu} \right)= \sum_{\alpha+\beta=(p_1,\dots, p_{k})} K_{\lambda^*,\alpha} K_{\mu^*, \beta\oplus (p_{k+1})},
\end{equation}
where $\alpha, \beta$ stand for tuples of non-negative integers and $\beta\oplus (p_{k+1})$ stands for $(\beta_1, \dots, \beta_{k}, p_{k+1})$, where $\beta=(\beta_1, \dots, \beta_{k})$.
\end{enumerate}
\end{corollary}

\begin{proof}[Proof of Proposition \ref{prop:typeB}]
The proof combines the branching laws given in Proposition \ref{prop:typeB:branchingOverLargeSubgroups} with the already established branching laws in type $A$.

We proceed by induction on $l-k$, the number of type $B$ components in the generalized parabolic subgroup. 
We start with the base case $l-k=0$, which has been formulated in \eqref{eq:typeB:MultSignParabolicNoTypeB}.

\noindent \emph{Proof of \eqref{eq:typeB:MultSignParabolicNoTypeB}.} Set $P=P_{(p_1, \dots, p_k), \emptyset}$ and compute:

\noindent $
\begin{array}{@{}r@{}c@{}l@{}l@{}|l}
\displaystyle \dim \Hom_{P} \left(\signRep{P}, V_{\lambda, \mu} \right)&= & \displaystyle \dim\Hom_{P}\left(\signRep{P}, \bigoplus_{\nu= \lambda+ \mu} c_{\lambda,\mu}^\nu V_\nu\right)&& \text{Prop. }\ref{prop:typeB:branchingOverLargeSubgroups}\\
&=&\displaystyle  \sum_{\nu=\lambda+\mu} c_{\lambda,\mu}^\nu \dim \Hom_{P}\left(\signRep{P}, V_\nu\right)\\
&=&\displaystyle  \sum_{\nu=\lambda+\mu} c_{\lambda,\mu}^\nu K_{\nu^*, (p_1, \dots, p_k)}&&\text{Prop. }\ref{prop:typeA}\\
&=&\displaystyle  \sum_{\nu=\lambda+\mu} c_{\mu^*, \lambda^*}^{\nu^*} K_{\nu^*, (p_1, \dots, p_k)}&& \text{\cite[page 62]{Fulton:YoungTableaux}}
\\
&=&\displaystyle  \sum_{\alpha+\beta=(p_1, \dots, p_k)} K_{\lambda^*,\alpha} K_{\mu^*, \beta}.&&\begin{array}{l}\text{use}\\ \eqref{eqShurPolyMultIdentity}\end{array}.
\end{array}
$

\noindent \emph{Proof of \eqref{eqMultSignoverGeneralizedParabolicsTypeB}.} Let $P=P_{(p_1, \dots, p_k), (p_{k+1}, \dots, p_{l})}=S_{p_1}\times \dots \times S_{p_k}\times B_{p_{k+1}}\times \dots \times B_{p_l}$. 
Denote by $Q$ the sum of the first $l-1$ factors of $ P$, i.e., 
\[
Q=S_{p_1}\times \dots \times S_{p_k}\times B_{p_{k+1}}\times \dots B_{p_{l-1}} .
\] 
We have the chain of embeddings 
\[
P=\displaystyle Q\times B_{p_l}\hookrightarrow B_{n-p_l}\times B_{p_l} \hookrightarrow B_n. 
\]
By induction hypothesis, we suppose we have proven our statement for the multiplicity of the sign representation of $Q$ in an irreducible $B_{n-p_k}$-representation. 
We proceed to compute $\dim\Hom_{P} \left(\signRep{P}, V_{\lambda, \mu} \right) $.

\medskip
\noindent $
\begin{array}{@{}r@{}c@{}l@{}l@{}|l}
\displaystyle \dim \Hom_{ P} \left(\signRep{ P}, V_{\lambda, \mu} \right)& =&\displaystyle  \dim \Hom_{ P} \left(\signRep {P},  \bigoplus_{ \substack{|\nu|+|\sigma|=n-p_l \\ |\xi|+|\zeta|=p_l}}  c_{\nu,\xi }^\lambda c_{\sigma, \zeta}^{\mu} V_{\nu, \sigma}\otimes V_{\xi, \zeta} \right) &&\begin{array}{l} \text{use} \\ \eqref{eqBranchingB_nToB_i+B_(n-i)} \end{array}\\
&=&\displaystyle \sum_{\substack{|\nu|+|\sigma|=n-p_l \\ |\xi|+|\zeta|=p_l}}c_{\nu,\xi }^ \lambda c_{\sigma, \zeta}^{\mu} \dim\Hom_{ P}\left(\signRep{P}, V_{\nu, \sigma}\otimes V_{\xi, \zeta}  \right)\\
&=&\displaystyle \sum_{\substack{|\nu|+|\sigma|=n-p_l \\ |\xi|+|\zeta|=p_l}} c_{\nu, \xi }^\lambda c_{\sigma, \zeta}^{\mu} \dim\Hom_{Q} \left( \signRep{Q}, V_{\nu, \sigma} \right) \\
&&\displaystyle \quad \quad \quad \quad \quad  \dim \Hom_{ B_{p_l}} \left(\signRep{B_{p_l}}, V_{\xi,\zeta} \right).
\end{array}
$

\medskip

Let $\pi$ denote the partition that corresponds to the sign representation of $S_{p_l}$, that is, $\pi=[\underbrace{1,\dots, 1}_{p_l}]$ times. Recall from \eqref{eqSignRepTypeB} that the sign representation $\signRep{B_{p_l}}$ equals $V_{\emptyset, \pi}$.
Therefore $\dim\Hom_{B_{p_l}} \left(\signRep{B_{p_l}}, V_{\xi,\zeta} \right)$ equals zero or one, with one achieved only when $\xi=\emptyset$  and $\zeta=\pi$. 
The rest of the proof consists is a computation using the induction hypothesis and \eqref{eqBoxRowColumnIdentity}.

\medskip
\noindent $
\begin{array}{@{}r@{}c@{}l@{}l@{}|l}
\displaystyle \dim \Hom_{ P} \left(\signRep{P}, V_{\lambda} \right)& =&\displaystyle \sum_{|\nu|+ |\sigma|=n-p_l } c_{\nu,\emptyset }^\lambda c_{ \sigma,\pi }^{\mu} \dim\Hom_{Q}\left(\signRep{Q}, V_{ \nu, \sigma} \right) \\
&=&\displaystyle \sum_{|\lambda|+|\sigma|= n-p_l}  c_{\sigma,\pi }^{\mu} \dim\Hom_{Q}\left(\signRep{Q}, V_{\lambda, \sigma} \right) \\
&=&\displaystyle \sum_{\substack{ |\sigma|+ p_l = n-|\lambda| \\ \alpha+ \beta' = (p_1, \dots, p_{l-1}) \\ d(\alpha)\leq k}}  c_{\sigma,\pi }^{\mu}  K_{\sigma^*, \beta'} K_{ \lambda^*, \alpha} && \begin{array}{l} \text{ind. hyp.} \end{array} \\
&=&\displaystyle \sum_{\substack{ \alpha+\beta'= \left( p_1, \dots, p_{l-1}\right) \\  d(\alpha)\leq k}}K_{ \lambda^*, \alpha} \sum_{|\sigma|+p_l=|\mu | } c_{\sigma, \pi }^{ \mu}  K_{\sigma^*,\beta'} \\
&=&\displaystyle \sum_{\substack{ \alpha+\beta'= \left(p_1, \dots, p_{l-1}\right) \\ d(\alpha)\leq k} }K_{ \lambda^*, \alpha} K_{\mu^* , \beta'\oplus (p_l)} &&\begin{array}{l} \text{use} \\ \eqref{eqBoxRowColumnIdentity}\end{array}\\
&=&\displaystyle \sum_{\substack{\alpha+\beta= \left(p_1, \dots, p_{l}\right)\\ d(\alpha)\leq k }}K_{ \lambda^*, \alpha} K_{\mu^* ,\beta }. 
\end{array}
$
\end{proof}
We finish this section by tabulating the sign multiplicities in $B_3$  in Table \ref{table:nonlin}.
The table entries indicate the multiplicities $\dim\Hom_P(\signRep{P}, V)$, where $P$ is the parabolic indicated in the column label and $V$ is the representation indicated by the row label. 
We also indicate the Dynkin type of the parabolic subgroup (here $A^{2}_1$ stands for short-rooted root subsystem and $A_1$ for long-rooted).

We illustrate the computation of a few table entries using our method. 
For example, consider $P=P_{(2),(1)}=S_2 \times B_1$ and $V=V_{[1,1],[1]}$, i.e., $\lambda=[1,1]$, $\mu=[1]$. 
In this case $\lambda^*=[2]$.  
To get a non-zero summand $K_{[2], \alpha}K_{[1],\beta\oplus(1)}$in \eqref{eq:typeB:MultSignParabolicWithTypeB} we need to have $\alpha =(2)$ and $\beta\oplus\left(p_{k+1}\right)=(0,1)$.
Hence, the multiplicity of the sign representation of $P$ in $V$ is $K_{[2], (2)} K_{[1],(0,1)} = 1\times 1 =1$.
As a second example, consider $P= P_{(2,1),\emptyset}=S_2\times S_1$ and the same representation $V=V_{[1,1],[1]}$. 
To get a non-zero summand $K_{[2],\alpha }K_{[1],\beta}$ in \eqref{eq:typeB:MultSignParabolicNoTypeB} we can only have $\alpha =(1,1)$, $\beta=(1,0)$ or $\alpha=(2)$ and $\beta=(0,1)$.
In both cases $K_{[2],\alpha }K_{[1],\beta} = 1$, and so the multiplicity of the sign representation of $P$ in $V$ is $2$.
The other entries of the table are computed in the same manner.

\begin{table}[ht]
\renewcommand{\arraystretch}{1}
\caption{Sign signatures in $B_3$} 
\begin{tabular}{c|
>{\centering\arraybackslash} p{1cm}
>{\centering\arraybackslash} p{1.1cm}
>{\centering\arraybackslash} p{1.2cm}
>{\centering\arraybackslash} p{1.2cm}
>{\centering\arraybackslash} p{1.2cm}
>{\centering\arraybackslash} p{1.2cm}
>{\centering\arraybackslash} p{1cm}} &$P_{\emptyset, (3)}$&$P_{(1), (2)}$&$P_{(1, 1), (1)}$&$P_{(2), (1)}$&$P_{(1, 1, 1), \emptyset}$&$P_{(2, 1), \emptyset}$&$P_{(3), \emptyset}$\\&$B_3$&$B_2$&$A^{2}_1$&$A^{2}_1+A_1$&$0$&$A_1$&$A_2$\\
Ir'reps. &$B_{3}$&$S_{1}\times B_{2}$&$S_{1}\times S_{1}\times B_{1}$&$S_{2}\times B_{1}$&$S_{1}\times S_{1}\times S_{1}$&$S_{2}\times S_{1}$&$S_{3}$\\\hline
$V_{\emptyset, [1, 1, 1]}$&$1$&$1$&$1$&$1$&$1$&$1$&$1$\\$V_{\emptyset, [2, 1]}$&$0$&$1$&$2$&$1$&$2$&$1$&$0$\\$V_{\emptyset, [3]}$&$0$&$0$&$1$&$0$&$1$&$0$&$0$\\$V_{[1], [1, 1]}$&$0$&$1$&$2$&$1$&$3$&$2$&$1$\\$V_{[1], [2]}$&$0$&$0$&$2$&$1$&$3$&$1$&$0$\\$V_{[1, 1], [1]}$&$0$&$0$&$1$&$1$&$3$&$2$&$1$\\$V_{[2], [1]}$&$0$&$0$&$1$&$0$&$3$&$1$&$0$\\$V_{[1, 1, 1], \emptyset}$&$0$&$0$&$0$&$0$&$1$&$1$&$1$\\$V_{[2, 1], \emptyset}$&$0$&$0$&$0$&$0$&$2$&$1$&$0$\\$V_{[3], \emptyset}$&$0$&$0$&$0$&$0$&$1$&$0$&$0$\\\end{tabular}

\label{table:nonlin} 
\end{table}

\subsection{Recovering representations from sign signatures/multiplicities in type $B$}
\subsubsection{Recovering irreducible representations in type $B$}\label{sec:typeB:recoveringIrreps}

\begin{theorem}
Let $V=V_{\lambda^*,\mu^*}$ be an irreducible representation of $B_n$.
Then $\lambda$ and $\mu$ can be computed from the sign signature of $V$ as follows:
\begin{enumerate}
\item Let $\alpha=(\alpha_1,\dots,\alpha_k)$ be the lexicographically largest partition such that $P_{\alpha,\emptyset}$ belongs to the sign signature of $V$.
\item For each $i$ between $1$ and $k$, let $\beta^{(i)}=(\beta^{(i)}_1,\dots,\beta^{(i)}_{k_i})$ be the lexicographically largest partition such that $P_{(\beta^{(i)}_1,\dots,\beta^{(i)}_{i-1},\beta^{(i)}_{i+1},\dots,\beta^{(i)}_{k_i}),(\beta^{(i)}_i)}$ lies in the sign signature of $V$.
\item Then $\lambda=(\beta^{(1)}_1,\dots,\beta^{(k)}_k)$ and $\mu=\alpha-\lambda$.
\end{enumerate}
\end{theorem}
\begin{proof}
This is a rephrasing of Corollary~\ref{cor:typeB:recoveringIrrepAlgorithm}
\end{proof}
\medskip
\begin{proposition}\label{prop:typeB:differentIrrepsDifferentSignSignatures}
For $(\lambda, \mu)\neq (\nu, \sigma)$ the $B_{n}$-representations $V_{\lambda, \mu}$ and $V_{\nu, \sigma}$ have different sign signatures.
\end{proposition}
\begin{proof}
Case 1. Suppose the partitions $\lambda^*+ \mu^*$ and $\nu^*+\sigma^*$ are different and $ \lambda^*+\mu^*> \nu^*+ \sigma^*$ in the lexicographic order. We claim that $P_{\lambda^*+\mu^*, \emptyset}$ belongs to the sign signature of $V_{\lambda, \mu}$ but not of $V_{\nu, \sigma}$. 
Indeed, by Corollary \ref{cor:typeBsignMultNonGeneralizedParabolics} we have that
\[
\dim\Hom_{P_{\lambda^*+\mu^*, \emptyset}} \left( \signRep{ P_{\lambda^*+\mu^*,\emptyset}}, V_{\nu, \sigma }\right)= \sum_{\alpha+ \beta=\lambda^* +\mu^* } K_{\nu^*, \alpha} K_{\sigma^*, \beta}=0,
\]
where we have used the fact that $\nu^*+\sigma^*< \alpha + \beta= \lambda^*+\mu^*$ implies that either $K_{\nu^*, \alpha}$ or $K_{\sigma^*, \beta}$ is zero.
On the other hand, the same consideration shows that 
\[
\begin{array}{@{}r@{}c@{}l}
\displaystyle \dim\Hom_{P_{\lambda^*+\mu^*, \emptyset}} \left(\signRep{ P_{\lambda^*+\mu^*, \emptyset }}, V_{\lambda, \mu }\right)&=& \displaystyle \sum_{ \alpha +\beta= \lambda^*+\mu^* } K_{\lambda^*, \alpha} K_{\mu^*, \beta}\\
&= &\displaystyle  K_{\lambda^*, \lambda^*}K_{\mu^*, \mu^*}=1.
\end{array}
\] 
Case 2. If $\lambda^*+\mu^*=\nu^*+\sigma^*$, we can assume $\mu^*>\sigma^*$. 
If the first different entry of $\mu^*$ and $\sigma^* $ occurs in the $t^{th}$ position, then we write $\mu^*_t> \sigma^*_t$. 
Let $\omega $ be the tuple obtained by replacing the $t^{th}$ entry of $\lambda^*+\mu^*$ with $\lambda^*_t$. 
The entries of $\omega$ are given by
\begin{equation}\label{eq:typeB:differentIrreps:omega}
\begin{array}{rcl}
\omega &=& \left(\lambda^*_1+\mu^*_1, \dots, \lambda^*_{t-1}+\mu^*_{t-1}, \lambda^*_t, \mu^*_{t+1}+ \lambda^*_{t+1}, \dots, \mu^*_{k}+\lambda^*_k\right).
\end{array}
\end{equation}
We claim that the parabolic subgroup $P_{\omega, \left(\mu_t^*\right)}$ belongs to the sign signature of  $V_{\lambda, \mu} $ but not of $V_{\nu, \sigma}$. 
In order to prove our claim, we first observe that 
\[
K_{\mu^*, \left(\omega-\lambda^*, \mu^*_t\right)} = K_{\mu^*,\left(\mu^*_1, \dots, \mu^*_{t-1},0, \mu^*_{t+1}, \dots, \mu^*_{k}, \mu^*_t\right)}= K_{\mu^*,\mu^*}=1 >0,
\]
where the second equality follows from Corollary \ref{cor:KostkaNumberInvarintUnderSn}.
Together with Corollary 
\ref{cor:typeBsignMultNonGeneralizedParabolics} this implies
\[
\begin{array}{rcl}
\displaystyle \dim \Hom_{ P_{\omega, (\mu^*_t)}} \left(\signRep { P_{\omega, \left( \mu^*_t \right)}}, V_{\lambda, \mu} \right)&=&\displaystyle \sum_{ \alpha +\beta= \omega} K_{\lambda^*, \alpha } K_{\mu^*, \beta\oplus\left( \mu^*_t\right)}\\
&\geq& \displaystyle K_{\lambda^*,\lambda^*} K_{ \mu^* , \left(\omega- \lambda^*\right)\oplus \left(\mu^*_t\right) }\\
&>&0,
\end{array}
\]
which shows  $ P_{\omega, \left(\mu^*_t\right)}$ belongs to the sign signature of $V_{\lambda,\mu}$.
Applied to the pair $\nu,\sigma$, Corollary \ref{cor:typeBsignMultNonGeneralizedParabolics} implies that 
\[
\dim\Hom_{ P_{\omega, (\mu^*_t)}} \left(\signRep{ P_{\omega, \left(\mu^*_t \right)} }, V_{\nu, \sigma }\right)= \sum_{ \alpha+\beta= \omega} K_{ \nu^*, \alpha} K_{\sigma^* , \beta\oplus \left( \mu^*_t\right)}\quad. 
\]
 All the above summands  are zero. 
To see that, observe first that the entries of $\alpha +\beta=\omega$ and $\nu^*+\sigma^*$ coincide for the first $t-1$ entries. 
To fill two tableaux whose first row sums are $\omega_1,\dots, \omega_{t-1} $ with contents $\alpha, \beta$ whose sum has the same first $t-1$ entries one needs to have that the first $t-1$ entries of $\alpha$ equal those of $\nu^*$ and the first $t-1$ entries of $\beta$ equal those of $\sigma^*$. 
It follows that $K_{ \nu^*, \alpha} K_{\sigma^* , \beta\oplus\left( \mu^*_t\right)}$ could possibly be non-zero only if $\alpha$ and $\nu^*$ as well as $\beta$ and $\sigma^*$ have coinciding first $t-1$ entries. 
Now suppose the first $t-1$ entries of $\beta $ and $\sigma^*$ coincide. 
To complete $\sigma^*$ to a tableau with content $\beta\oplus \left( \mu^*_t\right)$ we need to place $\mu^*_t $ copies of the value $k+1$ in the rows with index $\geq t$. 
This cannot be done: the $t^{th}$ row contains $\sigma^*_t<\mu^*_t $ entries and therefore a numbering with the already fixed first $t-1$ rows and content $\beta\oplus \left( \mu^*_t\right)$ would have the entry $k+1 $ appearing twice in the same column. 
It follows that  $K_{ \nu^*, \alpha} K_{\sigma^* , \beta\oplus \left( \mu^*_t\right) }$ is zero. 
This proves  the proposition. 
\end{proof}
The proof of Proposition \ref{prop:typeB:differentIrrepsDifferentSignSignatures} suggests an algorithm to recover the pair $(\lambda, \mu)$ if we know the sign signature of $ V_{\lambda, \mu}$. We shall now formulate the algorithm.
\begin{corollary} \label{cor:typeB:recoveringIrrepAlgorithm}
Let $V_{\lambda^*, \mu^*}$ be an irreducible representation of $B_n $, where we recall that $\lambda^*$ and $\mu^*$ denote the dual partitions of $\lambda=[\lambda_1, \dots, \lambda_k]$ and  $\mu=[\mu_1, \dots, \mu_k]$. Here we potentially add zeros to ensure that both $\mu$ and $\lambda$ have the same number of entries. 

Then $\lambda_i, \mu_i$ are computed consecutively from the following equalities.

\medskip

$\begin{array}{@{}r@{}c@{}l}  
\lambda_1 & = &\max\{p_{k+1} | \dim\Hom_{ P_{\left( p_1, \dots, p_k\right), \left(p_{k+1}\right)}}\left(\signRep{  P_{ \left(p_1, \dots, p_k\right), \left(p_{k+1}\right)}}, V_{\lambda^*,\mu^* }\right) >0  \} \\
\lambda_1+\mu_1 &=& \max\{p_1 | \dim\Hom_{P_{\left(p_1, \dots, p_k\right), \emptyset}}\left(\signRep{ P_{\left(p_1, \dots, p_k\right), \emptyset}} , V_{\lambda^*,\mu^* } \right) >0  \} \\
&\vdots& \\
\lambda_i  &=&  \max\{p_{k+1} | \dim \Hom_{ P_{\left(\lambda_1+\mu_1,\dots, \lambda_{i-1}+ \mu_{i-1}, p_i,  \dots, p_k\right),\left(p_{k+1}\right)}} \\
&& \quad \quad \quad \quad \quad \quad \quad \quad \left(\signRep{P_{\left(\lambda_1+\mu_1,\dots, \lambda_{i-1}+\mu_{i-1},p_i, \dots, p_k\right),\left(p_{k+1}\right)}} , V_{\lambda^*, \mu^* } \right) >0 
 \}\\
\lambda_i+\mu_i &=&  \max\{p_i | \dim \Hom_{P_{\left( \lambda_1+ \mu_1,\dots, \lambda_{i-1}+ \mu_{i-1}, p_i, \dots, p_k\right), \emptyset }} \\
&& \quad \quad \quad \quad \quad \quad \quad \quad \left( \signRep{ P_{\left(\lambda_1+\mu_1,\dots, \lambda_{i-1} +\mu_{i-1}, p_i, p_{i+1}, \dots\right), \emptyset } }, V_{\lambda^*,\mu^* } \right) >0 
 \}\\
&\vdots& \\
\end{array}
$
\end{corollary}
The proof of Corollary \ref{cor:typeB:recoveringIrrepAlgorithm} follows from the fact that 
\[
\dim\Hom \left(\signRep{ P_{\omega, \emptyset} }, V_{\mu^*, \lambda^* }\right)\left\{\begin{array}{l}=0 \text{ when }\omega>\lambda+\mu \\>0 \text{ when }\omega=\lambda+\mu \end{array}\right.
\]
and the fact that 
\[
\dim\Hom_{ P_{\omega, \left(p_{k+1}\right)}} \left(\signRep{  P_{\omega, \left(p_{k+1}\right)} }, V_{\mu^*, \lambda^* } \right) \left\{\begin{array}{ll}=0& \text{when the first entry }\\ & \omega_i \text{ for which } \\
&\lambda_i + \mu_i \neq \omega_i \\
& \text{ satisfies } \omega_i>\mu_i \\>0& \text{when }\omega\text{ is given by } \eqref{eq:typeB:differentIrreps:omega} \end{array}\right. ;
\]
these facts in turn follow from Case 1 and Case 2 of the proof of Corollary \ref{prop:typeB:differentIrrepsDifferentSignSignatures}.

\subsubsection{Recovering arbitrary representations in type $B$}
In the previous paragraph we described how to recover irreducible representations from their sign signatures. In this section we extend the result to arbitrary representations. To do this we need to generalize the notion of a sign signature to allow generalized parabolic subgroups. Given a representation $V$, define the generalized sign signature 
\[
\signSignatureGeneralized(V)
\]
of $V $ as the set of generalized parabolics whose sign representation appears as a subrepresentation of $V$.

For pairs of partitions $(\lambda ,\mu ) $ and $(\nu, \sigma)$, define an order $\succ$ by requesting that $(\lambda, \mu)\succ  (\nu, \sigma) $ if and only if $\mu^* > \sigma^*$ in the lexicographic order or $\mu=\sigma$ and $\lambda^*> \nu^*$ in the lexicographic order. 
Let
\[
\left(\lambda_1, \mu_1\right)\succ \dots \succ \left(\lambda_M,\mu_M \right)
\]
be all pairs of partitions with $ |\lambda_j|+ |\mu_j|=n$, ordered according to $\succ$, as indicated above. 
For every allowed index $j$ let $P_j$ be the generalized parabolic subgroup given by $P_j= P_{ \lambda^*_j, \mu^*_j}$. Let $\Pi$ be the set of the $P_i$'s:
\[
\Pi=\{P_{1},\dots, P_M \}.
\]
Just as in type $A$, the set $\Pi$ contains exactly one representative of each of the generalized parabolic subgroups of $B_n$, and the generalized sign signature of a representation is determined by its intersection with $\Pi$.

\begin{proposition}~\label{prop:typeB:VectorSpaceBasisBothParts}
\begin{enumerate}
\item \label{prop:typeB:VectorSpaceBasisPart1} The multiplicity $\dim \Hom_{P_{\nu^*, \sigma^* }}\left(\signRep{ P_{ \nu^*, \sigma^* }}, V_{\lambda, \mu} \right) $ of the sign representation of $P_{\nu^*, \sigma^*}$ in $V_{\lambda, \mu}$ equals $0$ if $(\nu, \sigma)\succ (\lambda, \mu)$ and equals $1$ if $(\nu, \sigma)=(\lambda, \mu)$. 
\item \label{prop:typeB:VectorSpaceBasisPart2} The generalized sign signature $\signSignatureGeneralized V_{\lambda_i}$ contains $P_i $ and does not contain $P_{1}, \dots, P_{i-1}$.
\end{enumerate}

\end{proposition}
\begin{proof}
\ref{prop:typeB:VectorSpaceBasisPart2}) follows from \ref{prop:typeB:VectorSpaceBasisPart1}); we are proving \ref{prop:typeB:VectorSpaceBasisPart1}). We treat first the case $(\lambda,\mu )=(\nu, \sigma)$.  

Now suppose $(\nu, \sigma)\succeq (\lambda, \mu)$. By Proposition \ref{prop:typeB} the multiplicity of the sign representation of $P_{\nu^*, \sigma^*}$ in $V_{\lambda,\mu}$ is

\[
\dim \Hom_{P_{\nu^*, \sigma^* }}\left(\signRep{ P_{ \nu^*, \sigma^* }}, V_{\lambda, \mu} \right) = \displaystyle \sum_{\substack{\alpha+\beta= \nu^*\oplus\sigma^*\\ d(\alpha)\leq k }}K_{ \lambda^*, \alpha} K_{\mu^* ,\beta }.
\]
It follows that for each $\beta$ we have that $\beta\geq (0,\dots, 0)\oplus \sigma^* $ (in the lexicographic order). 
Considering the number of semistandard tableaux of shape $ \mu^*$ we see that $K_{\mu^*, \beta}$ is zero when $\beta>(0,\dots, 0)\oplus \mu^*$ and equals one only in the case $\beta=(0,\dots, 0)\oplus \mu^* $. 
Since $(\nu, \sigma)\succ (\lambda, \mu)$, this is possible only when $\sigma^*  =\mu^*$. 
Consequently the term  $K_{\lambda^*, \alpha}K_{\mu^*,\beta} $ can be non-zero when $\beta=(0,\dots, 0)\oplus \mu^*$,  $\alpha=\nu^*$ and $\sigma^*  =\mu^*$. 
Since $(\nu, \sigma)\succ (\lambda, \mu)$ and $\sigma=\mu$ it follows that a non-zero term has $\alpha=\nu^*\geq\lambda^*$. 
Finally, a consideration of the number of semistandard tableaux of shape $\lambda^*$ shows $K_{\lambda^*, \alpha}K_{\mu^*,\beta}$ is one only when $(\lambda, \mu)=(\nu, \sigma)$ and $\alpha=\lambda^*$ and $\beta=(0,\dots, 0)\oplus \mu$, and is zero in all other cases. 
This completes the proof of the statement.
\end{proof}

Let $\mathbb C[\Pi]$ be the formal vector space generated freely over $\mathbb C$ by elements of  $\Pi$. 
In other words, an element of $\mathbb C[\Pi]$ is a formal sum of the form $\displaystyle \sum_{\lambda, \mu} a_{\lambda,\mu} P_{\lambda, \mu}$, where $a_{\lambda, \mu}$ are arbitrary complex numbers and $\lambda, \mu$ run over the partitions for which $|\lambda|+|\mu|=n$.
\begin{definition}\label{def:typeB:SignMult}
Given a representation $V$, define $\signMult (V)\in \mathbb C[\Pi] $ as the formal sum 
\[
\signMult(V)=\sum_{|\lambda|+|\mu|=n} \dim\Hom_{P_{\lambda,\mu}} \left(\signRep{ P_{\lambda, \mu}}, V \right) P_{\lambda, \mu} \quad .
\]
\end{definition} 
In the terminology of Definition \ref{def:typeB:SignMult}, Proposition \ref{prop:typeB:VectorSpaceBasisBothParts}.(\ref{prop:typeB:VectorSpaceBasisPart2}) can be restated by saying that the matrix formed by the coordinates $a_{\nu, \sigma}$ of the irreducible representations $V_{\lambda, \mu}$ (both listed in ascending order relative to $\succ$), form a triangular matrix with non-zero diagonal entries.
Linear algebra implies the following.
\begin{corollary}\label{cor:typeB:recoverArbitraryRepFromGeneralizedSignSignature}
~
\begin{enumerate}
\item The vectors $\signMult(V_{\lambda, \mu}) $ obtained by letting $V_{\lambda, \mu}$ run over the set of irreducible representations of $B_n$ are linearly independent.
\item $\signMult(V)=\signMult(U)$ if and only if $V\simeq U$.
\item \label{cor:typeB:MatrixSignMult} The matrix $ \left(\dim\Hom_{P_{\nu,\sigma}}\left( \signRep{ P_{\nu,\sigma}}, V_{\lambda, \mu} \right) \right)_{ |\nu|+|\sigma|= |\lambda| +|\mu| =n}$ is square and has non-zero determinant.
\end{enumerate}
\end{corollary}
Thus, an arbitrary representation $V$ is identified up to isomorphism by the tuple $ \left(\dim \Hom_{P_{\nu,\sigma}} \left( \signRep { P_{\nu, \sigma}}, V\right), \dots \right)_{|\nu|+|\sigma|=n} $. 
Just as in type $A$, to recover the multiplicities of each irreducible constituent of $V$ from $\signMult V$, one multiplies the vector-column $ \left(\dim \Hom_{P_{\nu,\sigma}} \left( \signRep{P_{\nu, \sigma}}, V\right), \dots\right)_{|\nu| +| \sigma| =n} $ on the left by inverse of the matrix given in Corollary \ref{cor:typeB:recoverArbitraryRepFromGeneralizedSignSignature}.\ref{cor:typeB:MatrixSignMult}).

\section{Type $D$}\label{sec:typeD}

In this section, we describe how to recover irreducible representations in type $D$ from sign signatures.
Our main result is the following algorithm.

\begin{theorem}\label{thm:TypeDRecoverIrrepsAlgorithm}
Let $V$ be an irreducible representation of $D_n$.
Then the isomorphism type of $V$ can be computed from its sign signature as follows.
\begin{enumerate}
\item Let $\alpha_+$ (resp.\ $\alpha_-$) be the lexicographically largest partition such that $P^+_{\alpha_+,\emptyset}$ (resp.\ $P^-_{\alpha_-,\emptyset}$) lies in the sign signature of $V$.
\item If $\alpha_+>\alpha_-$, put $\lambda=\alpha_+/2$.
Then $V$ is isomorphic to $V_{\{\lambda^*,\lambda^*\}}^+$.
\item If $\alpha_->\alpha_+$, put $\lambda=\alpha_-/2$.
Then $V$ is isomorphic to $V_{\{\lambda^*,\lambda^*\}}^-$.
\item Otherwise, $\alpha_+=\alpha_-$; let $\alpha$ denote their common value.
\item For each $i$ between $1$ and $k$, let $\beta^{(i)}=(\beta^{(i)}_1,\dots,\beta^{(i)}_{k_i})$ be the lexicographically largest partition such that $\overline{P}_{(\beta^{(i)}_1,\dots,\beta^{(i)}_{i-1},\beta^{(i)}_{i+1},\dots,\beta^{(i)}_{k_i}),(\beta^{(i)}_i)}$ lies in the sign signature of $V$.
\item There will be some index $i$ for which $\beta_i^{(i)}>\frac{\alpha_i}{2}$.
Let $s$ denote the least such index.
\item Let $\gamma$ denote the lexicographically largest partition such that \\ $\overline{P}_{(\alpha_1,\dots,\alpha_{s-1},\gamma_1,\gamma_2,\dots),(\beta_s^{(s)})}$ lies in the sign signature of $V$.
\item For $1\leq i\leq s$, put $\lambda_i=\beta_i^{(i)}$, and for positive $j$ put $\lambda_{s+j}=\gamma_{s+j}+\beta_{s+j-1}^{(s+j-1)}-\alpha_{s+j-1}$.
\item Finally, put $\mu=\alpha-\lambda$.
Then $V$ is isomorphic to $V_{\{\lambda^*,\mu^*\}}$.
\end{enumerate}
\end{theorem}
Details on our notation can be found in the rest of this section. This theorem is a rephrasing of Proposition \ref{prop:typeD:algorithmRecoverIrreps}, proved at the end of this section.

\subsection{Preliminaries}\label{sec:typeD:preliminaries}
The Weyl group of type $D_n=S_n\ltimes \mathbb Z_2^{b-1}$ can be realized as the index two subgroup of $B_n=S_n \ltimes \mathbb Z_2^n$  generated by $S_n$ and the diagonal matrices $\left\{A=\diag(\pm1, \dots, \pm 1 )| \det A=1 \right\}$. 
In this matrix realization, the first $n-1$ (standard) simple reflections are realized as the permutation matrices $t_1, \dots, t_{n-1}$ of the pairs of indices  $(1,2)$, $(2,3)$, \dots, $(n-1, n)$. 
The last simple reflection is realized as the matrix 
\begin{equation} \label{eq:typeD:lastSimpleReflection}
t_n=\left(\begin{array}{ccccc}1 \\&\ddots \\ &&1\\&&&0&-1\\&&&-1&0 \end{array}\right).
\end{equation}

The representations of Weyl groups in type $D_n$ depend on the parity of $n$. Suppose that $n$ is even. Let 
\[
q=\frac{n}{2} \quad \text{if }n \text{ is even.}
\]
Let $\sigma$ be the element with matrix realization
\begin{equation}\label{eq:typeD:sigma}
\sigma =\left(\begin{array}{c|c}
0& I\\\hline
I&0
\end{array}
\right),
\end{equation}
where $I$ is the $q\times q$ identity matrix.

\subsubsection{Preliminaries: the irreducible representations of $D_n$} In this section we construct  the irreducible  representations of type $D$ using the little groups method of Wigner and Mackey (Section \ref{sec:LittleGroups}). 
To do so we set $G=S_n$ and $H=\mathbb Z_2^{n-1}$. Let $\bar \chi_i$ be the restriction of the $\mathbb Z_2^n$-character given in \eqref{eqZ2characterTypeB} by 
\[
\underbrace{\epsilon_{1}\otimes \dots \otimes \epsilon_1}_{i \text{ copies}} \otimes \underbrace{\epsilon_{-1}\otimes \dots \otimes \epsilon_{-1}}_{n-i \text{ copies}}
\] 
down to $\mathbb Z_2^{n-1} $ under the natural embedding $\mathbb Z_2^{n-1}\hookrightarrow \mathbb Z_2^n$ arising from the embedding $D_n\hookrightarrow B_n$. 
Since $\mathbb Z_2^{n-1} $ is realized as diagonal matrices with determinant $1$ it follows that $\bar \chi_i=\bar \chi_{n-i}$. Let $X$ denote the characters of $\mathbb Z_2^{n-1}$. 
It follows that a set $X/S_n $ of representatives of the orbits of $X$ under the action of $S_n$ is given by 
\[X/S_n=\left\{\bar \chi_i| 0\leq i\leq \frac{n}{2} \right\}.
\]
For $i\neq \frac{n}{2} $, the stabilizer $G_i$ of $\chi_i $ under the action of $S_n$ is $S_i\times S_{n-i}$, and to apply the little groups method we need to induce from the $S_i\times S_{n-i}\ltimes Z_2^{n-1}$-module $U_{\{\lambda, \mu\}}\otimes \bar\chi_i$ where $U_{\{\lambda, \mu\}}$ denotes the module
\[
U_{\{\lambda, \mu\}}= V_\lambda\otimes V_\mu,
\]
with $|\lambda|=i, |\mu|=n-i, i\neq \frac{n}{2}$.

The case $i=q=\frac{n}{2}$ is more complicated. 
The stabilizer of $\bar \chi_q$ is the group $G_q\simeq \mathbb Z_2 \ltimes \left(S_q\times S_q \right) $ realized as the group generated by the element $\sigma$ given in \eqref{eq:typeD:sigma} and $S_q\times S_q$. 
Hence,
\[
G_q=\left\{ \left(\begin{array}{c|c}A& 0 \\\hline 0 & B\end{array} \right) \right\} \cup \left\{ \left(\begin{array}{c|c}0 &A\\\hline B&0 \end{array} \right) \right\}, 
\]
where $A,B$ run over the $q\times q$ permutation matrices. 
For $g\in S_q\times S_q$, denote by $\overline g$ the element $\sigma^{-1}g\sigma $, that is 
\[
\text{if } g=\left(\begin{array}{c|c}A& 0 \\\hline 0 & B\end{array} \right)\quad\text{then set}\quad \overline g=\left(\begin{array}{c|c}B& 0 \\\hline 0 & A\end{array} \right)\quad .
\]
Given a partition $\lambda $, let $p_\lambda$ be one permutation matrix with cycle structure $\lambda$. 
The conjugacy classes of $\mathbb Z_2\ltimes \left(S_q\times S_q\right)$ are then represented by the set of pairwise non-conjugate elements
\[
\left\{ \left(\begin{array}{c|c}p_\lambda& 0 \\\hline 0 & p_\mu\end{array} \right)| \lambda\geq \mu \right\}\cup \left\{ \left(\begin{array}{c|c} 0& I \\\hline p_\lambda&0\end{array} \right) \right\}.
\]

Set $U_{\{\lambda,\mu\}} =\Ind_{S_q\times S_q}^{G_q}\left(V_\lambda\otimes V_\mu\right)$. 
Then $U_{\{\lambda,\mu\}}$ can be understood as:
\[
U_{\{\lambda,\mu \}}= \linspan \left(\{(\id , v\otimes w)\}\cup \{(\sigma , v\otimes w)\} \right),
\]
where $v\in V_\lambda$ and $w\in V_\mu$ run over vector space bases in the corresponding representations. 
An element  $ g_1\times g_2$ lying in  $ S_q\times S_q$ acts by $g_1\times g_2 \cdot (\id, v\otimes w)=(\id, g_1\cdot v\otimes g_2\cdot  w)$ and $g_1\times g_2\cdot (\sigma, v\otimes w)= \left(\sigma, \overline g_2\cdot v \otimes \overline g_1\cdot w\right)$.

For an arbitrary $S_q\times S_q$-module $W$ recall that $W_\sigma$ denotes its  $\sigma$-conjugate module given by \eqref{eqVs-DoubleCosetRep}. 
Direct computation shows that $\left(V_\lambda\otimes V_\mu \right)_\sigma\simeq V_\mu\otimes V_\lambda$. \cite[Proposition 5.1(2), page 64]{FultonHarris:RepresentationTheoryFirstCourse} shows that when $\lambda\neq \mu$ the $G_q$-module $U_{\{\lambda,\mu\}}\simeq U_{\{\mu,\lambda\}}$ is irreducible. 
Two such modules are isomorphic only when the corresponding unordered pairs of defining partitions coincide and this motivates our use of the curly brace notation $\{\} $ to index $U_{\{\lambda,\mu\}}$.

We now proceed with the case $\lambda=\mu$; we temporarily replace the notation $V_\lambda\otimes V_\lambda$ by $V_\lambda'\otimes V_\lambda''$ to emphasize that the first factor $V_\lambda'$ is a representation of the first factor of $S_q\times S_q$ and $V_\lambda''$ - of the second. 
Let $\varphi$ be one invertible map that makes the following diagram commutative for all elements $g$ belonging to the first direct summand of $S_q\times S_q$: 
\[
\xymatrix{
V_\lambda' \ar[r]^{\varphi}& V_\lambda'' \\
V_\lambda' \ar[u]^{g} \ar[r]^{\varphi}& V_\lambda'' \ar[u]_{\overline g} \\
}.
\]
For convenience we introduce the notation
\[
\begin{array}{rcl}
\overline v&=&\varphi ( v) \text{ for all }v\in V_\lambda' \\
\overline w&=&\varphi^{-1}(w) \text{ for all }w\in V_\lambda''.
\end{array}
\]
In our notation, $\overline{(g \cdot v)}= \overline g \cdot \overline v$ and $\overline {\overline v}=v$ whenever the involved expressions are defined.

Let $U_{\{\lambda,\lambda\}}^+$ be the subspace of $\Ind_{S_q\times S_q}^{G_q} V_\lambda \otimes V_\lambda$ given by
\[
U_{\{\lambda,\lambda\}}^+= \linspan\limits_{v\in V_\lambda',w\in  V_\lambda''}\left((\id, v\otimes w)+(\sigma,\overline w\otimes\overline v) \right)
\]
and let $U_{\{\lambda,\lambda\}}^-$ be the subspace given by
\[
U_{\{\lambda,\lambda\}}^-= \linspan\limits_{v\in V_\lambda',w\in  V_\lambda''}\left((\id, v\otimes w)-(\sigma,\overline w\otimes\overline v) \right).
\]
Since for $g$ in the first factor of $S_q\times S_q$ we have that
\begin{equation}\label{eq:typeD:UlambdalambdaPMDefiningActions}
\begin{array}{rcl}
g\cdot \left((\id, v\otimes w)+(\sigma, \overline w\otimes\overline v)\right) &=&(\id, (g\cdot v)\otimes (g\cdot w))+(\sigma,\overline {g\cdot w} \otimes \overline {g\cdot v})\\
\sigma \cdot \left((\id, v\otimes w)+(\sigma, \overline w\otimes\overline v)\right) &=&(\sigma, v\otimes w)+(\id,\overline w\otimes \overline v)\\~\\
g\cdot \left((\id, v\otimes w)-(\sigma, \overline w \otimes\overline v)\right) &=&(\id, (g\cdot v)\otimes (g\cdot w))-(\sigma,\overline {g\cdot w}\otimes\overline {g\cdot v})\\
\sigma\cdot \left((\id, v\otimes w)-( \sigma, \overline w\otimes\overline v)\right) &=&( \sigma, v\otimes w)-(\id,\overline w \otimes \overline v)
\end{array}
\end{equation}
Thus, $U_{\{\lambda,\lambda\}}^+$ and $U_{\{\lambda,\lambda\}}^-$ are in fact $G_q$-(sub)modules. 
By \cite[Proposition 5.1(1), page 64]{FultonHarris:RepresentationTheoryFirstCourse}, these two representations are irreducible and distinct from the remaining $U_{\{\lambda,\mu\}}$ with $\lambda\neq \mu$. 
Finally, the irreducible representations of $\mathbb Z_2\ltimes \left(S_q\times S_q\right)$ are exhausted by the representations of the form $U_{\{\lambda,\lambda\}}^+$, $U_{\{ \lambda, \lambda\}}^-$ and $U_{ \{\lambda,\mu\}}$, as their number matches the number of conjugacy classes of  $\mathbb Z_2 \ltimes \left(S_q\times S_q\right)$.

The little groups method of Wigner and Mackey (Proposition \ref{theoremLittleGroupsMethod}) now implies the following.

\begin{proposition}\label{prop:typeD:irreps}
The set of irreducible representations of $D_n$ consists (up to isomorphism) of the pairwise non-isomorphic representations
\[
\begin{array}{@{}r@{~}c@{~}ll}
V_{\{\lambda, \mu \}}&=&\Ind_{S_i\times S_{n-i} \ltimes Z_2^{n-1}}^{D_n} \left(U_{\{\lambda,\mu \}}\otimes \bar \chi_i \right), & i\neq \frac{n}{2}, |\lambda|=i, |\mu|=n-i\\
V_{\{\lambda, \mu \}}&=&\Ind_{\mathbb Z_2 \ltimes \left(S_q\times S_{q}\right)\ltimes Z_2^{n-1}}^{D_n} \left(U_{\{\lambda,\mu \}}\otimes \bar \chi_i \right), &i=q= \frac{n}{2}, |\lambda|=|\mu|=q, \lambda\neq \mu\\
V_{\{\lambda, \lambda \}}^{\pm}&=&\Ind_{\mathbb Z_2 \ltimes \left(S_q\times S_{q}\right)\ltimes Z_2^{n-1}}^{D_n} \left(U_{\{\lambda,\lambda \}} ^\pm \otimes \bar \chi_i \right), & |\lambda|=q=\frac{n}{2},\\
\end{array}
\]
where the representations in the last two rows are defined only if $n$ is even and furthermore we have that
\[
V_{\{ \lambda,\lambda\}}=\Ind_{\mathbb Z_2 \ltimes \left(S_q\times S_{q}\right)\ltimes Z_2^{n-1} }^{D_n} \left(U_{\{\lambda,\lambda \}}\otimes \bar \chi_i \right)\simeq  V_{\{ \lambda,\lambda\}}^+\oplus  V_{\{ \lambda, \lambda\}}^-.
\]
\end{proposition}


\subsection{Preliminaries: branching from $D_n$ to large subgroups}
Most of this is in \cite{Stembridge:GuideToWeylGroupRepsEmphasisOnBranching}.
\begin{proposition}
\label{prop:typeD:branchingBdowntoD}
The restriction of the $B_n$-representation $V_{\lambda,\mu}$ to $D_n$ is $V_{\{\lambda,\mu\}}$.
\end{proposition}
The sign representation in type $B$ restricts to the sign representation in type $D$ and therefore the sign representation in type $D_n$ is given by:
\begin{equation}\label{eq:typeD:signRepTechnicalNotation}
\signRep{D_n}=V_{\{\pi,\emptyset \}}=V_{\{\emptyset,\pi \}},
\end{equation}
where $\pi=[\underbrace{1,\dots, 1}_{n\text{ times}}]$. From the remarks in \cite[\S 3]{Stembridge:GuideToWeylGroupRepsEmphasisOnBranching} we get the following.
\begin{corollary}
\label{cor:typeD:signDarisesfromTwoBrestrictions}
The $B_n$-modules $V_{\emptyset,\pi}=\signRep{B_n}$ and $V_{\pi,\emptyset} $ are the only two irreducible $B_n$-modules that restrict to $\signRep{D_n}$.
\end{corollary}
Let $s$ be the last simple reflection of $B_n$, i.e., the element with matrix realization
\[
s=\diag (1,\dots, 1, -1).
\]
The element $s$ lies outside of $D_n$ and conjugation by $s$ swaps the two last simple reflections $t_n$ and $t_{n-1}$ and keeps all other simple reflections in place.

For an arbitrary $D_n$-module $U$, let $U_s$ denote its conjugate $D_n$-module given by the same vector space as $U$ but equipped with action $\cdot_s$ given by
\[
g\cdot_s u= \left(sg s^{-1}\right) \cdot u\quad .
\]
The following proposition describes the behavior of the irreducibles $D_n$-modules under conjugation with $s$.
\begin{proposition}(\cite[\S 3, Remarks]{Stembridge:GuideToWeylGroupRepsEmphasisOnBranching})\label{prop:typeD:conjugateModules}
\begin{itemize}
\item $\left(V_{\{\lambda,\mu\}}\right)_s=V_{\{\lambda,\mu\}}$
\item $\left(V_{\{\lambda,\lambda\}}^\pm\right)_s=V_{\{\lambda,\lambda\}}^\mp$
\end{itemize}
\end{proposition}
We note that the proposition above follows from Proposition \ref{prop:typeD:branchingBdowntoD} and general facts about branching a group to an index two subgroup (see \cite[Proposition 5.1]{FultonHarris:RepresentationTheoryFirstCourse}). We will also need the following.

\begin{corollary}~\label{cor:typeD:RepTensorSign}~
\begin{enumerate}
\item \label{cor:typeD:VlambdamutensorSignEqualsItself} $V_{\{ \lambda, \mu\} }\otimes \signRep D_n\simeq  V_{\{ \mu^*, \lambda^*\} } \simeq V_{\{ \lambda^*, \mu^*\} } $
\item \label{cor:typeD:VlambdalambdaPMtensorSignEqualsItself} $V_{\{\lambda, \lambda\}}^\pm \otimes \signRep D_n\simeq  V_{\{ \lambda^*, \lambda^*\} }^\pm $.
\end{enumerate}
\end{corollary}
\begin{proof} \ref{cor:typeD:RepTensorSign}.(\ref{cor:typeD:VlambdamutensorSignEqualsItself}) follows directly from Proposition \ref{prop:typeD:branchingBdowntoD}.

To prove \ref{cor:typeD:RepTensorSign}.\eqref{cor:typeD:VlambdalambdaPMtensorSignEqualsItself} we need to do a little extra work. 
We use the notation from the preceding section.
Let $Q$ be the subgroup from which $V_{\{\lambda, \lambda\}}^+$ is induced, i.e., $Q= \mathbb Z_2\ltimes \left(S_q\times S_q \right) \ltimes \mathbb Z_2^{n-1} $. 
We have that $V_{\{\lambda, \lambda\}}^+ \otimes \signRep D_n$ is contained in the module $W=\Ind_{Q}  \left( U^+_{\{\lambda, \lambda \}}\otimes \signRep {Q} \right) $.
 
Direct computation with \eqref{eq:typeD:UlambdalambdaPMDefiningActions} using the fact that $\signFunction\sigma=1$ ($\sigma$ was defined in \eqref{eq:typeD:sigma}) shows the $Q$-module isomorphism 
\[
U^+_{\{\lambda, \lambda \}}\otimes  \signRep {Q}\simeq U^+_{\{\lambda^*,\lambda^* \}}.
\]
Proposition \ref{prop:typeD:irreps} now implies that $W=V_{\{\lambda^*, \lambda^*\}}^+$, and the fact that the latter is irreducible and contains $V_{\{\lambda,\lambda\} }^+$ shows the two are isormophic. The proof for $V_{\{\lambda, \lambda\}}^-$ is similar and we omit it.
\end{proof}

Next we recall the non-split rules for branching irreducible representations of $D_n$ to $D_{n-k}\times D_k$.

\begin{proposition}\cite[\S 3.A]{Stembridge:GuideToWeylGroupRepsEmphasisOnBranching}\label{prop:typeD:branching:nonSplitOverLargeSubgroups}
Let $\lambda,\mu$ be partitions with $\abs{\lambda}+\abs{\mu}=n$ (we allow $\lambda=\mu$).
Then
\begin{eqnarray}
\displaystyle {V_{\{\lambda, \mu\}}}_{|D_k\times D_{n-k}} &\simeq& \displaystyle  \bigoplus_{\begin{array}{@{}r@{}c@{}l} |\nu| + |\sigma| &=&k\\ |\xi| +|\zeta|&=&n-k\end{array}} c_{\nu,\xi }^\lambda c_{\sigma,\zeta}^{\mu} V_{\{\nu, \sigma\}}\otimes V_{\{\xi, \zeta\}} \label{eqBranchingNonSplitD_nToD_i+D_(n-i)}\\
{V_{\{\lambda, \mu\}}}_{|S_n}& \simeq& \bigoplus_{ |\lambda|+|\mu|=|\nu|} c_{\lambda, \mu}^\nu V_\nu \label{eqBranchingNonSplitD_nToS_n}.
\end{eqnarray}
\end{proposition}
This proposition is a direct corollary of Propositions \ref{prop:typeD:branchingBdowntoD} and \ref{prop:typeB:branchingOverLargeSubgroups} and we omit its proof. 
To have a complete collection of branching laws over $S_n$ and $D_k\times D_{n-k}$ we need additional laws to branch the split case $V_{\{ \lambda, \lambda \}}^{\pm}$.
These rules can again be found in \cite[\S 3.C, D]{Stembridge:GuideToWeylGroupRepsEmphasisOnBranching}, where the rules for restriction over $S_n$ are in turn based on \cite{CarreLeclerc:SplittingSquareOfSchurFunction}. 
For the reader's convenience we include these branching rules here.

\medskip
\begin{proposition}~\label{prop:typeD:branching:DifficultSplit}

$
\begin{array}{lrcl}
\multicolumn{3}{l}{{\text{If $k$ is odd then:}}}
\\
\refstepcounter{equation}
\label{eqBranchingSplitD_nToD_i+D_(n-i)i-odd}
(\theequation)
& 
\displaystyle {V_{\{\lambda, \lambda\}}^{\pm} }_{|D_k\times D_{n-k}} &\simeq& \displaystyle \displaystyle  \bigoplus_{ \begin{array}{@{}r@{}c@{}l} |\nu| + |\sigma| &=&k \\ |\xi| +|\zeta|&=&n-k\\ \nu&>&\sigma \end{array}} \!\!\!\!\!\!\!\!\!\!\!\! c_{\nu,\xi }^\lambda c_{\sigma,\zeta}^{\lambda} V_{\{\nu, \sigma\}}\otimes V_{\{\xi, \zeta\}}.\\
\multicolumn{3}{l}{{\text{If $k$ is even then:}}}
\\
\refstepcounter{equation}
\label{eqBranchingSplitD_nToD_i+D_(n-i)i-even}
(\theequation)
& \displaystyle {V_{\{\lambda, \lambda\}}^{\pm} }_{|D_k\times D_{n-k}} &\simeq &\displaystyle  \bigoplus_{ \begin{array}{@{}r@{}c@{}l} |\nu| + |\sigma| &=&k\\ |\xi| +|\zeta|&=&n-k \\ \nu&\geq&\sigma\\ (\nu, \xi )&\neq& (\sigma, \zeta)\\  \end{array}}\!\!\!\!\!\!\!\!\!\!\!\! c_{\nu,\xi }^\lambda c_{\sigma,\zeta}^{\lambda} V_{\{\nu, \sigma\}}\otimes V_{\{\xi, \zeta\}} \\
&&&\displaystyle \oplus \left(\bigoplus_{|\nu|=k, |\xi|=n-k} d^{\lambda}_{\nu, \xi} V_{\{\nu, \nu\} }^{\pm} \otimes  V_{\{\xi, \xi\} }^{\pm} \right),\\
\multicolumn{4}{l}{\text{where $d^{\lambda}_{\nu, \xi} =\begin{cases} \binom{c^{\lambda}_{\nu, \xi}+1}{2}  &\text{if the number of '-' in }V_{\lambda,\lambda}^{\pm}, V_{\nu}^{\pm}, V_{\xi}^{\pm}\text{ is even} \\
\binom{c^{\lambda}_{\nu, \xi}}{2}  &\text{if the number of '-' in }V_{\lambda,\lambda}^{\pm}, V_{\nu}^{\pm}, V_{\xi}^{\pm}\text{ is odd.} 
\end{cases}$ }}\\~\\~\\
\refstepcounter{equation}
\label{eq:typeD:BranchingSplitV_lambda,lambdaOverS_n}
(\theequation)
&
\displaystyle {V_{\{\lambda, \lambda\}}^{\pm} }_{|S_n}&\simeq& \displaystyle \bigoplus_{\nu} c^{\lambda, \pm}_{\nu}V_{\nu},\\
\multicolumn{4}{l}{\text{where $\begin{array}{rcl} 
c^{\lambda, +}_\nu&=& \dim \Hom_{GL(n)}\left( V^\nu, S^2(V^\lambda) \right)\\
c^{\lambda, -}_\nu&=& \dim \Hom_{GL(n)}\left(  V^\nu, \Lambda^2(V^\lambda) \right)
\end{array} $and  $ \nu$ runs over all partitions.}}
\end{array}
$
\end{proposition}
\noindent  $V^\lambda$ is the representation with highest weight vector $v$ and weight $\lambda$, that is, $\diag(h_1, \dots, h_n)\cdot v= h_1^{\lambda_1}\dots h_n^{\lambda_n}$.

We need a few more technical results.
\begin{lemma}\label{le:typeD:clambdastarmustarequalsclambdamu}
$c^{\lambda^*, \pm}_{\nu^*}=c_\nu^{\lambda,\pm}$.
\end{lemma}
\begin{proof}
We have that
\[
\begin{array}{rcll|l}
{V_{\{\lambda,\lambda \}}^\pm}_{|S_n}&\simeq & \displaystyle \bigoplus_{\nu} c^{\lambda, \pm}_{\nu}V_{\nu}& &\text{use } \eqref{eq:typeD:BranchingSplitV_lambda,lambdaOverS_n}, \text{see \cite[\S 3.C]{Stembridge:GuideToWeylGroupRepsEmphasisOnBranching}}\\
{\left( V_{\{\lambda,\lambda \}}^\pm\otimes \signRep{D_n}\right) }_{|S_n}&\simeq & \displaystyle \bigoplus_{\nu} c^{\lambda, \pm}_{\nu}V_{\nu}\otimes \left(\signRep D_n\right)_{| S_n}\\
{\left( V_{\{\lambda^*,\lambda^* \}}^\pm\right) }_{|S_n}&\simeq & \displaystyle \bigoplus_{\nu} c^{\lambda, \pm}_{\nu}V_{\nu^*}&&\begin{array}{l}
\text{use Lemma \ref{le:typeA:tensoringbySignrepstarsLambda}}\\ 
\text{and Corollary \ref{cor:typeD:RepTensorSign}}\end{array} \\
\displaystyle \bigoplus_{\nu^*} c^{\lambda^*, \pm}_{\nu^*}V_{\nu^*} &\simeq&\displaystyle \bigoplus_{\nu} c^{\lambda, \pm}_{\nu}V_{\nu^*}&&\text{use } \eqref{eq:typeD:BranchingSplitV_lambda,lambdaOverS_n}.
\end{array}
\]
This follows from comparing coefficients in the isomorphism above.
\end{proof}

Let $>$ be the lexicographic order on tuples. 
\begin{lemma}\label{le:typeD:cpmnumbersViaKostkaNumbers}
Let $(p_1,\dots, p_k)$ tuple with non-negative entries. Then 
\[ \begin{array}{rcl}
\displaystyle \sum_{\footnotesize \begin{array}{r@{~}c@{~}l} \alpha+\beta&=& (p_1,\dots, p_k)\\ \alpha&\geq& \beta \end{array} } K_{\lambda, \alpha}K_{ \lambda, \beta}&=& \displaystyle \sum_{\nu} c_\nu^{ \lambda, +} K_{\nu, \beta}\\
\displaystyle \sum_{\footnotesize \begin{array}{r@{~}c@{~}l} \alpha+\beta&=& (p_1,\dots, p_k)\\ \alpha&>& \beta\end{array} } K_{\lambda, \alpha}K_{\lambda,\beta} &=& \displaystyle \sum_{\nu} c_\nu^{\lambda, -} K_{\nu, \beta}\\
\end{array}
\]
where $ \lambda, \nu$ are partitions (in the summations $\nu$ runs over all partitions) and $\beta$ stands for an arbitrary tuple of non-negative integers.
\end{lemma}
Lemma \ref{le:typeD:cpmnumbersViaKostkaNumbers} is a direct corollary to Proposition \ref{prop:KostkaNumberInterpretation} and we omit its proof.


\subsection{Sign representation multiplicities restricted to parabolic subgroups in type $D$}
A parabolic subgroups $\mathcal P$ of $\mathcal D_n$ is generated by a subset of the simple reflections $t_1, \dots, t_n$ fixed in Section \ref{sec:typeD:preliminaries}. 
Relative to the choice of fixed realizations $t_1, \dots,t_{n-1}, t_n$, we separate the parabolic subgroups into three kinds, depending on whether $t_{n-1}$ and $t_n$ are chosen. 

\begin{itemize}
\item $P_{(p_1,\dots, p_k), \emptyset}^+= S_{p_1} \times \dots \times  S_{ p_{k-1} }\times  S_{p_k}$, where each $S_{p_j}$ is realized as a permutation group of consecutive indices (as in \eqref{eq:typeA:parabolicDef}) with $ \sum\limits_i p_i = n $. 
\item $P_{(p_1, \dots, p_k), \emptyset}^-= S_{p_1} \times \dots \times  S_{ p_{k-1} }\times  S_{p_k}$ with $p_k>1$, obtained from $P_{(p_1,\dots, p_k), \emptyset}^+$ by replacing the generator $t_{n-1}$ by $t_n$. 
Here, $p_k>1$ guarantees $t_{n-1}$ is a generator of $P_{(p_1,\dots, p_k), \emptyset }^+$. 
We extend this definition to the case $p_k=1$ by setting $P_{(p_1, \dots, p_{k-1},1), \emptyset }^-= P_{(p_1, \dots, p_{k-1},1), \emptyset }^+$.

\item $\bar P_{(p_1, \dots, p_{k-1}), (p_k) } = S_{p_1}\times \dots\times S_{p_{k-1}}\times D_{p_k}$, where each $S_{p_j}$ is realized as in \eqref{eq:typeA:parabolicDef} and the last summand  $D_{p_k}$ is realized as a Weyl subgroup of type $D$ using the last $ p_k$ indices $n,n-1,\dots, n-p_k+1$.

We over-line the letter $P$ to distinguish the parabolic subgroup $\bar P_{(p_1, \dots, p_{k-1}), (p_k) } $ in type $D$ from the parabolic subgroup $P_{(p_1,\dots, p_{k-1}), (p_k)}$ in type $B$. \end{itemize}
As usual, let $\signRep{P} $ denote the sign representation of a parabolic subgroup $P$. 
Conjugation by the element $s=\diag(1,\dots, 1,-1)$ of $B_n$ swaps $P_{(p_1,\dots, p_k), \emptyset}^+ $ and $P_{(p_1,\dots, p_k), \emptyset }^-$ and maps $\bar P_{(p_1,\dots, p_{k-1}), (p_k)}$ onto itself. Together with
Proposition \ref{prop:typeD:conjugateModules} this implies the following.
\begin{lemma}\label{le:typeD:SignMultConjugateParabolics}
~\\
$\dim\! \Hom_{P_{(p_1,\dots, p_k), \emptyset }^+}\!\!\! \left( \signRep{ P_{(p_1,\dots, p_k), \emptyset}^+}, V_{\{\lambda, \mu\}} \right)\!\! = \!\dim\! \Hom_{P_{(p_1,\dots, p_k), \emptyset}^-}\!\!\! \left( \signRep{ P_{(p_1,\dots, p_k), \emptyset }^-}, V_{ \{\lambda, \mu\}} \right)$\\
$\dim\! \Hom_{P_{ (p_1,\dots, p_k), \emptyset }^+}\!\!\! \left(\signRep{ P_{(p_1, \dots, p_k), \emptyset}^+}, V_{\{\lambda, \lambda\}}^+ \right)\!\! = \!\dim\! \Hom_{P_{(p_1,\dots, p_k), \emptyset}^-}\!\!\! \left( \signRep{ P_{(p_1,\dots, p_k), \emptyset}^-}, V_{\{\lambda, \lambda\}}^- \right)$\\
$\dim\! \Hom_{P_{(p_1,\dots, p_k), \emptyset }^+}\!\!\! \left(\signRep{ P_{(p_1,\dots, p_k), \emptyset}^+}, V_{\{\lambda, \lambda\}}^- \right)\!\! = \!\dim\! \Hom_{P_{(p_1,\dots, p_k), \emptyset}^-}\!\!\! \left(\signRep{ P_{(p_1,\dots, p_k), \emptyset}^-}, V_{\{\lambda, \lambda\}}^+ \right)$\\
$\dim\! \Hom_{\bar P_{(p_1,\dots, p_{k-1}), (p_k)}}\!\!\! \left(\signRep{P_{(p_1,\dots, p_{k-1}), (p_k)}}, V_{\{\lambda, \mu\}} \right)\!\!= $ \\  $ \!\dim\! \Hom_{\bar P_{(p_1,\dots, p_{k-1}),(p_k)}}\!\!\! \left(\signRep{\bar P_{(p_1,\dots, p_{k-1}), (p_k)}}, V_{\{\lambda, \mu\}} \right)$\\
$\dim\! \Hom_{\bar P_{(p_1,\dots, p_{k-1}), (p_k)}}\!\!\! \left(\signRep{P_{(p_1,\dots, p_{k-1}), (p_k)}}, V_{\{\lambda, \lambda\}}^+ \right)\!\!= $ \\  $ \!\dim\! \Hom_{\bar P_{(p_1,\dots, p_{k-1}),(p_k)}}\!\!\! \left( \signRep{\bar P_{(p_1,\dots, p_{k-1}), (p_k)}}, V_{\{\lambda, \lambda\}}^- \right)$.\\
\end{lemma}

In the statement below only one occurrence of a parabolic subgroup will be indicated in full technical notation; the remaining occurrences will be abbreviated by the letter $P$.

\begin{proposition}\label{prop:typeD} 
The multiplicity $\dim \Hom_{P}(\signRep{P}, V)$ of the sign representation of a parabolic subgroup $P$ in an irreducible representation $V$ of $D_n$ is given as follows. 
\begin{enumerate}
\item  \label{prop:typeD:part1:lambdaneqmu} 

\noindent $
\begin{array}{@{}lr@{}c@{}l}
\refstepcounter{equation}
\label{eq:typeD:signMultlambdaneqmuTypeAparabolic}(\theequation)&
\displaystyle \dim\Hom_P \left(\signRep{P_{ (p_1,\dots, p_k), \emptyset}^\pm}, \hspace{0.5cm} V_{\{\lambda, \mu\}} \right) &=&\displaystyle \sum \limits_{\alpha + \beta = (p_1,\dots, p_k)} \!\!\!\!\!\!\!\!\!\! K_{ \lambda^*, \alpha} K_{\mu^*, \beta}\\
\refstepcounter{equation}
\label{eq:typeD:signMultlambdaneqmuTypeDparabolic}(\theequation)&
\displaystyle \dim\Hom_P\left(\signRep{ \bar P_{(p_1,\dots, p_{k-1}), (p_k)}} , V_{ \{ \lambda, \mu\}} \right) &=& \displaystyle \sum \limits_{\substack{ \alpha + \beta = (p_1, \dots, p_{k}) \\ \alpha_k=0 \text{ or } \beta_k=0}} \!\!\!\!\!\!\!\!\!\! K_{ \lambda^*, \alpha}  K_{\mu^*, \beta},
\end{array}
$

\noindent where $\alpha,\beta$ stand for tuples of non-negative integers.
\item  \label{prop:typeD:part2:lambdalambda}
$ \begin{array}{@{}lr@{}c@{}l}
\refstepcounter{equation}
\label{eq:typeD:signMultlambdalambdaPlusPlus}(\theequation)&
\displaystyle \dim\Hom_{P} \left(\signRep{ P_{ (p_1, \dots, p_k), \emptyset}^+}, V_{\{\lambda, \lambda\}}^+ \right)&=&\\
&\displaystyle \dim\Hom_{P}\left( \signRep{P_{ (p_1,\dots, p_k), \emptyset}^-}, V_{\{\lambda, \lambda\}}^- \right)&=& \displaystyle \frac{1}{2}
\!\!\!\!\!\!\!\!\!\! \sum\limits_{{\alpha+ \beta = (p_1,\dots, p_k)}, {\alpha \not =\beta }} \!\!\!\!\!\!\!\!\!\! K_{\lambda^*, \alpha} K_{\lambda^*, \beta} \\
&&&\displaystyle  + \sum\limits_{2\alpha = (p_1,\dots, p_k)}K_{\lambda^*,\alpha}^2, \\
\refstepcounter{equation}
\label{eq:typeD:signMultlambdalambdaMinusPlus}(\theequation)
&\displaystyle \dim\Hom_{P}\left(\signRep{ P_{(p_1,\dots, p_k), \emptyset}^+}, V_{\{\lambda, \lambda\}}^- \right)&=&\\
&\displaystyle \dim\Hom_{P}\left(\signRep{ P_{(p_1,\dots, p_k), \emptyset}^-}, V_{\{\lambda, \lambda\}}^+ \right)&=&
\displaystyle \frac{1}{2}\!\!\!\!\!\!\!\!\!\! \sum \limits_{{\alpha+ \beta = (p_1,\dots, p_k)}, {\alpha \not =\beta }}\!\!\!\!\!\!\!\!\!\! K_{\lambda^*, \alpha} K_{\lambda^*, \beta} \\
\refstepcounter{equation}
\label{eq:typeD:signMultlambdalambdaTypeDparabolic}(\theequation)&
\displaystyle \dim\Hom_{P}\left(\signRep{\bar P_{(p_1,\dots, p_{k-1}), (p_k)}}, V_{\{\lambda, \lambda\}}^\pm \right)&=&
\displaystyle \frac{1}{2} \!\!\!\!\!\!\!\!\!\! \sum \limits_{\substack{ {\alpha+ \beta = (p_1,\dots, p_k)}, {\alpha \not =\beta }\\ \alpha_k=0 \text{ or }\beta_k=0}} \!\!\!\!\!\!\!\!\!\! K_{\lambda^*, \alpha} K_{\lambda^*, \beta},
\end{array}
$

\noindent where $\alpha,\beta$ stand for tuples of non-negative integers.
\end{enumerate}
\end{proposition}
\begin{proof}
The proof of the proposition consists in carefully restricting the branching rules for type $B$ to type $D$. 

\noindent (\ref{prop:typeD:part1:lambdaneqmu}). Suppose $\lambda\neq\mu$. By Proposition \ref{prop:typeD:branchingBdowntoD} the $B_n$-representation $V_{\lambda,\mu} $ restricts down to $V_{\{\lambda, \mu\}}$ over $D_n$. 

We prove first 
\eqref{eq:typeD:signMultlambdaneqmuTypeAparabolic}. For $P=P_{(p_1,\dots, p_k), \emptyset}^+ $ the statement follows from \eqref{eq:typeB:MultSignParabolicNoTypeB} as $P$ is the isomorphic restriction of the parabolic subgroup $ P_{(p_1,\dots, p_k), \emptyset}$ in type $B$. The case $P=P_{(p_1,\dots, p_k), \emptyset}^-$ follows from Lemma \ref{le:typeD:SignMultConjugateParabolics}.

Next we prove 
\eqref{eq:typeD:signMultlambdaneqmuTypeDparabolic}. Consider the type $B$ parabolic subgroup $P_{(p_1, \dots, p_{k-1}), (p_k)} $. We have the chain of embeddings
\[
\underbrace{ \bar P_{(p_1,\dots, p_{k-1}), (p_k)}}_{=P} \hookrightarrow  \underbrace{ P_{(p_1, \dots, p_{k-1}), (p_k)}}_{=Q} \hookrightarrow B_n\quad .
\]
Here, we abbreviate $\bar P_{(p_1,\dots, p_{k-1}), (p_k)}$ by $P$ and $P_{(p_1, \dots, p_{k-1}), (p_k)} $ by $Q$.
We will obtain the multiplicity of $\signRep{P}$ in $V_{\lambda,\mu}$ by first branching down to $Q$.
In view of Corollary \ref{cor:typeD:signDarisesfromTwoBrestrictions}, $\signRep{P}$ is the restriction of two $Q$- representations: $\signRep{Q}$ and the representation $W$ given by
\[ W= \signRep{S_{p_1}}\otimes \dots\otimes \signRep{ S_{p_{k-1}}}\otimes V_{\pi, \emptyset}.
\] 
The multiplicity of $\signRep{Q}$ in $V_{\lambda, \mu}$ was computed in \eqref{eq:typeB:MultSignParabolicWithTypeB} as

\[
\dim \Hom_{ Q} \left(\signRep{Q}, V_{\lambda, \mu} \right)= \sum_{ \alpha+\beta=(p_1,\dots, p_{k})} K_{\lambda^*,\alpha} K_{\mu^*, \beta\oplus \left(p_{k+1}\right)}.
\]
We claim that the multiplicity of $W$ in $V_{\lambda,\mu}$ is given by
\[
\dim \Hom_{ Q} \left(W, V_{\lambda,\mu} \right) = \sum_{ \alpha +\beta=(p_1,\dots, p_{k})} K_{\lambda^*,  \alpha\oplus\left( p_{k+1} \right)} K_{\mu^*, \beta}.
\]
Let $U$ be the 1-dimensional $B_n$-module on which every element of $B_n$ which has even number of $-1$'s in its matrix realization acts by $1$ and the other elements act by $-1$. 
A short consideration shows that $W\otimes U=\signRep{Q}$ and that $V_{\lambda, \mu} \otimes  U=V_{\mu,\lambda}$. 

For an arbitrary finite group $G$, arbitrary $G$-modules $S$, $T$, and a one-dimensional $G$-module $U$ we have a vector space isomorphism $\Hom_{ G} \left(S, T\right) \simeq  \Hom_{ G} \left(S\otimes U, T\otimes U  \right)$: indeed, it is straightforward to see that the map $\varphi\mapsto \bar{\varphi} $ which maps an element $\varphi \in \Hom_{ G} \left(S, T\right)$ to $\bar \varphi\in \Hom_G\left(S\otimes U, T\otimes U  \right) $ via $\bar\varphi (s\otimes u)=\varphi(s)\otimes u$ establishes an isomorphism $\Hom_{ G} \left(S, T\right) \simeq  \Hom_{ G} \left(S\otimes U, T\otimes U  \right)$. 
Therefore $\dim \Hom_{ Q} \left(W, V_{\lambda, \mu} \right)=\dim \Hom_{ Q} \left(W\otimes U, V_{\lambda, \mu}\otimes U \right)$. 

It follows that
\begin{equation}\label{eq:typeD:multWinVlambdamu}
\begin{array}{rcll|l}
\displaystyle \dim \Hom_{ Q} \left(W, V_{\lambda,\mu} \right)&=&\displaystyle \dim \Hom_{ Q} \left(W\otimes U, V_{\lambda,\mu} \otimes U \right)\\
&=&\displaystyle \dim \Hom_Q\left(\signRep{Q}, V_{\mu, \lambda} \right)\\
&=& \displaystyle \sum_{ \alpha +\beta=(p_1, \dots, p_{k})} K_{\lambda^*, ( \alpha,  p_{k +1})} K_{\mu^*, \beta}&& \text{use }\eqref{eq:typeB:MultSignParabolicWithTypeB}. \\
\end{array}
\end{equation}
\eqref{eq:typeD:signMultlambdaneqmuTypeDparabolic} now follows by adding \eqref{eq:typeD:multWinVlambdamu} and \eqref{eq:typeB:MultSignParabolicWithTypeB} and the fact that
\[
\dim \Hom_P(\signRep{P}, V_{\lambda,\mu})=\dim \Hom_Q(\signRep{Q}, V_{\lambda,\mu})+\dim \Hom_Q(W, V_{\lambda,\mu}).
\]

\noindent (\ref{prop:typeD:part2:lambdalambda}). 
Next, we prove \eqref{eq:typeD:signMultlambdalambdaPlusPlus}. 
First, let $P=P^+_{(p_1,\dots, p_k), \emptyset}$. 
We branch first $V_{\lambda,\lambda}$ over $S_n\hookrightarrow D_n$ (using \eqref{eq:typeD:BranchingSplitV_lambda,lambdaOverS_n}) and then branch the $S_n$ representations down to $P$ via the type $A$ sign multiplicity rules (using Proposition \ref{prop:typeA}).
We compute
\[
\begin{array}{@{}r@{}c@{}ll|l}
\displaystyle \dim \Hom_{P}\left(\signRep P, V_{\{\lambda,\lambda\}}^+\right)&=&\displaystyle \sum_{\nu} c_\nu^{\lambda, +} \dim \Hom_P \left( \signRep P,V_\nu \right)&&\text{use } \eqref{eq:typeD:BranchingSplitV_lambda,lambdaOverS_n}\\
&=&\displaystyle \sum_{\nu} c_\nu^{\lambda, +} K_{\nu^*, (p_1, \dots, p_k)} && \text{Proposition }\ref{prop:typeA}\\
&=&\displaystyle \sum_{\nu} c_{\nu^* }^{ \lambda^*, +} K_{\nu^*, (p_1, \dots, p_k)} &&\begin{array}{l} c_{\nu^*}^{\lambda^*, +}=c_{\nu}^{\lambda, +} \\ \text{by Lemma }\ref{le:typeD:clambdastarmustarequalsclambdamu}\end{array}\\
&=&
\displaystyle \sum_{\footnotesize \begin{array}{r@{~}c@{~}l} \alpha+\beta&=& (p_1,\dots, p_k)\\ \alpha&\geq& \beta \end{array} } K_{\lambda, \alpha} K_{\lambda,\beta}
 &&\text{Lemma }\ref{le:typeD:cpmnumbersViaKostkaNumbers},
\end{array}
\]
and this proves \eqref{eq:typeD:signMultlambdalambdaPlusPlus} for $P=P^+_{(p_1,\dots, p_k), \emptyset}$. For $P=P^-_{(p_1,\dots, p_k), \emptyset}$ the statement follows from Lemma \ref{le:typeD:SignMultConjugateParabolics}.

The proof of \eqref{eq:typeD:signMultlambdalambdaMinusPlus} is similar (the only difference being the use of the second part of Lemma \ref{le:typeD:cpmnumbersViaKostkaNumbers} rather than the first) and we omit the details. 

Finally, set $P=\bar P_{\left(p_1, \dots, p_{k-1}\right), \left(p_k\right)}$. Lemma \ref{le:typeD:SignMultConjugateParabolics} implies $\dim \Hom_P \left( \signRep P, V^+_{ \{\lambda,\lambda\}}\right)= \Hom_P\left(\signRep P, V^-_{\{\lambda,\lambda\}}\right)=\frac{1}{2} \Hom_P \left( \signRep P, V_{\{\lambda,\lambda\}} \right)$ and the statement follows from \eqref{eq:typeD:signMultlambdaneqmuTypeDparabolic}.
\end{proof}
\subsection{Recovering irreducible representations from sign signatures in type $D$.}
In this section, we produce both an algorithm for recovering an irreducible representation from its sign signature, and show a direct criterion for distinguishing between different irreducible representations. 


\begin{proposition}\label{prop:typeD:algorithmRecoverIrreps}
Let $V$ be an irreducible representation of $D_n$,  we recall that $\lambda^*$ and $\mu^*$ denote the dual partitions of $\lambda=[\lambda_1, \dots, \lambda_k]$ and $\mu=[\mu_1, \dots, \mu_k]$, and we add zeros if necessary to ensure both $\mu$ and $\lambda$ have the same number of entries. Then $V$ is computed using the following algorithm.

\noindent 1. Compute

\[\begin{array}{rcl}
\alpha_1 &=& \max \left\{p_1 | P^+_{(p_1,\dots),\emptyset} \in \signSignature (V)\right\}\\
\alpha_2 &=& \max \left\{p_2 | P^+_{(\alpha_1, p_2, \dots), \emptyset} \in \signSignature (V)\right\}\\
&\vdots&\\
\alpha_i &=& \max \left\{p_i | P^+_{(\alpha_1,\dots,\alpha_{i-1}, p_i, \dots), \emptyset} \in   \signSignature (V)\right\}\\
&\vdots& .
\end{array}
\]

Stop when $\sum\alpha_i = n$. Set $\alpha=\left( \alpha_1, \dots, \alpha_k \right)$.

\noindent 2. Compute

\[\begin{array}{rcl}
\beta_1 &=& \max \left\{p_1 | P^-_{(p_1,\dots),\emptyset} \in \signSignature (V)\right\}\\
\beta_2 &=& \max \left\{p_2 | P^-_{(\beta_1, p_2, \dots), \emptyset} \in \signSignature (V)\right\}\\
&\vdots&\\
\beta_i &=& \max \left\{p_i | P^-_{(\beta_1,\dots,\beta_{i-1}, p_i, \dots), \emptyset} \in   \signSignature (V)\right\}\\
&\vdots& .
\end{array}
\]

Stop when $\sum\beta_i = n$. Set $\beta=\left( \beta_1, \dots, \beta_k \right)$.

\noindent 3. Suppose $\alpha \neq \beta$. If $\alpha>\beta$ in the lexicographic order then $V$ equals $V^+_{\{\frac{\alpha^*}{2}, \frac{\alpha^*}{2}\} }$. Else $V$ equals $V^-_{ \{ \frac{ \beta^*}{ 2}, \frac{\beta^*}{2}\} }$. Stop.

\noindent 4. Else $\alpha = \beta$. Compute consecutively:
\[
\begin{array}{rcl}
d_1&=& \max \{p_1 |\bar P_{\left(\dots\right), \left(p_1\right)} \in  \signSignature(V) \}\\
d_2&=& \max \{p_2 |\bar P_{\left(\alpha_1,\dots\right), \left(p_2\right)} \in   \signSignature(V)\}\\
&\vdots&\\
d_i&=& \max \{p_i \bar| P_{\left(\alpha_1,\dots, \alpha_{i-1}, \dots\right), \left(p_i\right)} \in   \signSignature(V)\}
\end{array}
\]
Stop computing the $d_s$'s when some $d_s > \frac{\alpha_s}{2}$ for some index $s$; let $s$ denote the first index where that happens.

\noindent 5. Compute consecutively:
\[
\begin{array}{rcl}
f_{s+1} &=& \max \left\{p |\bar P_{\left(\alpha_1,\dots, \alpha_{s-1},p, \dots\right), \left(d_s\right)} \in   \signSignature(V)\right\}\\
d_{s+1} &=&f_{s+1} + d_s - \alpha_s\\
f_{s+2} &=& \max \left\{p |\bar P_{(\alpha_1,\dots, \alpha_{s-1}, f_{s+1},p, \dots), \left(d_s\right)} \in  \signSignature(V) \right\}\\
d_{s+2} &=& f_{s+2} + d_{s +1}- \alpha_{s+1}\\
&\vdots&\\
f_{i} &=& \max \left\{p |\bar P_{(\alpha_1,\dots, \alpha_{s-1}, f_{s+1}, \dots, f_{i-1},p,\dots), \left(d_s\right)} \in  \signSignature(V) \right\}\\
d_{i} &=& f_{i} + d_{i-1}- \alpha_{i-1}\\
&\vdots&
\end{array}
\]
Stop when all $d_i$'s are obtained and then set $e_i = \alpha_i - d_i$ for all $i$.

The representation $V$ equals $V_{\{\delta^*,\epsilon^*\}}$, where $\delta=\left(d_1,d_2, \dots d_k\right)$ and $\epsilon=\left(e_1,e_2, \dots e_k\right)$.

\end{proposition}
\begin{proof}
The multiplicities of the sign representation of $P^{\pm}_{\left(p_1, \dots, p_k \right), \emptyset} $ in $V_{\{\lambda,\mu \}}$, $\lambda\neq \mu$ do not depend on the sign $\pm$. 
Therefore if $\alpha\neq\beta$ then the representation is of the form $ V^\pm_{ \{\lambda, \lambda\}}$. 
If the sign is $+$, then \eqref{eq:typeD:signMultlambdalambdaPlusPlus} shows that the largest index $p_1$ for which $P^+_{\left(p_1, \dots, \right), \emptyset}$ is in the sign signature of $V$ is $2\lambda^*_1$ and \eqref{eq:typeD:signMultlambdalambdaMinusPlus} shows that the largest index $p_1$ for which $P^-_{\left(p_1, \dots, \right), \emptyset}$ is in the sign signature of $V$ is smaller than $p_1$. 
Repeating the same considerations consecutively for the indices $p_2, \dots$ justifies the ``$+$'' part of Step 3. 
The ``$-$'' part of Step 3 is justified in a similar fashion and we omit the details.

Now suppose that $\alpha=\beta$ and therefore the representation is of the form $V_{\{\lambda, \mu\}}$, $\mu\neq \lambda$.
The multiplicity formula \eqref{eq:typeD:signMultlambdaneqmuTypeAparabolic} shows that for each $i$ we have that $\alpha_i = \beta_i$ equals the sum $\lambda^*_i + \mu^*_i$.
Furthermore \eqref{eq:typeD:signMultlambdaneqmuTypeDparabolic} shows that the numbers $d_i$ computed in Step 4 measure $\max\{\mu^*_i, \lambda^*_i \}$ for $i\leq s$. 
The definition of $s$ in Step 4 now implies that $s$ is the smallest index for which $\mu^*_s\neq \lambda^*_s$. 
Suppose that $\mu^*>\lambda^*$ in the lexicographic order, i.e., $\mu^*_s=d_s$. 
At the start of Step 5, the information known about $\mu^*, \lambda^*$ is summarized in the following figure.

\begin{equation} \label{eq:typeD:pairOfShapesForAlgorithm}
\xygraph {
!{<0cm, 0cm>;<0cm, 0.5cm>:<0.6cm, 0cm>: :}
!{(-2.5, 4)}{\mu^*}
!{(2.5, 4)}{\lambda^*}
!{(-4, 3)} *=(1,1){}*\frm{-}
!{(-3, 3)} *=(1,1){}*\frm{-}
!{(-2, 3)} *=(1,1){}*\frm{-}
!{(-1, 3)} *=(1,1){}*\frm{-}
!{(4, 3)} *=(1,1){}*\frm{-}
!{(3, 3)} *=(1,1){}*\frm{-}
!{(2, 3)} *=(1,1){}*\frm{-}
!{(1, 3)} *=(1,1){}*\frm{-}
!{(-2.5,2.2 ) } *+{\vdots}
!{(2.5,2.2 ) } *+{\vdots}
!{(-4, 1)} *=(1,1){}*\frm{-}
!{(-3, 1)} *=(1,1){}*\frm{-}
!{(-2, 1)} *=(1,1){}*\frm{-}
!{(3,  1)} *=(1,1){}*\frm{-}
!{(2,  1)} *=(1,1){}*\frm{-}
!{(1,  1)} *=(1,1){}*\frm{-}
!{(-5, 2)} *=(0,3){}*\frm{\{}
[lll]{\frac{\alpha_i}{2}=\mu^*_i=\lambda^*_i }
!{(-4, 0)} *=(1,1){}*\frm{-}
!{(-5,0)}{}[ll]{\mu^*_s>\lambda^*_s}
!{(-3, 0)} *=(1,1){}*\frm{-}
!{(-2, 0)} *=(1,1){}*\frm{-}
!{(2,  0)} *=(1,1){}*\frm{-}
!{(1,  0)} *=(1,1){}*\frm{-}
!{(-2.5,-0.8 ) } *+{\vdots}
!{(2.5,-0.8 ) } *+{\vdots}
}
\end{equation}
Consider the number $f_{s+1} = \max \left\{p |\bar P_{\left(\alpha_1,\dots, \alpha_{s-1},p, \dots\right), \left(d_s\right)} \in   \signSignature(V)\right\}$. 
Let $P=P_{\left(\alpha_1,\dots, \alpha_{s-1},p, \dots\right), \left(d_s\right)}$ be a parabolic  at the time  $f_{s+1}$ is being computed.
In order for $P$ to be in the sign signature of $V$ we need to have a non-zero summand $K_{\lambda^*, \xi'} K_{\mu^*, \psi'}$ in \eqref{eq:typeD:signMultlambdaneqmuTypeDparabolic}. 
For such a summand, define $\xi$ as the tuple obtained from $\xi'$ by shifting all entries with indices larger than $s$ one unit to the right and then placing the last index $\xi'_k$ in the $s^{th}$ position. 
In other words, $\xi=\left(\xi_1, \dots\right)=\left(\xi'_1,\dots, \xi'_{s-1}, \xi'_k, \xi'_s, \dots \right)$. Define $\psi$ by shuffling the entries of $\psi'$ in a similar fashion. By Corollary \ref{cor:KostkaNumberInvarintUnderSn} $K_{\lambda^*, \xi'} K_{\mu^*, \psi'}=K_{\lambda^*, \xi} K_{\mu^*, \psi}$. The indexing conditions on the sum \eqref{eq:typeD:signMultlambdaneqmuTypeDparabolic} imply that either $\xi_s=d_s$ and $\psi_s=0$ or $\xi_s=0$ and $\psi_s=d_s$. 
The equality $\psi_s=d_s$ is impossible due to $d_s=\mu_s^*>\lambda^*_s$.  
This implies that $K_{\lambda^*, \xi} K_{\mu^*, \psi}$ is non-zero only if the pair of shapes \eqref{eq:typeD:pairOfShapesForAlgorithm} is filled to a tableau as illustrated below.
\begin{equation*}
\xygraph {
!{<0cm, 0cm>;<0cm, 0cm>:<0.6cm, 0cm>: :}
!{(-2.5, 4)}{\mu^*}
!{(2.5, 4)}{\lambda^*}
!{(-4, 3)} *=(1,1){\text{\tiny $1$}}*\frm{-}
!{(-3, 3)} *=(1,1){\text{\tiny $1$}}*\frm{-}
!{(-2, 3)} *=(1,1){\text{\tiny $1$}}*\frm{-}
!{(-1, 3)} *=(1,1){\text{\tiny $1$}}*\frm{-}
!{(4, 3)} *=(1,1){\text{\tiny $1$}}*\frm{-}
!{(3, 3)} *=(1,1){\text{\tiny $1$}}*\frm{-}
!{(2, 3)} *=(1,1){\text{\tiny $1$}}*\frm{-}
!{(1, 3)} *=(1,1){\text{\tiny $1$}}*\frm{-}
!{(-2.5,2.2 ) } *+{\vdots}
!{(2.5,2.2 ) } *+{\vdots}
!{(-4, 1)} *=(1,1){\text{\tiny $s-1$}}*\frm{-}
!{(-3, 1)} *=(1,1){\text{\tiny $s-1$}}*\frm{-}
!{(-2, 1)} *=(1,1){\text{\tiny $s-1$}}*\frm{-}
!{(3,  1)} *=(1,1){\text{\tiny $s-1$}}*\frm{-}
!{(2,  1)} *=(1,1){\text{\tiny $s-1$}}*\frm{-}
!{(1,  1)} *=(1,1){\text{\tiny $s-1$}}*\frm{-}
!{(-5, 2)} *=(0,3){}*\frm{\{}
[lll]{\frac{\alpha_i}{2}=\mu^*_i=\lambda^*_i }
!{(-4, 0)} *=(1,1){\text{\tiny $s$}}*\frm{-}
!{(-5,0)}{}[ll]{\mu^*_s>\lambda^*_s}
!{(-3, 0)} *=(1,1){\text{\tiny $s$}}*\frm{-}
!{(-2, 0)} *=(1,1){\text{\tiny $s$}}*\frm{-}
!{(2,  0)} *=(1,1){\text{\tiny$s+1$}}*\frm{-}
!{(1,  0)} *=(1,1){\text{\tiny$s+1$}}*\frm{-}
!{(-4, -1)} *=(1,1){\text{\tiny $s+1$}}*\frm{-}
!{(-3, -1)} *=(1,1){\text{\tiny $s+1$}}*\frm{-}
!{(-3.5,-1.8 ) } *+{\vdots}
!{(1.5,-0.8 ) } *+{\vdots}
}
\end{equation*}
$f_{s+1}$ equals the number of $s+1$'s in the tableau above and so $f_{s+1}=\mu^*_{s+1}+\lambda^*_s$. On the other hand we can compute $\lambda^*_s$ from $\mu^*_s+\lambda^*_s=\alpha_s$, and so we get that $d_{s+1}= \mu^*_{s+1} = f_{s+1} - (\alpha_s-\mu^*_s)=f_{s+1} - \alpha_s +d_s$. 

Once we have computed $d_{s+1} = \mu^*_{ s+1}$, we can carry out a similar procedure to compute $d_{s+2}$ and so on. This justifies Step 5 of the algorithm and completes the proof.
\end{proof}

The following proposition is the type $D$ analog of Proposition \ref{prop:typeB:differentIrrepsDifferentSignSignatures}.
\begin{proposition}\label{prop:typeD:differentIrrepsDifferentSignSignatures}
The irreducible representations of $D_n$ have different sign signatures. 
\end{proposition}
\begin{proof}
This proof can be regarded as an analysis of the algorithm given in Proposition \ref{prop:typeD:algorithmRecoverIrreps}.

Case 1. Consider the representations $V_{\{\lambda, \mu\}}$ and $V_{\{\nu, \sigma \}} $, $\lambda^*\neq \mu^*$, $\nu^*\neq \sigma^*$, $\{ \lambda^*, \mu^* \} \neq \{\nu^*,\sigma^* \}$. Assume $\mu^*>\lambda^*$, $\sigma^* > \nu^*$, where $>$ is the lexicographic order on partitions. Recall \eqref{eq:typeD:signMultlambdaneqmuTypeDparabolic}:
\[
\displaystyle \dim\Hom_P\left(\signRep{ \bar P_{(p_1,\dots, p_{k-1}), (p_k)}} , V_{\{\lambda, \mu\}} \right) = \displaystyle \sum \limits_{\substack{ \alpha + \beta = (p_1,\dots, p_{k}) \\ \alpha_k=0 \text{ or } \beta_k=0}} \!\!\!\!\!\!\!\!\!\! K_{ \lambda^*, \alpha}  K_{\mu^*, \beta}.
\]

Case 1.1. Suppose  $\lambda^*+ \mu^*\neq \nu^*+\sigma^*$. 
Since the right hand sides of formulas \eqref{eq:typeD:signMultlambdaneqmuTypeAparabolic} and \eqref{eq:typeB:MultSignParabolicNoTypeB} are the same, an argument similar to Case 1 in the proof of Proposition \ref{prop:typeB:differentIrrepsDifferentSignSignatures} proves the two representations have different sign signatures; we omit the details.

Case 1.2. Suppose $\lambda^*+\mu^*=\nu^*+\sigma^*$. 
Let $t$ be the smallest index for which $\lambda^*_t\neq \mu^*_t $ and lest $s$ be the smallest index for which $\nu^*_s\neq \sigma^*_s$.

Case 1.2.1. $s\neq t$, without loss suppose $t<s$. 
Therefore $\sigma^*_t=\nu^*_t=\frac{1}{2}\left(\mu^*_t+\lambda^*_t\right)$; together with $\mu^*_t>\lambda^*_t$ this implies that $\mu^*_t>\nu^*_t= \sigma^*_t$. 
Set $P=$ \\$\bar P_{\left(\mu^*_1 + \lambda^*_1,\dots, \mu^*_{t-1}+\lambda^*_{t-1}, \mu^*_{t+1}+\lambda^*_t,\dots, \mu_k^*+\lambda^*_{k-1}, \lambda^*_k \right), \left(\mu^*_t \right)}$.
We claim $P$ is in the sign signature of $V_{\{\lambda, \mu\}} $ but not in the sign signature of $V_{\{\nu,\sigma\}} $. Indeed, if we choose $\alpha=(\mu^*_2, \dots, \mu^*_k, \mu_1^*)$ and $\beta=\lambda$, then the summand $K_{\lambda^*,\alpha}K_{\mu^*, \beta}$ in \eqref{eq:typeD:signMultlambdaneqmuTypeDparabolic} equals one (see Corollary \ref{cor:KostkaNumberInvarintUnderSn}) and so $P$ is in the sign signature of $V_{\{\lambda, \mu\}} $. Since either $\alpha_k=\mu_t^*$ or $\beta_k=\mu^*_t$ and $\mu^*_t>\sigma_t^*=\nu^*_t$, it follows that either $K_{\nu^*, \alpha} $ or $K_{\sigma^*, \alpha} $ is zero; in particular the sum \eqref{eq:typeD:signMultlambdaneqmuTypeDparabolic} is zero and our claim follows.

Case 1.2.2. $s= t$. 

Case 1.2.2.1. Suppose $\mu^*_t>\sigma^*_t$. Since $\sigma^*>\nu^*$ and $\sigma^*_i=\nu^*_i $ for $i<t$ it follows that $\mu^*_t>\nu^*_t$. Now define $P$ as in Case 1.2.1; the arguments of Case 1.2.1 apply without change and it follows that $P$ is in the sign signature of $V_{\{\lambda,\mu \}}$ but not in the sign signature of $V_{\{\nu, \sigma \}}$.

Case 1.2.2.2. Suppose $\mu^*_t=\sigma^*_t$. 
Without loss assume that $\mu^*>\sigma^*$ and let $ r$ be the smallest index for which $\mu^*_r\neq \sigma^*_r$. It follows that $\mu^*_r>\sigma^*_r$. 
Set $P=\bar P_{\left(\mu^*_1 + \lambda^*_1,\dots, \mu^*_{t-1} + \lambda^*_{ t-1}, \mu^*_{t+1}+\lambda^*_t,\dots, \mu_k^* + \lambda^*_{ k-1}, \lambda^*_k \right), \left(\mu^*_t \right)}$. 
Just as in Case 1.2.1, we get that $P$ is in the sign signature of $V_{\{\lambda, \mu\}} $; we claim it does not lie in the sign signature of $V_{\{\nu,\sigma\}} $.
Suppose on the contrary a summand $K_{\sigma^*, \alpha}K_{\nu^*, \beta}$ in \eqref{eq:typeD:signMultlambdaneqmuTypeDparabolic} is non-zero. 
Reorder the entries of $\alpha$ by placing the last entry, $\alpha_k$, in the $t^{th}$ position and shifting the entries $\alpha_{t},\dots$ one unit to the right; call the new weight $\alpha'$. 
In other words, if $\alpha=\left(\alpha_1,\dots,\alpha_{t-1}, \alpha_t, \dots, \alpha_{k-1}, \alpha_k\right)$ then set $\alpha'=\left(\alpha_1,\dots, \alpha_{t-1}, \alpha_k, \alpha_t, \dots, \alpha_{k-1}\right) $. 
Define $\beta'$ by shuffling the entries of $\beta$ in a similar fashion. Corollary \ref{cor:KostkaNumberInvarintUnderSn} implies that $K_{\sigma^*, \alpha}K_{\nu^*, \beta}=K_{\sigma^*, \alpha'}K_{\nu^*, \beta'}$, where the conditions on $\alpha$ and $\beta$ correspond to the requirement that $\{\alpha'_t,\beta'_t\}=\{\mu^*_t, 0\}$.

Since the first $t-1$ entries of $\alpha'$ and $\beta'$ sum to the maximal possible value, it follows that $ \alpha'_1=\sigma^*_1, \dots, \alpha'_{t-1}=\sigma^*_{t-1}$ and  $ \beta'_1=\nu^*_1, \dots, \beta'_{t-1}=\nu^*_{t-1}$. 
It is not possible that $\beta'_t=\mu^*_t$ as that would mean $\beta'_t=\mu^*_t>\nu^*_{t}$ which would make $K_{\nu^*, \beta'}$ equal to zero. 
It follows that $\alpha'_t=\mu^*_t$ and $\beta'_t=0$.
It follows that $\alpha'_{t+1}\leq \sigma^*_{t+1},\beta'_{t+1}\leq \nu^*_{t} $ and $\alpha'_{t+1}+\beta'_{t+1}=\mu^*_{t+1}+\lambda^*_t$. 
So long as $t+1<r$ we have that $\mu^*_{t+1}=\sigma^*_{t+1}$ and $\lambda^*_t=\nu^*_t$. 
It follows that $\alpha'_{t+1}=\sigma^*_{t+1} $ and $\beta'_{t}=\nu^*_t$. 
Now the same argumentation applies consecutively to show that, so long as $t<s<r-1$, we have that $\alpha'_{s+1}=\sigma^*_{s+1}$ and $\beta'_{s}=\nu^*_s$. 
Finally we get that $\alpha'_{r}\leq \sigma^*_{r}<\mu^*_r$ and $\beta'_{r-1}\leq \nu^*_{r-1}=\lambda^*_{r-1}$. 
At the same time $\alpha'_{r}+\beta'_{r-1} $ must equal $\mu^*_r+\lambda^*_{r-1}$; contradiction. 
This proves our claim and completes the investigation of Case 1.2.2.2.

Case 2. Consider the representations $V_{\{\lambda, \mu\}}$ and $V^\pm_{\{\sigma,\sigma  \}} $, $\lambda^*\neq \mu^*$. 

Case 2.1 Suppose $\lambda^+\mu^*\neq 2\sigma^*$. This case is now similar to Case 1.1 and we omit the details.

Case 2.2. Suppose $\lambda^*+\mu^*= 2\sigma^*$. Now let $t$ be the smallest index for which $\lambda^*_t\neq \mu^*_t $. The rest of this case is treated similarly to Case 1.2.1 and we omit the details.

Case 3. Consider the representations $V^\pm_{\{\lambda, \lambda\}}$ and $V^\pm_{\{\sigma,\sigma  \}} $, $\lambda^*\neq \sigma^*$. This implies that $2\lambda^*\neq 2\sigma^*$ and so a consideration similar to Case 1.1 shows that the two representations have a different sign signature. 

Case 4. Consider the representations $V^+_{\{\lambda, \lambda\}}$ and $V^-_{\{\lambda,\lambda\}} $. 
We claim that $P^+_{2\lambda^*, \emptyset}$ is in the sign signature of $V^+_{\{\lambda, \lambda\}}$ but not of $V^-_{\{\lambda, \lambda\}}$. 
Indeed, for this choice of parabolic, \eqref{eq:typeD:signMultlambdalambdaPlusPlus} equals one as only the summand $K_{\lambda^*, \lambda^*}^2$ (obtained for $\alpha=\beta=\lambda^*$) is non-zero. 
On the other hand all summands in \eqref{eq:typeD:signMultlambdalambdaMinusPlus} are zero and this proves our claim.
\end{proof}

\section{Special irreducible Representations and their $\tau$-invariants}
In this section we establish a one to one correspondence between special irreducible representations of $W$ and the list of $\tau$-invariants used by the Atlas software to identify irreducible Harish-Chandra modules.

Let $G_{\mathbb R}$ be a real form of $G$ defined over $\mathbb R$.  Denote by $\hat G_\lambda$ its classes of irreducible admissible representations with infinitesimal regular integral character $\lambda$. Then there is a standard procedure to equip $\hat G_\lambda$ with a set of $\mathbf W$-graphs (Definition \ref{defn:wgraph}) such that each connected component of each graph determines a representation of $W$ associated to a special nilpotent orbit (in the sense of Lusztig, see \cite {Carter:FiniteGroupsLieType}). Each of these $\mathbf W$-graphs is called a \emph{block} and its connected components are called \emph{cells}. Each vertex of a cell corresponds to an irreducible Harish-Chandra module. 
Among all blocks given by this construction, the largest by size (``the big block'') is guaranteed to have each special nilpotent orbit of ${\mathfrak g}$ attached to at least one of its cells.

We recall from Definition \ref{defn:wgraph} that the $\tau$-invariant of a vertex $v$ is the set $\tau(v)$.

\begin{definition} The $\tau$-signature of a cell (or of its corresponding special nilpotent orbit)  is the collection of the $\tau$-invariants of its vertices.
\end{definition}

\begin{theorem}

Let $V$ be a representation of $W$  from a cell $\mathfrak C$ of Harish-Chandra modules corresponding to a special nilpotent orbit $\mathcal O$ and a special $W$-representation $\sigma(\mathcal O)$. Then the $\tau$-signature of $\mathfrak C$, or equivalently that of $V$, is determined by $\sigma(\mathcal O)$.
\end{theorem}

\begin{proof}
In the classical cases, Theorem \ref{thm:Classical} gives a correspondence between $\tau$-invariants and irreducible (and in particular, special) representations of $W$ and our statement follows from Corollaries \ref{cor:inv} and \ref{cor:sign}. 

In the exceptional cases we use the Atlas package. We compute the $\tau$-signature of each cell of the big block. In each case the number of distinct $\tau$-signatures is equal to the number of special nilpotent orbits in ${\mathfrak g}$; that is 46, 35, 17, 11, 3 for $E_8$, $E_7$, $E_6$, $F_4$, and $G_2$ respectively. The proof now follows from the Springer correspondence.
\end{proof}

\begin{example} The Atlas command \emph{wcell} produces the following cell decomposition for the big block of representations of the split real form of $G_2$.

\bigskip
\begin{verbatim}
------------------------------------------
// Cells and their vertices.
#0={0}
#1={1,3,6,7,11}
#2={2,4,5,8,10}
#3={9}

// Induced graph on cells.
#0:.
#1:->#0.
#2:->#0.
#3:->#1,#2.

// Individual cells.
// cell #0:
0[0]: {}

// cell #1:
0[1]: {2} --> 1
1[3]: {1} --> 0,2
2[6]: {2} --> 1,3
3[7]: {1} --> 2,4
4[11]: {2} --> 3

// cell #2:
0[2]: {1} --> 1
1[4]: {2} --> 0,2
2[5]: {1} --> 1,3
3[8]: {2} --> 2,4
4[10]: {1} --> 3

// cell #3:
0[9]: {1,2}
--------------------------------------
\end{verbatim}

\bigskip

In this case we find four  cells 0, 1, 2, and 3 with $\tau$-signature respectively $\{ \}$, $\{\{1\},\{2\}\}$, $\{\{1\},\{2\}\}$,  and $\{\{1,2\}\}$. Hence the number of distinct $\tau$-signatures is the same of that of special complex  nilpotent orbits. Using Carter's label we can identify the corresponding three special nilpotent orbits $G_2$, $G_2(a_1)$ and $0$. These labels are taken from the book of Collingwood and McGovern \cite{CollingWoodMacgovern:NilpotentOrbitsSSLieAlg}. \end{example}

\section{Sign signatures and a generalization of Theorem \ref{thm:Classical}} \label{sec:Exceptional}
This section defines \emph{extended sign signatures of representations} and shows that they distinguish the irreducible representations of all simple Weyl groups.

As usual, we denote a simple complex Lie algebra by $\mathfrak g$. Fix a Cartan subalgebra $\mathfrak h\subset \mathfrak g$ and a corresponding root system $\Delta$. Finally, select a set of positive roots in $\Delta$. We recall that the group generated by the reflections with respect to the roots in $\Delta$ is the Weyl group of $\mathfrak g$ and that all Weyl groups arise in this fashion. For the rest of this section, we assume given data as above such that the root system $\Delta$ generates the Weyl group $W$. We also assume that the Lie algebra  $\mathfrak g$ is simple, hence so is $W$.

Given a root $\alpha$, we denote the reflection with respect to $\alpha$ by $s_\alpha$. 
Recall that a simple basis of the root system determines the (vertices of) the Dynkin diagram of ${\mathfrak g}$. 
For a simple Weyl group $W$, adjoining the lowest root of ${\mathfrak g}$ to the Dynkin diagram of ${\mathfrak g}$ gives the extended Dynkin diagram of ${\mathfrak g}$. 
We will need the following standard definitions.
\begin{definition}
\label{def:pseudoSignSignature}  
Let $W$ be the Weyl group of a simple ${\mathfrak g}$ and $ H\subset W$ be a subgroup.
\begin{enumerate}
\item We say that $H$ is a \emph{root reflection subgroup} of $ W$ if it is generated by a set of reflections $s_{\alpha_1}, \dots, s_{\alpha_l}$ with respect to roots $\alpha_1,\dots, \alpha_l $ of ${\mathfrak g}$.
\item We say that $H $ is a \emph{parabolic subgroup} if is it is a root reflection subgroup for which we can choose the roots $\alpha_1, \dots, \alpha_l$ to be a subset of some simple basis of the root system.
\item We say that $H$ is a \emph{pseudo-parabolic subgroup} if it is a root reflection subgroup for which we can choose the roots $\alpha_1, \dots, \alpha_l$ to be a subset of the vertices of an extended Dynkin diagram of ${\mathfrak g}$. 
\end{enumerate}
\end{definition}


We also need the notion of an additively closed root subsystem\footnote{depending on the author, the shorter term ``root subsystem'' is used to denote a set satisfying either just Definition \ref{def:additivelyClosedRootSubsystem}.(1) or both Definition \ref{def:additivelyClosedRootSubsystem}.(1) and (2). 
To avoid confusion we use the verbose (but certainly unambiguous) terminology ``additively closed root subsystem''.}. 
\begin{definition}\label{def:additivelyClosedRootSubsystem}
Let $\Phi$ be a root system and let $\Psi\subset \Phi$. We say that $\Psi$ is an additively closed root subsystem of $\Phi$ if:
\begin{enumerate}
\item $\Psi$ is a root system (with respect to the scalar product defining $\Phi$) and
\item $\Psi$ is additively closed, that is, if $\alpha, \beta\in \Psi$ and $\alpha+\beta\in \Phi$ then it follows that  $\alpha+\beta\in \Psi$.
\end{enumerate}
\end{definition}
\begin{remark}
In types $A, D, E$  every subset of a root system which satisfies (1) automatically satisfies (2). This fails in types $B$, $C$, $F_4$ and $G_2$: for example, in type $B_2$, the set $\Psi=\{\pm\varepsilon_1, \pm \varepsilon_2\}$ satisfies (1) but not (2).
\end{remark}
The additively closed root subsystems were classified, among other places, in \cite{Dynkin:Semisimple}.

For an additively closed root subsytem $\Delta$, let $H(\Delta)$ denote the subgroup of $W$ generated by the reflections with respect to the roots in $\Delta$. 
Define the sign representation $\signRep{H(\Delta)}$ as the one-dimensional representation of $H(\Delta)$ on which each reflection $s_\alpha$, $\alpha\in \Delta$ acts by $-1$. 
\begin{definition}
Let $ W$ be a Weyl group and $V$ a representation of $ W$.
\begin{enumerate}
\item We define the \emph{extended sign signature} of $V$ to be the set of subgroups $H$ generated by reflections with respect to additively closed root subsystems such that the restriction of $V$ down to $H$ contains a copy of the sign representation of $H$.
\item We define the \emph{pseudo sign signature} of $V$ to be the  restriction of the extended sign signature of $V$ to the pseudo-parabolic subgroups. 
\end{enumerate}
\end{definition}
The sign signature of $V$ is the restriction of the pseudo sign signature of $V$ to the parabolic subgroups.

Let $\omega$ be a root system automorphism of $W$, i.e., a linear map that preserves the roots and their scalar products. The action of $W$ on the root system gives a tool to compute most of these automorphisms; we call such automorphisms inner. To get all root system automorphisms, one extends the group of inner automorphisms by the group of Dynkin diagram automorphisms. The Dynkin diagram automorphisms has two elements for $A_n$,$n\geq 2$, $D_n$, $n \geq 5$ and $E_6$, and is the permutation group of three elements for $D_4$. All other simple types have only trivial Dynkin diagram automorphisms. We call root system automorphisms that are not inner outer.

The map $\omega$ induces a natural group automorphism $\Omega:W\to W$. The representations $V$ and $V_\Omega $ may fail to be isomorphic if $\Omega$ permutes the conjugacy classes of $W$, and indeed that happens in the case of the type $D_n$ ($n$-even) representations $V_{\{\lambda, \lambda\}}^\pm$ with respect to the outer automorphism of $D_n$, however this is the only possible example. Given a representation $V$ of $W$, define $V_{\Omega}$ as the same vector space but equipped with action $\cdot_{\Omega}$ given by $g\cdot_{\Omega} (x)=\Omega(g) \cdot v$.

As a side note we point out that not all group automorphisms of $W$ 
arise from root system automorphisms, for example automorphisms of $W$ arising from swapping long with short roots in $B_2$ and $F_4$ and automorphisms of the permutation group $S_6$ (type $A_5$) that swap permutations of type $(\bullet\bullet)$ with permutations of type $(\bullet\bullet)(\bullet\bullet)(\bullet\bullet)$.
\begin{lemma}\label{le:ConjugacyClassesAndOuterAutos}~
Let $\omega$ be a root system automorphism as defined in the preceding remarks.
\begin{enumerate}
\item $\Omega$ fails to preserve the conjugacy classes of $W$ if and only if $\omega$ is an outer root automorphism of order $2$, $W=D_{n}$ and $n$ is even.
\item If $W\neq D_n$, $n$-even and $\Delta$ is an additively closed root subsystem then $\dim\Hom_{H(\Delta)} \left(\signRep{H(\Delta)}, V\right)=
\dim\Hom_{H(\omega(\Delta))} \left(\signRep{H(\omega(\Delta))}, V\right)$.
\end{enumerate}
\end{lemma}
We will apply this lemma only for $W=G_2, F_4, E_6, E_7, E_8$, and so we give a proof only in these cases. 

\begin{proof}[Proof for $W=G_2, F_4, E_6, E_7, E_8$]
First, suppose we have already proved (1).

$\Omega$ is an automorphism and therefore maps the conjugacy classes of $W$ onto one another.
Since $W$ is exceptional, by (1) we have that $\Omega$ maps each conjugacy class of $W$ onto itself.
This means that each element $g\in W$ has the same eigenvalues on $V_{\Omega}$. 
It follows that $V_{\Omega}\simeq V$ which implies (2). 

We are now proving (1). 
If $\omega$ is an inner root system automorphism then $\Omega$ is given by conjugation and the statement is immediate. 
This implies the statement for $G_2, F_4, E_7, E_8$ as their root systems have no outer root automorphisms. 

Up to conjugation, the root system of $E_6$ has only one outer automorphism $\omega$. 
$E_6$ has 51840 elements, split into 25 conjugacy classes. 
We represent each element as a word of minimal length in the simple generators.
We compute each conjugacy class and select a representative element from each class.
The action of $\omega$ on every representative is then computed by relabeling the simple generators.
The final result is that $\omega$ preserves each conjugacy class of $E_6$.
The computation is carried out using our software \cite{Folz-DonahueMilev:vpf}.
\end{proof}

\begin{remark}
In a subsequent remark by D. Vogan, it was pointed out to us that an alternative proof in the $E_6$ can be obtained by inspecting the table of irreducible characters of $E_6$ and the fact that root systems preserve dimension and $b_{\chi}$, see \cite[page 413]{GeckPfeiffer:CharactersCoxeterGroups}
\end{remark}

The following is a generalization of Theorem \ref{thm:Classical}.
\begin{theorem}
\label{thm:classicalAndExceptional} 
Let ${\mathfrak g}$ be a simple complex Lie algebra with Weyl group $ W$.
\begin{enumerate}
\item (Theorem \ref{thm:Classical})If ${\mathfrak g}$ is classical, there is a one-to one correspondence between the irreducible representations of $ W$ and their sign signatures.
\item If ${\mathfrak g} $ is of type $F_4$ or $E_6$ (or classical), there is a one-to-one correspondence between the irreducible representations of $ W$ and their pseudo sign signatures.
\item If ${\mathfrak g}$ is arbitrary simple, there is a one-to one correspondence between the irreducible representations of $ W$ and their extended sign signatures.
\end{enumerate}
\end{theorem}

Part (1) was already proved; we handle the exceptional cases with the help of computer. 

More precisely, the proof of the theorem is carried out by inspection of Tables \ref{table:SignSignatureG12}, \ref{table:SignSignatureF14}, \ref{table:SignSignatureE16}, \ref{table:SignSignatureE17}, \ref{table:SignSignatureE18}. 
The tables are laid out in a similar fashion to Table \ref{table:nonlin}. 
The rows are labeled by the irreducible representations of the ambient Weyl group $ W$. 
The columns are labeled by the additively closed root subsystems that generate the root reflection subgroups of $W$.
The column-dividing lines separate the additively closed root subsystems arising from parabolic subgroups, pseudo-parabolic subgroups which are not parabolic, and root reflection subgroups which are not pseudo-parabolic. Finally, the entries of the tables give the multiplicities of the sign representation of the root reflection subgroups in the restrictions of the irreducible representations of $ W$.

In view of Lemma \ref{le:ConjugacyClassesAndOuterAutos} the sign multiplicities over a root reflection subgroup depend only on the isomorphism class of the underlying additively closed root subsystem, and so the data presented in our tables determines the extended sign signature unambiguously.

We label the irreducible representations of $  W$ following \cite[Chapter 13]{Carter:FiniteGroupsLieType} using the expressions of the form $\phi_{a, b}$. 
Here, $a$ stands for the dimension of the representation and $b$ stands for the smallest polynomial degree needed to realize the representation (see \cite[page 411]{Carter:FiniteGroupsLieType}). 
\extendedVersion{For the reader's convenience, in the appendix we reproduce the Weyl group representation notation for $G_2, F_4 $ and $E_6$, as specified by \cite[Chapter 13]{Carter:FiniteGroupsLieType} and extracted from the CHEVIE package of the computer algebra system GAP3 (\cite{GeckHissLuebeckMallePfeifer:CHEVIE-GAP}, \cite{SchoenertEtAl:GAP}).}

As implied by the statement of Theorem \ref{thm:classicalAndExceptional}, the sign signature is not sufficient to distinguish irreducible representations in the exceptional types. For example, the sign representation multiplicity table of $G_2$ (Table \ref{table:SignSignatureG12} below) reveals that the two dimensional representations $\phi_{2,1}$ and $ \phi_{2,2}$  share the same sign signature. 

The tables were computed in the computer algebra system GAP 3 \cite{SchoenertEtAl:GAP} using the built-in functions of the CHEVIE package \cite{GeckHissLuebeckMallePfeifer:CHEVIE-GAP}. Example GAP 3 code for generating the $A_1$ column of the $E_8$ table is given below.

\begin{verbatim}
SignSigValues := function(G,l)
  local H;
  H := ReflectionSubgroup(G,l);
  return List(
    Irr(G), chi -> Sum(
      List(ConjugacyClasses(H), cc -> 
        Representative(cc)^chi * 
        (1-2*(CoxeterLength(G,Representative(cc)) mod 2)) * 
        Size(cc) / Size(H)
      )
    )
  );
end;
E8 := CoxeterGroup("E",8);
SignSigValues(E8, [1]);
\end{verbatim}

The root reflection subgroup is computed by the $ReflectionSubgroup$ function from a list of roots. The root list is given by specifying the positions of the roots in the root system: in the example above, the second argument of the last function, $[1] $, stands for the first (in the order used by CHEVIE) positive root of the root system. To compute another column of the $E_8$ table, one simply needs to replace the list of indices of generating roots. 

A list of the additively closed root subsystems of the exceptional Lie algebras can be found in \cite[Table 11]{Dynkin:Semisimple}. We realize explicitly the generating roots for the additively closed root subsystems using the ``calculator project'' in the programming language C++, \cite{Folz-DonahueMilev:vpf}. The final data processing and table generation were also done in C++.

\begin{proof}[Proof of Theorem \ref{thm:classicalAndExceptional}]
Combine the results of Sections \ref{sec:typeA}, \ref{sec:typeB}, \ref{sec:typeD} with inspection of Tables \ref{table:SignSignatureG12}, \ref{table:SignSignatureF14}, \ref{table:SignSignatureE16}, \ref{table:SignSignatureE17}, \ref{table:SignSignatureE18}.
\end{proof}
\bibliographystyle{plain}
\bibliography{FolzDonahueJacksonNoelMilev_TauSignatures}

\hideme{
\section{Tables used in Section \ref{sec:Exceptional}}

\begin{landscape}
{\tiny 
\renewcommand{\arraystretch}{0}%
\subsection{$G_2$}
The following families of representations share the same sign signature. $(\phi_{2,1} , \phi_{2,2} )$ 

\begin{longtable}{c|p{0.275cm}p{0.275cm}p{0.275cm}p{0.275cm}|p{0.275cm}p{0.275cm}} 
\caption{\label{table:SignSignatureG12}Multiplicity of the sign representation over the classes of root subgroups. There are 4 parabolic subgroup classes, 2 pseudo-parabolic subgroup classes that are not parabolic, and 0 non-pseudo-parabolic subgroup classes. 
}\\ &$L_{1}$&$L_{2}$&$L_{3}$&$L_{4}$&${\widehat{L}}_{1}$&${\widehat{L}}_{2}$\\ 
&$0$&$A_1$&$A^{3}_1$&$G_2$&$A_2$&$A^{3}_1+A_1$\\\hline 
$\phi_{1,6} $&1&1&1&1&1&1\\ 
\\ 
$\phi_{1,3}''$&1&0&1&0&0&0\\ 
\\ 
$\phi_{1,3}' $&1&1&0&0&1&0\\ 
\\ 
$\phi_{1,0} $&1&0&0&0&0&0\\ 
\\ 
$\phi_{2,1} $&2&1&1&0&0&0\\ 
\\ 
$\phi_{2,2} $&2&1&1&0&0&1\\ 
\end{longtable}

\subsection{$F_4$}
The following families of representations share the same sign signature. $(\phi_{9,10} , \phi_{4,13} )$, $(\phi_{12,4} , \phi_{16,5} )$.

\begin{longtable}{c|p{0.275cm}p{0.275cm}p{0.275cm}p{0.275cm}p{0.275cm}p{0.275cm}p{0.275cm}p{0.275cm}p{0.275cm}p{0.275cm}p{0.275cm}p{0.275cm}|p{0.275cm}p{0.275cm}p{0.275cm}p{0.275cm}p{0.275cm}p{0.275cm}p{0.275cm}p{0.275cm}|p{0.275cm}p{0.275cm}p{0.275cm}p{0.275cm}} 
\caption{\label{table:SignSignatureF14}Multiplicity of the sign representation over the classes of root subgroups. There are 12 parabolic subgroup classes, 8 pseudo-parabolic subgroup classes that are not parabolic, and 4 non-pseudo-parabolic subgroup classes. 
}\\ &$L_{1}$&$L_{2}$&$L_{3}$&$L_{4}$&$L_{5}$&$L_{6}$&$L_{7}$&$L_{8}$&$L_{9}$&$L_{10}$&$L_{11}$&$L_{12}$&${\widehat{L}}_{1}$&${\widehat{L}}_{2}$&${\widehat{L}}_{3}$&${\widehat{L}}_{4}$&${\widehat{L}}_{5}$&${\widehat{L}}_{6}$&${\widehat{L}}_{7}$&${\widehat{L}}_{8}$&-&-&-&-\\ 
Irrep label&$0$&$A_1$&$A^{2}_1$&$A_2$&$B_2$&$A^{2}_1+A_1$&$A^{2}_2$&$B_3$&$A^{2}_1+A_2$&$A^{2}_2+A_1$&$C_3$&$F_4$&$2A_1$&$B_2+A_1$&$A_3$&$A^{2}_1+2A_1$&$B_4$&$A^{2}_1+A_3$&$A^{2}_2+A_2$&$C_3+A_1$&$3A_1$&$4A_1$&$B_2+2A_1$&$D_4$\\\hline \endhead 
$\phi_{1,0} $&1&0&0&0&0&0&0&0&0&0&0&0&0&0&0&0&0&0&0&0&0&0&0&0\\ 
\\ 
$\phi_{1,12}''$&1&0&1&0&0&0&1&0&0&0&0&0&0&0&0&0&0&0&0&0&0&0&0&0\\ 
\\ 
$\phi_{1,12}' $&1&1&0&1&0&0&0&0&0&0&0&0&1&0&1&0&0&0&0&0&1&1&0&1\\ 
\\ 
$\phi_{1,24} $&1&1&1&1&1&1&1&1&1&1&1&1&1&1&1&1&1&1&1&1&1&1&1&1\\ 
\\ 
$\phi_{2,4}'' $&2&0&1&0&0&0&0&0&0&0&0&0&0&0&0&0&0&0&0&0&0&0&0&0\\ 
\\ 
$\phi_{2,16}' $&2&2&1&2&1&1&0&1&1&0&0&0&2&1&2&1&1&1&0&0&2&2&1&2\\ 
\\ 
$\phi_{2,4}' $&2&1&0&0&0&0&0&0&0&0&0&0&1&0&0&0&0&0&0&0&1&1&0&0\\ 
\\ 
$\phi_{2,16}''$&2&1&2&0&1&1&2&0&0&1&1&0&1&1&0&1&0&0&0&1&1&1&1&0\\ 
\\ 
$\phi_{4,8} $&4&2&2&0&1&1&0&0&0&0&0&0&2&1&0&1&0&0&0&0&2&2&1&0\\ 
\\ 
$\phi_{9,2} $&9&3&3&0&0&1&0&0&0&0&0&0&1&0&0&0&0&0&0&0&0&0&0&0\\ 
\\ 
$\phi_{9,6}'' $&9&3&6&0&1&2&3&0&0&1&0&0&1&0&0&1&0&0&0&0&0&0&0&0\\ 
\\ 
$\phi_{9,6}' $&9&6&3&3&1&2&0&0&1&0&0&0&4&1&1&1&0&0&0&0&3&3&1&0\\ 
\\ 
$\phi_{9,10} $&9&6&6&3&3&4&3&1&2&2&1&0&4&2&1&3&0&1&1&1&3&3&2&0\\ 
\\ 
$\phi_{6,6}' $&6&3&3&1&0&2&1&0&1&1&0&0&1&0&0&1&0&0&1&0&0&0&0&0\\ 
\\ 
$\phi_{6,6}'' $&6&3&3&1&1&1&1&0&0&0&0&0&1&0&0&0&0&0&0&0&0&0&0&0\\ 
\\ 
$\phi_{12,4} $&12&6&6&2&1&3&2&0&1&1&0&0&2&0&0&1&0&0&0&0&0&0&0&0\\ 
\\ 
$\phi_{4,1} $&4&1&1&0&0&0&0&0&0&0&0&0&0&0&0&0&0&0&0&0&0&0&0&0\\ 
\\ 
$\phi_{4,7}'' $&4&1&3&0&0&1&2&0&0&1&0&0&0&0&0&0&0&0&0&0&0&0&0&0\\ 
\\ 
$\phi_{4,7}' $&4&3&1&2&0&1&0&0&1&0&0&0&2&0&1&1&0&1&0&0&1&0&0&0\\ 
\\ 
$\phi_{4,13} $&4&3&3&2&2&2&2&1&1&1&1&0&2&1&1&1&0&0&0&0&1&0&0&0\\ 
\\ 
$\phi_{8,3}'' $&8&2&4&0&0&1&1&0&0&0&0&0&0&0&0&0&0&0&0&0&0&0&0&0\\ 
\\ 
$\phi_{8,9}' $&8&6&4&4&2&3&1&1&2&1&0&0&4&1&2&2&0&1&1&0&2&0&0&0\\ 
\\ 
$\phi_{8,3}' $&8&4&2&1&0&1&0&0&0&0&0&0&2&0&0&1&0&0&0&0&1&0&0&0\\ 
\\ 
$\phi_{8,9}'' $&8&4&6&1&2&3&4&0&1&2&1&0&2&1&0&1&0&0&1&0&1&0&0&0\\ 
\\ 
$\phi_{16,5} $&16&8&8&2&2&4&2&0&1&1&0&0&4&1&0&2&0&0&0&0&2&0&0&0\\ 
\end{longtable} 

\subsection{$E_6$}
The following families of representations share the same sign signature. $(\phi_{6,25} , \phi_{20,20})$, $(\phi_{60,8} , \phi_{80,7} )$.

\begin{longtable}{c|p{0.275cm}p{0.275cm}p{0.275cm}p{0.275cm}p{0.275cm}p{0.275cm}p{0.275cm}p{0.275cm}p{0.275cm}p{0.275cm}p{0.275cm}p{0.275cm}p{0.275cm}p{0.275cm}p{0.275cm}p{0.275cm}p{0.275cm}|p{0.275cm}p{0.275cm}p{0.275cm}p{0.275cm}} 
\caption{\label{table:SignSignatureE16}Multiplicity of the sign representation over the classes of root subgroups. There are 17 parabolic subgroup classes, 4 pseudo-parabolic subgroup classes that are not parabolic, and 0 non-pseudo-parabolic subgroup classes. 
}\\ &$L_{1}$&$L_{2}$&$L_{3}$&$L_{4}$&$L_{5}$&$L_{6}$&$L_{7}$&$L_{8}$&$L_{9}$&$L_{10}$&$L_{11}$&$L_{12}$&$L_{13}$&$L_{14}$&$L_{15}$&$L_{16}$&$L_{17}$&${\widehat{L}}_{1}$&${\widehat{L}}_{2}$&${\widehat{L}}_{3}$&${\widehat{L}}_{4}$\\ 
Irrep label&$0$&$A_1$&$2A_1$&$A_2$&$3A_1$&$A_2+A_1$&$A_3$&$A_2+2A_1$&$2A_2$&$A_3+A_1$&$A_4$&$D_4$&$2A_2+A_1$&$A_4+A_1$&$A_5$&$D_5$&$E_6$&$4A_1$&$A_3+2A_1$&$3A_2$&$A_5+A_1$\\\hline 
\endhead 
$\phi_{1,0} $&1&0&0&0&0&0&0&0&0&0&0&0&0&0&0&0&0&0&0&0&0\\ 
\\ 
$\phi_{1,36} $&1&1&1&1&1&1&1&1&1&1&1&1&1&1&1&1&1&1&1&1&1\\ 
\\ 
$\phi_{10,9} $&10&5&3&1&2&1&0&1&1&0&0&0&1&0&0&0&0&1&0&1&0\\ 
\\ 
$\phi_{6,1} $&6&1&0&0&0&0&0&0&0&0&0&0&0&0&0&0&0&0&0&0&0\\ 
\\ 
$\phi_{6,25} $&6&5&4&4&3&3&3&2&2&2&2&2&1&1&1&1&0&2&1&0&0\\ 
\\ 
$\phi_{20,10}$&20&10&4&4&1&1&1&0&0&0&0&0&0&0&0&0&0&0&0&0&0\\ 
\\ 
$\phi_{15,5} $&15&5&1&1&0&0&0&0&0&0&0&0&0&0&0&0&0&0&0&0&0\\ 
\\ 
$\phi_{15,17}$&15&10&6&6&3&3&3&1&1&1&1&1&0&0&0&0&0&1&0&0&0\\ 
\\ 
$\phi_{15,4} $&15&5&2&0&1&0&0&0&0&0&0&0&0&0&0&0&0&1&0&0&0\\ 
\\ 
$\phi_{15,16}$&15&10&7&5&5&4&2&3&3&2&1&0&2&1&1&0&0&4&2&1&1\\ 
\\ 
$\phi_{20,2} $&20&5&1&0&0&0&0&0&0&0&0&0&0&0&0&0&0&0&0&0&0\\ 
\\ 
$\phi_{20,20}$&20&15&11&10&8&7&6&5&4&4&3&3&3&2&1&1&0&6&3&2&1\\ 
\\ 
$\phi_{24,6} $&24&10&4&2&1&1&0&0&1&0&0&0&0&0&0&0&0&0&0&0&0\\ 
\\ 
$\phi_{24,12}$&24&14&8&6&5&3&2&2&1&1&0&1&1&0&0&0&0&4&1&1&0\\ 
\\ 
$\phi_{30,3} $&30&10&3&1&1&0&0&0&0&0&0&0&0&0&0&0&0&0&0&0&0\\ 
\\ 
$\phi_{30,15}$&30&20&13&11&8&7&5&4&4&3&2&1&2&1&1&0&0&4&1&1&0\\ 
\\ 
$\phi_{60,8} $&60&30&16&8&9&5&1&3&2&1&0&0&1&0&0&0&0&6&1&0&0\\ 
\\ 
$\phi_{80,7} $&80&40&20&12&10&6&2&3&2&1&0&0&1&0&0&0&0&4&0&0&0\\ 
\\ 
$\phi_{90,8} $&90&45&21&15&9&6&3&2&1&1&0&0&0&0&0&0&0&3&0&0&0\\ 
\\ 
$\phi_{60,5} $&60&25&11&4&5&2&0&1&0&0&0&0&0&0&0&0&0&2&0&0&0\\ 
\\ 
$\phi_{60,11}$&60&35&21&14&13&9&4&6&4&3&1&0&3&1&0&0&0&8&2&2&0\\ 
\\ 
$\phi_{64,4} $&64&24&8&4&2&1&0&0&0&0&0&0&0&0&0&0&0&0&0&0&0\\ 
\\ 
$\phi_{64,13}$&64&40&24&20&14&11&8&6&4&4&2&2&2&1&0&0&0&8&2&0&0\\ 
\\ 
$\phi_{81,6} $&81&36&15&9&6&3&1&1&0&0&0&0&0&0&0&0&0&3&0&0&0\\ 
\\ 
$\phi_{81,10}$&81&45&24&18&12&9&5&4&3&2&1&0&1&0&0&0&0&6&1&0&0\\ 
\end{longtable} 

\subsection{$E_7$}
The following families of representations share the same sign signature. 
$(\phi_{7,46}, \phi_{27,37})$, 
$(\phi_{35,31}, \phi_{56,30})$, 
$(\phi_{105,26}, \phi_{168,21})$, 
$(\phi_{120,25}, \phi_{189,22})$, 
$(\phi_{216,9}, \phi_{405,8})$, 
$(\phi_{280,8}, \phi_{315,7})$, 
$(\phi_{280,17}, \phi_{315,16})$, 
$(\phi_{512,12}, \phi_{512,11})$.

The following families of representations share the same pseudo sign signature. $(\phi_{512,12}, \phi_{512,11})$.

\begin{longtable}{c|p{0.275cm}p{0.275cm}p{0.275cm}p{0.275cm}p{0.275cm}p{0.275cm}p{0.275cm}p{0.275cm}p{0.275cm}p{0.275cm}p{0.275cm}p{0.275cm}p{0.275cm}p{0.275cm}p{0.275cm}p{0.275cm}p{0.275cm}p{0.275cm}p{0.275cm}p{0.275cm}p{0.275cm}p{0.275cm}p{0.275cm}p{0.275cm}p{0.275cm}} 
\caption{\label{table:SignSignatureE17}Multiplicity of the sign representation over the classes of root subgroups. There are 32 parabolic subgroup classes, 12 pseudo-parabolic subgroup classes that are not parabolic, and 3 non-pseudo-parabolic subgroup classes. 
}\\ &$L_{1}$&$L_{2}$&$L_{3}$&$L_{4}$&$L_{5}$&$L_{6}$&$L_{7}$&$L_{8}$&$L_{9}$&$L_{10}$&$L_{11}$&$L_{12}$&$L_{13}$&$L_{14}$&$L_{15}$&$L_{16}$&$L_{17}$&$L_{18}$&$L_{19}$&$L_{20}$&$L_{21}$&$L_{22}$&$L_{23}$&$L_{24}$&$L_{25}$\\ 
Irrep label&$0$&$A_1$&$2A_1$&$A_2$&$3A_1$&$3A_1$&$A_2+A_1$&$A_3$&$4A_1$&$A_2+2A_1$&$2A_2$&$A_3+A_1$&$A_3+A_1$&$A_4$&$D_4$&$A_2+3A_1$&$2A_2+A_1$&$A_3+2A_1$&$A_3+A_2$&$A_4+A_1$&$D_4+A_1$&$A_5$&$A_5$&$D_5$&$A_3+A_2+A_1$\\\hline 
\endhead 
$\phi_{1,0}$&1&0&0&0&0&0&0&0&0&0&0&0&0&0&0&0&0&0&0&0&0&0&0&0&0\\ 
\\ 
$\phi_{1,63}$&1&1&1&1&1&1&1&1&1&1&1&1&1&1&1&1&1&1&1&1&1&1&1&1&1\\ 
\\ 
$\phi_{7,46}$&7&6&5&5&4&4&4&4&3&3&3&3&3&3&3&2&2&2&2&2&2&2&2&2&1\\ 
\\ 
$\phi_{7,1}$&7&1&0&0&0&0&0&0&0&0&0&0&0&0&0&0&0&0&0&0&0&0&0&0&0\\ 
\\ 
$\phi_{15,28}$&15&10&7&5&5&4&4&2&3&3&3&2&1&1&0&2&2&1&2&1&0&1&0&0&1\\ 
\\ 
$\phi_{15,7}$&15&5&2&0&1&2&0&0&1&0&0&0&0&0&0&0&0&0&0&0&0&0&0&0&0\\ 
\\ 
$\phi_{21,6}$&21&6&1&1&0&0&0&0&0&0&0&0&0&0&0&0&0&0&0&0&0&0&0&0&0\\ 
\\ 
$\phi_{21,33}$&21&15&10&10&6&6&6&6&3&3&3&3&3&3&3&1&1&1&1&1&1&1&1&1&0\\ 
\\ 
$\phi_{21,36}$&21&16&12&11&9&8&8&7&6&6&5&5&4&4&4&4&4&3&3&3&2&2&1&2&2\\ 
\\ 
$\phi_{21,3}$&21&5&1&0&0&1&0&0&0&0&0&0&0&0&0&0&0&0&0&0&0&0&0&0&0\\ 
\\ 
$\phi_{27,2}$&27&6&1&0&0&0&0&0&0&0&0&0&0&0&0&0&0&0&0&0&0&0&0&0&0\\ 
\\ 
$\phi_{27,37}$&27&21&16&15&12&12&11&10&9&8&7&7&7&6&6&6&5&5&4&4&4&3&3&3&3\\ 
\\ 
$\phi_{35,22}$&35&20&10&10&4&4&4&4&1&1&1&1&1&1&1&0&0&0&0&0&0&0&0&0&0\\ 
\\ 
$\phi_{35,13}$&35&15&5&5&1&1&1&1&0&0&0&0&0&0&0&0&0&0&0&0&0&0&0&0&0\\ 
\\ 
$\phi_{35,4}$&35&10&3&0&1&0&0&0&0&0&0&0&0&0&0&0&0&0&0&0&0&0&0&0&0\\ 
\\ 
$\phi_{35,31}$&35&25&18&15&13&14&11&8&10&8&7&6&7&4&3&6&5&5&4&3&3&2&3&1&3\\ 
\\ 
$\phi_{56,30}$&56&40&28&25&19&20&17&14&13&11&10&9&10&7&6&7&6&6&5&4&4&3&4&2&3\\ 
\\ 
$\phi_{56,3}$&56&16&4&1&1&0&0&0&0&0&0&0&0&0&0&0&0&0&0&0&0&0&0&0&0\\ 
\\ 
$\phi_{70,18}$&70&40&24&15&15&16&10&4&10&7&5&3&4&1&0&5&4&3&2&1&0&0&1&0&2\\ 
\\ 
$\phi_{70,9}$&70&30&14&5&7&6&3&0&3&2&1&0&0&0&0&1&1&0&0&0&0&0&0&0&0\\ 
\\ 
$\phi_{84,12}$&84&40&20&10&10&8&6&1&4&3&3&1&0&0&0&1&1&0&1&0&0&0&0&0&0\\ 
\\ 
$\phi_{84,15}$&84&44&24&14&14&16&8&3&10&5&3&2&3&0&1&4&2&2&1&0&1&0&0&0&1\\ 
\\ 
$\phi_{105,26}$&105&70&45&40&28&28&24&20&17&14&11&11&11&8&8&8&6&6&4&4&4&2&2&2&2\\ 
\\ 
$\phi_{105,5}$&105&35&10&5&2&2&1&0&0&0&0&0&0&0&0&0&0&0&0&0&0&0&0&0&0\\ 
\\ 
$\phi_{105,6}$&105&40&16&5&7&8&2&0&4&1&0&0&0&0&0&1&0&0&0&0&0&0&0&0&0\\ 
\\ 
$\phi_{105,21}$&105&65&41&30&26&25&20&11&16&13&11&8&7&4&1&8&7&5&5&3&1&2&1&0&3\\ 
\\ 
$\phi_{105,12}$&105&50&23&15&11&8&6&3&4&3&1&1&0&0&1&1&1&0&0&0&0&0&0&0&0\\ 
\\ 
$\phi_{105,15}$&105&55&28&20&13&16&10&5&6&4&4&2&4&1&0&1&1&1&1&0&0&0&1&0&0\\ 
\\ 
$\phi_{120,4}$&120&40&12&5&3&4&1&0&1&0&0&0&0&0&0&0&0&0&0&0&0&0&0&0&0\\ 
\\ 
$\phi_{120,25}$&120&80&52&45&33&32&28&22&20&17&14&13&12&9&8&10&8&7&6&5&4&3&2&2&3\\ 
\\ 
$\phi_{168,6}$&168&64&24&10&8&8&4&0&2&1&1&0&0&0&0&0&0&0&0&0&0&0&0&0&0\\ 
\\ 
$\phi_{168,21}$&168&104&64&50&40&40&30&20&26&19&13&12&12&6&6&13&9&8&5&4&4&1&1&1&4\\ 
\\ 
$\phi_{189,10}$&189&84&34&24&12&12&8&4&3&2&1&1&1&0&0&0&0&0&0&0&0&0&0&0&0\\ 
\\ 
$\phi_{189,17}$&189&105&55&45&27&27&21&15&12&9&6&6&6&3&3&3&2&2&1&1&1&0&0&0&0\\ 
\\ 
$\phi_{189,22}$&189&120&76&60&48&48&38&25&30&24&19&16&16&9&6&15&12&10&8&6&4&3&3&1&5\\ 
\\ 
$\phi_{189,5}$&189&69&25&9&9&9&3&0&3&1&0&0&0&0&0&0&0&0&0&0&0&0&0&0&0\\ 
\\ 
$\phi_{189,20}$&189&114&67&54&39&36&30&21&21&17&12&11&9&6&6&9&7&5&4&3&2&1&0&1&2\\ 
\\ 
$\phi_{189,7}$&189&75&28&15&9&12&5&1&3&1&1&0&1&0&0&0&0&0&0&0&0&0&0&0&0\\ 
\\ 
$\phi_{210,6}$&210&80&28&15&9&8&4&1&2&1&0&0&0&0&0&0&0&0&0&0&0&0&0&0&0\\ 
\\ 
$\phi_{210,21}$&210&130&78&65&45&46&37&27&25&20&16&14&15&9&7&10&8&7&5&4&3&2&3&1&2\\ 
\\ 
$\phi_{210,10}$&210&100&50&25&26&28&14&3&15&8&5&2&3&0&0&5&3&2&1&0&0&0&0&0&1\\ 
\\ 
$\phi_{210,13}$&210&110&60&35&34&32&21&7&19&13&9&5&4&1&0&8&6&3&3&1&0&0&0&0&2\\ 
\\ 
$\phi_{216,16}$&216&120&68&45&39&36&27&12&21&16&12&8&6&3&0&9&7&4&4&2&0&1&0&0&2\\ 
\\ 
$\phi_{216,9}$&216&96&44&21&21&24&10&2&12&5&2&1&2&0&0&3&1&1&0&0&0&0&0&0&0\\ 
\\ 
$\phi_{280,18}$&280&160&88&70&46&48&36&24&24&17&13&11&12&6&4&8&5&5&3&2&2&1&1&0&1\\ 
\\ 
$\phi_{280,9}$&280&120&48&30&18&16&10&4&6&3&1&1&0&0&0&1&0&0&0&0&0&0&0&0&0\\ 
\\ 
$\phi_{280,8}$&280&120&52&25&23&20&11&2&9&5&2&1&0&0&0&2&1&0&0&0&0&0&0&0&0\\ 
\\ 
$\phi_{280,17}$&280&160&92&65&53&56&38&20&32&22&16&12&14&5&2&13&9&8&5&3&2&1&2&0&3\\ 
\\ 
$\phi_{315,16}$&315&180&102&75&57&60&42&24&33&23&17&13&15&6&3&13&9&8&5&3&2&1&2&0&3\\ 
\\ 
$\phi_{315,7}$&315&135&57&30&24&21&12&3&9&5&2&1&0&0&0&2&1&0&0&0&0&0&0&0&0\\ 
\\ 
$\phi_{336,14}$&336&176&88&66&42&40&30&18&18&13&8&7&6&3&2&5&3&2&1&1&0&0&0&0&0\\ 
\\ 
$\phi_{336,11}$&336&160&72&50&30&32&20&10&12&7&4&3&4&1&0&2&1&1&0&0&0&0&0&0&0\\ 
\\ 
$\phi_{378,14}$&378&204&110&75&60&60&40&20&33&22&14&11&11&3&3&12&8&6&4&2&2&0&0&0&2\\ 
\\ 
$\phi_{378,9}$&378&174&80&45&36&36&21&6&15&9&6&3&3&0&0&3&2&1&1&0&0&0&0&0&0\\ 
\\ 
$\phi_{405,8}$&405&180&78&45&33&36&18&6&15&7&3&2&3&0&0&3&1&1&0&0&0&0&0&0&0\\ 
\\ 
$\phi_{405,15}$&405&225&123&90&66&63&48&27&33&25&18&14&12&6&3&12&9&6&5&3&1&1&0&0&2\\ 
\\ 
$\phi_{420,10}$&420&200&92&60&40&40&26&11&16&10&7&4&4&1&0&3&2&1&1&0&0&0&0&0&0\\ 
\\ 
$\phi_{420,13}$&420&220&112&80&56&56&38&21&28&18&11&9&9&3&3&9&5&4&2&1&1&0&0&0&1\\ 
\\ 
$\phi_{512,12}$&512&256&128&80&64&64&40&16&32&20&12&8&8&2&0&10&6&4&2&1&0&0&0&0&1\\ 
\\ 
$\phi_{512,11}$&512&256&128&80&64&64&40&16&32&20&12&8&8&2&0&10&6&4&2&1&0&0&0&0&1\\ 
\end{longtable} 

\begin{longtable}{c|p{0.275cm}p{0.275cm}p{0.275cm}p{0.275cm}p{0.275cm}p{0.275cm}p{0.275cm}|p{0.275cm}p{0.275cm}p{0.275cm}p{0.275cm}p{0.275cm}p{0.275cm}p{0.275cm}p{0.275cm}p{0.275cm}p{0.275cm}p{0.275cm}p{0.275cm}|p{0.275cm}p{0.275cm}p{0.275cm}} 
\addtocounter{table}{-1}
&$L_{26}$&$L_{27}$&$L_{28}$&$L_{29}$&$L_{30}$&$L_{31}$&$L_{32}$&${\widehat{L}}_{1}$&${\widehat{L}}_{2}$&${\widehat{L}}_{3}$&${\widehat{L}}_{4}$&${\widehat{L}}_{5}$&${\widehat{L}}_{6}$&${\widehat{L}}_{7}$&${\widehat{L}}_{8}$&${\widehat{L}}_{9}$&${\widehat{L}}_{10}$&${\widehat{L}}_{11}$&${\widehat{L}}_{12}$&-&-&-\\ 
Irrep label&$A_4+A_2$&$A_5+A_1$&$D_5+A_1$&$A_6$&$D_6$&$E_6$&$E_7$&$4A_1$&$5A_1$&$A_3+2A_1$&$3A_2$&$A_3+3A_1$&$2A_3$&$D_4+2A_1$&$A_5+A_1$&$2A_3+A_1$&$A_5+A_2$&$D_6+A_1$&$A_7$&$6A_1$&$7A_1$&$D_4+3A_1$\\\hline 
\endhead 
$\phi_{1,0}$&0&0&0&0&0&0&0&0&0&0&0&0&0&0&0&0&0&0&0&0&0&0\\ 
\\ 
$\phi_{1,63}$&1&1&1&1&1&1&1&1&1&1&1&1&1&1&1&1&1&1&1&1&1&1\\ 
\\ 
$\phi_{7,46}$&1&1&1&1&1&1&0&3&2&2&1&1&1&1&1&0&0&0&0&1&0&0\\ 
\\ 
$\phi_{7,1}$&0&0&0&0&0&0&0&0&0&0&0&0&0&0&0&0&0&0&0&0&0&0\\ 
\\ 
$\phi_{15,28}$&1&0&0&1&0&0&0&4&2&2&1&1&2&0&1&1&0&0&1&1&0&0\\ 
\\ 
$\phi_{15,7}$&0&0&0&0&0&0&0&1&1&0&0&0&0&0&0&0&0&0&0&1&1&0\\ 
\\ 
$\phi_{21,6}$&0&0&0&0&0&0&0&0&0&0&0&0&0&0&0&0&0&0&0&0&0&0\\ 
\\ 
$\phi_{21,33}$&0&0&0&0&0&0&0&3&1&1&0&0&0&0&0&0&0&0&0&0&0&0\\ 
\\ 
$\phi_{21,36}$&2&1&1&1&0&1&0&7&4&4&3&2&2&1&2&1&1&0&1&2&0&0\\ 
\\ 
$\phi_{21,3}$&0&0&0&0&0&0&0&0&0&0&0&0&0&0&0&0&0&0&0&0&0&0\\ 
\\ 
$\phi_{27,2}$&0&0&0&0&0&0&0&0&0&0&0&0&0&0&0&0&0&0&0&0&0&0\\ 
\\ 
$\phi_{27,37}$&2&2&2&1&1&1&0&9&7&5&3&4&2&3&2&2&1&1&0&6&6&3\\ 
\\ 
$\phi_{35,22}$&0&0&0&0&0&0&0&1&0&0&0&0&0&0&0&0&0&0&0&0&0&0\\ 
\\ 
$\phi_{35,13}$&0&0&0&0&0&0&0&0&0&0&0&0&0&0&0&0&0&0&0&0&0&0\\ 
\\ 
$\phi_{35,4}$&0&0&0&0&0&0&0&1&0&0&0&0&0&0&0&0&0&0&0&0&0&0\\ 
\\ 
$\phi_{35,31}$&2&2&1&1&1&0&0&10&8&5&3&4&3&3&2&2&1&1&1&7&7&3\\ 
\\ 
$\phi_{56,30}$&2&2&1&1&1&0&0&12&8&5&3&3&2&2&1&1&1&0&0&4&0&0\\ 
\\ 
$\phi_{56,3}$&0&0&0&0&0&0&0&0&0&0&0&0&0&0&0&0&0&0&0&0&0&0\\ 
\\ 
$\phi_{70,18}$&1&1&0&0&0&0&0&9&6&2&3&2&1&0&0&1&1&0&0&3&0&0\\ 
\\ 
$\phi_{70,9}$&0&0&0&0&0&0&0&3&1&0&1&0&0&0&0&0&0&0&0&0&0&0\\ 
\\ 
$\phi_{84,12}$&0&0&0&0&0&0&0&6&2&1&0&0&1&0&0&0&0&0&0&1&0&0\\ 
\\ 
$\phi_{84,15}$&0&0&0&0&0&0&0&10&8&2&1&2&1&1&0&1&0&0&0&7&7&1\\ 
\\ 
$\phi_{105,26}$&1&1&1&0&0&0&0&17&10&6&2&3&1&2&1&0&0&0&0&5&0&0\\ 
\\ 
$\phi_{105,5}$&0&0&0&0&0&0&0&0&0&0&0&0&0&0&0&0&0&0&0&0&0&0\\ 
\\ 
$\phi_{105,6}$&0&0&0&0&0&0&0&3&2&0&0&0&0&0&0&0&0&0&0&1&0&0\\ 
\\ 
$\phi_{105,21}$&2&1&0&1&0&0&0&16&10&5&4&3&2&1&1&1&1&0&0&7&7&1\\ 
\\ 
$\phi_{105,12}$&0&0&0&0&0&0&0&7&2&1&1&0&0&0&0&0&0&0&0&1&0&0\\ 
\\ 
$\phi_{105,15}$&0&0&0&0&0&0&0&6&2&1&0&0&1&0&0&0&0&0&0&0&0&0\\ 
\\ 
$\phi_{120,4}$&0&0&0&0&0&0&0&0&0&0&0&0&0&0&0&0&0&0&0&0&0&0\\ 
\\ 
$\phi_{120,25}$&2&1&1&1&0&0&0&20&12&7&3&4&2&2&1&1&0&0&0&8&8&2\\ 
\\ 
$\phi_{168,6}$&0&0&0&0&0&0&0&2&0&0&0&0&0&0&0&0&0&0&0&0&0&0\\ 
\\ 
$\phi_{168,21}$&2&1&1&0&0&0&0&26&18&8&5&6&2&3&1&2&1&0&0&14&14&3\\ 
\\ 
$\phi_{189,10}$&0&0&0&0&0&0&0&3&0&0&0&0&0&0&0&0&0&0&0&0&0&0\\ 
\\ 
$\phi_{189,17}$&0&0&0&0&0&0&0&12&4&2&0&0&0&0&0&0&0&0&0&0&0&0\\ 
\\ 
$\phi_{189,22}$&3&2&1&1&0&0&0&30&18&10&6&6&3&2&2&2&1&0&0&9&0&0\\ 
\\ 
$\phi_{189,5}$&0&0&0&0&0&0&0&3&1&0&0&0&0&0&0&0&0&0&0&0&0&0\\ 
\\ 
$\phi_{189,20}$&1&0&0&0&0&0&0&24&12&7&3&3&2&1&1&1&0&0&0&6&0&0\\ 
\\ 
$\phi_{189,7}$&0&0&0&0&0&0&0&3&1&0&0&0&0&0&0&0&0&0&0&0&0&0\\ 
\\ 
$\phi_{210,6}$&0&0&0&0&0&0&0&3&0&0&0&0&0&0&0&0&0&0&0&0&0&0\\ 
\\ 
$\phi_{210,21}$&1&1&0&0&0&0&0&25&13&7&3&3&1&1&1&0&0&0&0&7&7&1\\ 
\\ 
$\phi_{210,10}$&0&0&0&0&0&0&0&13&8&1&1&1&0&0&0&0&0&0&0&4&0&0\\ 
\\ 
$\phi_{210,13}$&1&0&0&0&0&0&0&19&11&3&3&2&1&0&0&1&0&0&0&7&7&0\\ 
\\ 
$\phi_{216,16}$&1&0&0&0&0&0&0&24&12&6&3&3&2&0&1&1&0&0&0&6&0&0\\ 
\\ 
$\phi_{216,9}$&0&0&0&0&0&0&0&12&8&1&0&1&0&0&0&0&0&0&0&6&6&0\\ 
\\ 
$\phi_{280,18}$&0&0&0&0&0&0&0&22&12&4&1&2&0&1&0&0&0&0&0&6&0&0\\ 
\\ 
$\phi_{280,9}$&0&0&0&0&0&0&0&6&2&0&0&0&0&0&0&0&0&0&0&0&0&0\\ 
\\ 
$\phi_{280,8}$&0&0&0&0&0&0&0&12&4&1&0&0&0&0&0&0&0&0&0&2&0&0\\ 
\\ 
$\phi_{280,17}$&1&1&0&0&0&0&0&32&20&8&3&5&2&2&1&1&0&0&0&14&14&2\\ 
\\ 
$\phi_{315,16}$&1&1&0&0&0&0&0&30&18&6&3&4&1&1&0&1&0&0&0&9&0&0\\ 
\\ 
$\phi_{315,7}$&0&0&0&0&0&0&0&9&3&0&0&0&0&0&0&0&0&0&0&0&0&0\\ 
\\ 
$\phi_{336,14}$&0&0&0&0&0&0&0&20&8&3&0&1&0&0&0&0&0&0&0&4&0&0\\ 
\\ 
$\phi_{336,11}$&0&0&0&0&0&0&0&12&4&1&0&0&0&0&0&0&0&0&0&0&0&0\\ 
\\ 
$\phi_{378,14}$&1&0&0&0&0&0&0&33&18&6&3&3&1&1&0&0&0&0&0&9&0&0\\ 
\\ 
$\phi_{378,9}$&0&0&0&0&0&0&0&15&5&1&0&0&0&0&0&0&0&0&0&0&0&0\\ 
\\ 
$\phi_{405,8}$&0&0&0&0&0&0&0&12&6&0&0&0&0&0&0&0&0&0&0&3&0&0\\ 
\\ 
$\phi_{405,15}$&1&0&0&0&0&0&0&33&15&6&3&2&1&0&0&0&0&0&0&6&6&0\\ 
\\ 
$\phi_{420,10}$&0&0&0&0&0&0&0&16&6&1&0&0&0&0&0&0&0&0&0&3&0&0\\ 
\\ 
$\phi_{420,13}$&0&0&0&0&0&0&0&28&14&4&1&2&0&0&0&0&0&0&0&7&7&0\\ 
\\ 
$\phi_{512,12}$&0&0&0&0&0&0&0&32&16&4&2&2&0&0&0&0&0&0&0&8&0&0\\ 
\\ 
$\phi_{512,11}$&0&0&0&0&0&0&0&32&16&4&2&2&0&0&0&0&0&0&0&8&8&0\\ 
\end{longtable} 

\subsection{$E_8$}
The following families of representations share the same sign signature. 
$(\phi_{35,74}, \phi_{84,64}, \phi_{8,91}, \phi_{112,63})$, 
$(\phi_{50,56}, \phi_{400,43})$, 
$(\phi_{210,52}, \phi_{700,42}, \phi_{560,47})$, $(\phi_{300,44}, \phi_{160,55})$, 
$(\phi_{567,46}, \phi_{1344,38}, \phi_{1400,37})$, $(\phi_{700,28}, \phi_{4200,21})$, 
$(\phi_{840,26}, \phi_{840,31})$, 
$(\phi_{972,12}, \phi_{3240,9})$, 
$(\phi_{972,32}, \phi_{1575,34}, \phi_{2268,30})$, $(\phi_{1050,10}, \phi_{1400,8})$, 
$(\phi_{1050,34}, \phi_{1400,32}, \phi_{3240,31})$, \\
$(\phi_{1344,8}, \phi_{1400,7})$, 
$(\phi_{3150,18}, \phi_{4200,18}, \phi_{4480,16}, \phi_{7168,17})$, 
$(\phi_{2240,10}, \phi_{1400,11})$, 
$(\phi_{2240,28}, \phi_{1400,29})$, 
$(\phi_{4536,18}, \phi_{5670,18}, \phi_{5600,19})$, $(\phi_{3200,22}, \phi_{5600,21})$, 
$(\phi_{4096,12}, \phi_{4096,11})$, \\
$(\phi_{4096,26}, \phi_{4200,24}, \phi_{4096,27}, \phi_{4536,23})$, 
$(\phi_{6075,22}, \phi_{2800,25})$ .

The following families of representations share the same pseudo sign signature. 
$(\phi_{1344,38}, \phi_{1400,37})$, 
$(\phi_{4536,18}, \phi_{5670,18})$, 
$(\phi_{4096,12}, \phi_{4096,11})$, 
$(\phi_{4096,26}, \phi_{4096,27})$, \\
$(\phi_{4200,24}, \phi_{4536,23})$. 

\begin{longtable}{c|p{0.275cm}p{0.275cm}p{0.275cm}p{0.275cm}p{0.275cm}p{0.275cm}p{0.275cm}p{0.275cm}p{0.275cm}p{0.275cm}p{0.275cm}p{0.275cm}p{0.275cm}p{0.275cm}p{0.275cm}p{0.275cm}p{0.275cm}p{0.275cm}p{0.275cm}p{0.275cm}p{0.275cm}p{0.275cm}p{0.275cm}p{0.275cm}p{0.275cm}} 
\caption{\label{table:SignSignatureE18}Multiplicity of the sign representation over the classes of root subgroups. There are 41 parabolic subgroup classes, 26 pseudo-parabolic subgroup classes that are not parabolic, and 10 non-pseudo-parabolic subgroup classes. 
}\\ &$L_{1}$&$L_{2}$&$L_{3}$&$L_{4}$&$L_{5}$&$L_{6}$&$L_{7}$&$L_{8}$&$L_{9}$&$L_{10}$&$L_{11}$&$L_{12}$&$L_{13}$&$L_{14}$&$L_{15}$&$L_{16}$&$L_{17}$&$L_{18}$&$L_{19}$&$L_{20}$&$L_{21}$&$L_{22}$&$L_{23}$&$L_{24}$&$L_{25}$\\ 
Irrep label&$0$&$A_1$&$2A_1$&$A_2$&$3A_1$&$A_2+A_1$&$A_3$&$4A_1$&$A_2+2A_1$&$2A_2$&$A_3+A_1$&$A_4$&$D_4$&$A_2+3A_1$&$2A_2+A_1$&$A_3+2A_1$&$A_3+A_2$&$A_4+A_1$&$D_4+A_1$&$A_5$&$D_5$&$2A_2+2A_1$&$A_3+A_2+A_1$&$2A_3$&$A_4+2A_1$\\\hline 
\endhead 
$\phi_{1,0}$&1&0&0&0&0&0&0&0&0&0&0&0&0&0&0&0&0&0&0&0&0&0&0&0&0\\ 
\\ 
$\phi_{1,120}$&1&1&1&1&1&1&1&1&1&1&1&1&1&1&1&1&1&1&1&1&1&1&1&1&1\\ 
\\ 
$\phi_{28,8}$&28&7&1&1&0&0&0&0&0&0&0&0&0&0&0&0&0&0&0&0&0&0&0&0&0\\ 
\\ 
$\phi_{28,68}$&28&21&15&15&10&10&10&6&6&6&6&6&6&3&3&3&3&3&3&3&3&1&1&1&1\\ 
\\ 
$\phi_{35,2}$&35&7&1&0&0&0&0&0&0&0&0&0&0&0&0&0&0&0&0&0&0&0&0&0&0\\ 
\\ 
$\phi_{35,74}$&35&28&22&21&17&16&15&13&12&11&11&10&10&9&8&8&7&7&7&6&6&6&5&4&5\\ 
\\ 
$\phi_{70,32}$&70&35&15&15&5&5&5&1&1&1&1&1&1&0&0&0&0&0&0&0&0&0&0&0&0\\ 
\\ 
$\phi_{50,8}$&50&15&5&0&2&0&0&1&0&0&0&0&0&0&0&0&0&0&0&0&0&0&0&0&0\\ 
\\ 
$\phi_{50,56}$&50&35&25&20&18&15&10&13&11&10&8&5&3&8&7&6&6&4&3&3&1&5&4&4&3\\ 
\\ 
$\phi_{84,4}$&84&21&5&0&1&0&0&0&0&0&0&0&0&0&0&0&0&0&0&0&0&0&0&0&0\\ 
\\ 
$\phi_{84,64}$&84&63&47&42&35&31&26&26&23&20&19&15&14&17&15&14&12&11&10&8&7&11&9&7&8\\ 
\\ 
$\phi_{168,24}$&168&84&44&24&24&14&4&14&8&6&3&0&1&5&3&2&2&0&1&0&0&2&1&1&0\\ 
\\ 
$\phi_{175,12}$&175&70&30&10&14&5&0&7&3&1&0&0&0&2&1&0&0&0&0&0&0&1&0&0&0\\ 
\\ 
$\phi_{175,36}$&175&105&65&45&41&30&15&26&20&16&11&5&1&13&11&8&7&4&1&2&0&7&5&4&3\\ 
\\ 
$\phi_{210,4}$&210&63&17&6&4&1&0&1&0&0&0&0&0&0&0&0&0&0&0&0&0&0&0&0&0\\ 
\\ 
$\phi_{210,52}$&210&147&101&90&68&60&50&45&39&34&32&25&23&25&21&20&17&15&14&11&9&13&10&8&9\\ 
\\ 
$\phi_{420,20}$&420&210&110&60&60&35&10&34&21&14&7&1&0&13&9&5&4&1&0&0&0&6&3&2&1\\ 
\\ 
$\phi_{300,8}$&300&105&35&15&10&5&0&2&1&1&0&0&0&0&0&0&0&0&0&0&0&0&0&0&0\\ 
\\ 
$\phi_{300,44}$&300&195&125&105&80&65&50&52&41&31&30&20&20&27&20&19&13&12&12&6&6&14&9&5&8\\ 
\\ 
$\phi_{350,14}$&350&140&50&35&15&10&5&3&2&1&1&0&0&0&0&0&0&0&0&0&0&0&0&0&0\\ 
\\ 
$\phi_{350,38}$&350&210&120&105&65&55&45&33&27&21&21&15&15&12&9&9&6&6&6&3&3&3&2&1&2\\ 
\\ 
$\phi_{525,12}$&525&210&80&45&29&15&5&9&5&2&1&0&1&1&1&0&0&0&0&0&0&0&0&0&0\\ 
\\ 
$\phi_{525,36}$&525&315&185&150&106&85&60&58&47&37&32&20&17&24&20&16&13&10&7&5&4&9&6&4&4\\ 
\\ 
$\phi_{567,6}$&567&189&57&27&15&6&1&3&1&0&0&0&0&0&0&0&0&0&0&0&0&0&0&0&0\\ 
\\ 
$\phi_{567,46}$&567&378&246&216&156&135&110&96&82&69&65&49&45&48&40&37&30&27&24&18&15&22&16&11&14\\ 
\\ 
$\phi_{1134,20}$&1134&567&279&189&135&90&45&63&42&27&21&6&6&18&12&9&6&3&3&0&0&4&2&1&1\\ 
\\ 
$\phi_{700,16}$&700&315&145&75&70&35&10&36&18&8&5&0&2&10&5&3&1&0&1&0&0&3&1&0&0\\ 
\\ 
$\phi_{700,28}$&700&385&215&145&120&85&40&66&48&38&25&10&2&26&20&14&13&6&2&3&0&10&6&5&3\\ 
\\ 
$\phi_{700,6}$&700&245&85&30&29&10&0&10&3&1&0&0&0&1&0&0&0&0&0&0&0&0&0&0&0\\ 
\\ 
$\phi_{700,42}$&700&455&295&240&191&155&110&124&100&81&71&45&36&65&52&46&37&29&24&17&10&34&24&17&19\\ 
\\ 
$\phi_{1400,20}$&1400&700&340&240&160&110&60&74&48&32&25&10&7&21&12&10&6&3&3&1&0&5&2&1&1\\ 
\\ 
$\phi_{840,14}$&840&378&174&90&80&45&10&36&21&15&6&0&0&9&6&3&3&0&0&0&0&2&1&1&0\\ 
\\ 
$\phi_{840,26}$&840&462&258&174&148&99&50&88&59&39&30&10&10&37&25&19&13&7&7&1&1&17&9&5&5\\ 
\\ 
$\phi_{1680,22}$&1680&840&400&300&180&130&80&76&52&34&30&14&10&19&11&10&5&4&3&1&0&3&1&0&1\\ 
\\ 
$\phi_{972,12}$&972&405&171&81&72&36&6&30&15&9&3&0&0&6&3&1&1&0&0&0&0&1&0&0&0\\ 
\\ 
$\phi_{972,32}$&972&567&333&243&198&144&84&120&87&63&51&24&18&54&39&32&23&15&12&6&3&25&15&9&10\\ 
\\ 
$\phi_{1050,10}$&1050&420&170&75&71&30&5&31&13&4&2&0&0&6&2&1&0&0&0&0&0&1&0&0&0\\ 
\\ 
$\phi_{1050,34}$&1050&630&380&285&229&175&105&137&106&84&66&35&19&63&50&40&33&22&13&12&3&29&19&13&13\\ 
\\ 
$\phi_{2100,20}$&2100&1050&510&360&240&165&90&108&73&48&37&15&12&30&20&14&9&5&3&1&1&7&3&1&1\\ 
\\ 
$\phi_{1344,8}$&1344&504&184&84&64&29&4&20&9&4&1&0&0&2&1&0&0&0&0&0&0&0&0&0&0\\ 
\\ 
$\phi_{1344,38}$&1344&840&520&420&320&255&180&196&155&120&107&66&56&94&73&64&48&39&32&20&14&44&29&18&23\\ 
\\ 
$\phi_{2688,20}$&2688&1344&672&432&336&216&96&168&108&68&48&16&4&54&34&24&14&8&2&2&0&17&7&2&4\\ 
\\ 
$\phi_{1400,8}$&1400&560&220&110&86&40&10&34&15&5&3&0&0&6&2&1&0&0&0&0&0&1&0&0&0\\ 
\\ 
$\phi_{1400,32}$&1400&840&500&390&294&230&150&170&133&105&87&50&34&75&59&49&39&28&19&15&6&32&21&14&15\\ 
\\ 
$\phi_{1575,10}$&1575&630&240&135&87&45&15&30&14&5&4&0&0&4&1&1&0&0&0&0&0&0&0&0&0\\ 
\\ 
$\phi_{1575,34}$&1575&945&555&450&318&255&180&177&140&110&97&60&48&74&57&50&38&30&24&16&9&28&18&11&14\\ 
\\ 
$\phi_{3150,18}$&3150&1575&795&495&405&255&105&207&133&85&56&15&6&69&46&30&20&9&3&1&0&24&11&5&5\\ 
\\ 
$\phi_{2100,16}$&2100&945&405&270&163&105&50&60&37&22&16&5&2&11&6&4&2&1&0&0&0&1&0&0&0\\ 
\\ 
$\phi_{2100,28}$&2100&1155&615&480&317&240&160&158&116&82&73&40&34&54&37&32&20&16&13&6&4&16&8&3&6\\ 
\\ 
$\phi_{4200,18}$&4200&2100&1060&660&540&340&140&278&176&114&75&20&9&92&59&40&27&11&6&2&0&31&14&7&6\\ 
\\ 
$\phi_{2240,10}$&2240&952&408&204&176&89&20&76&39&20&9&0&0&17&9&4&2&0&0&0&0&4&1&0&0\\ 
\\ 
$\phi_{2240,28}$&2240&1288&744&540&432&315&180&252&185&136&107&50&32&109&81&64&48&31&20&12&4&48&29&18&19\\ 
\\ 
$\phi_{4480,16}$&4480&2240&1120&720&560&360&160&280&180&116&80&24&12&90&58&40&26&12&6&2&0&29&13&6&6\\ 
\\ 
$\phi_{2268,10}$&2268&945&381&216&147&81&26&54&28&15&8&1&0&9&4&2&1&0&0&0&0&1&0&0&0\\ 
\\ 
$\phi_{2268,30}$&2268&1323&759&594&429&330&220&240&181&135&117&66&54&99&72&62&44&33&27&15&9&39&23&13&17\\ 
\\ 
$\phi_{4536,18}$&4536&2268&1116&756&540&360&180&258&168&108&81&30&15&78&48&36&21&12&6&3&0&22&9&3&5\\ 
\\ 
$\phi_{2835,14}$&2835&1323&627&351&303&171&55&150&86&48&28&5&0&45&26&15&8&3&0&0&0&15&5&1&2\\ 
\\ 
$\phi_{2835,22}$&2835&1512&816&540&444&300&140&243&167&117&81&30&9&93&66&47&34&18&6&6&0&37&20&11&11\\ 
\\ 
$\phi_{5670,18}$&5670&2835&1395&945&675&450&225&321&210&135&102&36&21&96&60&45&27&15&9&3&0&26&11&4&6\\ 
\\ 
$\phi_{3200,16}$&3200&1520&720&440&336&210&80&152&96&64&38&10&0&41&27&16&12&4&0&1&0&10&4&2&1\\ 
\\ 
$\phi_{3200,22}$&3200&1680&880&600&464&310&160&248&164&104&82&30&24&89&57&44&26&16&12&3&2&32&15&6&9\\ 
\\ 
$\phi_{4096,12}$&4096&1792&768&448&320&184&64&128&72&40&24&4&0&26&14&8&4&1&0&0&0&4&1&0&0\\ 
\\ 
$\phi_{4096,26}$&4096&2304&1280&960&704&520&320&384&280&200&168&84&64&150&106&88&60&43&32&16&8&56&31&16&22\\ 
\\ 
$\phi_{4200,12}$&4200&1890&850&480&382&215&70&172&96&53&31&5&0&43&23&14&7&2&0&0&0&10&3&1&1\\ 
\\ 
$\phi_{4200,24}$&4200&2310&1270&900&698&495&270&384&272&193&149&65&38&150&106&82&58&36&22&13&3&59&32&17&20\\ 
\\ 
$\phi_{6075,14}$&6075&2835&1305&810&594&360&150&267&159&90&63&15&9&69&39&26&13&6&3&0&0&16&5&1&2\\ 
\\ 
$\phi_{6075,22}$&6075&3240&1710&1215&891&630&345&456&321&225&174&75&45&159&111&85&59&36&21&12&3&52&27&14&16\\ 
\\ 
$\phi_{8,1}$&8&1&0&0&0&0&0&0&0&0&0&0&0&0&0&0&0&0&0&0&0&0&0&0&0\\ 
\\ 
$\phi_{8,91}$&8&7&6&6&5&5&5&4&4&4&4&4&4&3&3&3&3&3&3&3&3&2&2&2&2\\ 
\\ 
$\phi_{56,19}$&56&21&6&6&1&1&1&0&0&0&0&0&0&0&0&0&0&0&0&0&0&0&0&0&0\\ 
\\ 
$\phi_{56,49}$&56&35&20&20&10&10&10&4&4&4&4&4&4&1&1&1&1&1&1&1&1&0&0&0&0\\ 
\\ 
$\phi_{112,3}$&112&28&6&1&1&0&0&0&0&0&0&0&0&0&0&0&0&0&0&0&0&0&0&0&0\\ 
\\ 
$\phi_{112,63}$&112&84&62&57&45&41&36&32&29&26&25&21&20&20&18&17&15&14&13&11&10&12&10&8&9\\ 
\\ 
$\phi_{160,7}$&160&48&12&6&2&1&0&0&0&0&0&0&0&0&0&0&0&0&0&0&0&0&0&0&0\\ 
\\ 
$\phi_{160,55}$&160&112&76&70&50&45&40&32&28&24&24&20&20&17&14&14&11&11&11&8&8&8&6&4&6\\ 
\\ 
$\phi_{448,25}$&448&224&104&84&44&34&24&16&12&8&8&4&4&3&2&2&1&1&1&0&0&0&0&0&0\\ 
\\ 
$\phi_{400,7}$&400&140&50&15&19&5&0&8&2&0&0&0&0&1&0&0&0&0&0&0&0&0&0&0&0\\ 
\\ 
$\phi_{400,43}$&400&260&170&135&111&90&60&72&59&50&41&25&16&38&32&27&24&17&12&11&4&20&15&12&11\\ 
\\ 
$\phi_{448,9}$&448&168&64&24&24&10&0&8&4&2&0&0&0&1&1&0&0&0&0&0&0&0&0&0&0\\ 
\\ 
$\phi_{448,39}$&448&280&176&136&112&86&56&72&56&42&36&20&16&37&29&24&18&14&10&6&4&20&13&8&10\\ 
\\ 
$\phi_{560,5}$&560&182&56&20&16&5&0&4&1&0&0&0&0&0&0&0&0&0&0&0&0&0&0&0&0\\ 
\\ 
$\phi_{560,47}$&560&378&252&216&166&141&110&108&91&76&70&50&44&58&48&44&36&31&27&20&15&30&22&16&19\\ 
\\ 
$\phi_{1344,19}$&1344&672&344&204&180&110&40&96&60&40&24&4&4&33&22&14&11&3&3&0&0&12&6&4&2\\ 
\\ 
$\phi_{840,13}$&840&357&148&84&58&34&10&20&12&8&4&0&0&3&2&1&1&0&0&0&0&0&0&0&0\\ 
\\ 
$\phi_{840,31}$&840&483&274&210&155&115&75&88&64&44&40&20&20&36&25&22&14&11&11&3&3&14&8&4&6\\ 
\\ 
$\phi_{1008,9}$&1008&378&132&72&42&21&6&12&5&2&1&0&0&1&0&0&0&0&0&0&0&0&0&0&0\\ 
\\ 
$\phi_{1008,39}$&1008&630&384&324&228&189&144&132&107&86&79&54&48&59&46&42&32&27&24&16&12&24&16&10&13\\ 
\\ 
$\phi_{2016,19}$&2016&1008&516&306&270&165&60&144&91&58&35&9&0&51&34&21&14&6&0&1&0&20&9&4&4\\ 
\\ 
$\phi_{1296,13}$&1296&540&210&135&75&45&20&24&13&6&5&1&0&3&1&1&0&0&0&0&0&0&0&0&0\\ 
\\ 
$\phi_{1296,33}$&1296&756&426&351&231&186&136&120&94&72&66&42&36&45&33&30&21&18&15&9&6&14&8&4&7\\ 
\\ 
$\phi_{1400,11}$&1400&595&260&120&118&55&10&56&27&12&5&0&0&14&7&3&1&0&0&0&0&4&1&0&0\\ 
\\ 
$\phi_{1400,29}$&1400&805&470&330&277&200&105&164&121&92&67&30&12&73&56&42&34&20&9&9&1&34&21&14&13\\ 
\\ 
$\phi_{1400,7}$&1400&525&190&90&65&30&5&20&9&4&1&0&0&2&1&0&0&0&0&0&0&0&0&0&0\\ 
\\ 
$\phi_{1400,37}$&1400&875&540&440&330&265&190&200&159&124&111&70&60&95&74&65&49&40&33&21&15&44&29&18&23\\ 
\\ 
$\phi_{2400,17}$&2400&1140&520&360&228&150&80&96&60&34&29&10&8&23&12&10&4&3&2&0&0&4&1&0&1\\ 
\\ 
$\phi_{2400,23}$&2400&1260&640&480&312&230&140&144&104&74&61&30&20&43&30&24&16&11&7&4&1&10&5&2&3\\ 
\\ 
$\phi_{2800,13}$&2800&1260&550&345&233&140&60&96&55&30&21&5&4&21&11&7&3&1&1&0&0&4&1&0&0\\ 
\\ 
$\phi_{2800,25}$&2800&1540&830&625&437&325&200&224&164&120&98&50&36&80&57&46&33&22&15&9&4&26&14&8&9\\ 
\\ 
$\phi_{5600,19}$&5600&2800&1380&930&670&445&220&320&209&134&101&35&20&96&60&45&27&15&9&3&0&26&11&4&6\\ 
\\ 
$\phi_{3240,9}$&3240&1323&534&270&213&105&25&84&40&18&9&0&0&15&6&3&1&0&0&0&0&2&0&0&0\\ 
\\ 
$\phi_{3240,31}$&3240&1917&1128&864&660&504&320&384&292&222&185&100&72&168&127&106&80&57&42&27&12&72&45&28&32\\ 
\\ 
$\phi_{3360,13}$&3360&1512&676&390&302&170&60&136&75&38&25&5&0&34&17&11&4&2&0&0&0&8&2&0&1\\ 
\\ 
$\phi_{3360,25}$&3360&1848&1012&726&550&396&220&296&213&154&119&55&28&113&81&63&45&29&15&12&2&42&23&12&15\\ 
\\ 
$\phi_{7168,17}$&7168&3584&1792&1152&896&576&256&448&288&184&128&40&16&144&92&64&40&20&8&4&0&46&20&8&10\\ 
\\ 
$\phi_{4096,11}$&4096&1792&768&448&320&184&64&128&72&40&24&4&0&26&14&8&4&1&0&0&0&4&1&0&0\\ 
\\ 
$\phi_{4096,27}$&4096&2304&1280&960&704&520&320&384&280&200&168&84&64&150&106&88&60&43&32&16&8&56&31&16&22\\ 
\\ 
$\phi_{4200,15}$&4200&1995&950&570&457&270&105&224&131&74&50&10&8&66&37&25&13&5&4&0&0&20&7&2&3\\ 
\\ 
$\phi_{4200,21}$&4200&2205&1160&780&608&415&200&316&217&154&108&40&16&111&78&56&42&21&10&7&0&38&20&12&10\\ 
\\ 
$\phi_{4536,13}$&4536&2079&960&540&444&255&80&204&119&72&39&6&0&54&33&18&12&3&0&0&0&14&5&2&1\\ 
\\ 
$\phi_{4536,23}$&4536&2457&1338&918&735&504&259&408&281&192&145&56&36&159&110&83&57&33&21&9&3&64&34&18&20\\ 
\\ 
$\phi_{5600,15}$&5600&2660&1240&800&564&360&160&248&156&98&67&20&8&64&39&26&16&7&2&1&0&14&5&2&2\\ 
\\ 
$\phi_{5600,21}$&5600&2940&1520&1080&776&540&300&392&268&178&143&60&44&132&87&68&42&27&19&7&3&42&20&8&12\\ 
\end{longtable} 

\begin{longtable}{c|p{0.275cm}p{0.275cm}p{0.275cm}p{0.275cm}p{0.275cm}p{0.275cm}p{0.275cm}p{0.275cm}p{0.275cm}p{0.275cm}p{0.275cm}p{0.275cm}p{0.275cm}p{0.275cm}p{0.275cm}p{0.275cm}|p{0.275cm}p{0.275cm}p{0.275cm}p{0.275cm}p{0.275cm}p{0.275cm}p{0.275cm}p{0.275cm}p{0.275cm}} 
&$L_{26}$&$L_{27}$&$L_{28}$&$L_{29}$&$L_{30}$&$L_{31}$&$L_{32}$&$L_{33}$&$L_{34}$&$L_{35}$&$L_{36}$&$L_{37}$&$L_{38}$&$L_{39}$&$L_{40}$&$L_{41}$&${\widehat{L}}_{1}$&${\widehat{L}}_{2}$&${\widehat{L}}_{3}$&${\widehat{L}}_{4}$&${\widehat{L}}_{5}$&${\widehat{L}}_{6}$&${\widehat{L}}_{7}$&${\widehat{L}}_{8}$&${\widehat{L}}_{9}$\\ 
Irrep label&$A_4+A_2$&$D_4+A_2$&$A_5+A_1$&$D_5+A_1$&$A_6$&$D_6$&$E_6$&$A_4+A_2+A_1$&$A_4+A_3$&$D_5+A_2$&$A_6+A_1$&$E_6+A_1$&$A_7$&$D_7$&$E_7$&$E_8$&$4A_1$&$5A_1$&$A_3+2A_1$&$A_2+4A_1$&$3A_2$&$A_3+3A_1$&$2A_3$&$D_4+2A_1$&$A_5+A_1$\\\hline 
\endhead 
$\phi_{1,0}$&0&0&0&0&0&0&0&0&0&0&0&0&0&0&0&0&0&0&0&0&0&0&0&0&0\\ 
\\ 
$\phi_{1,120}$&1&1&1&1&1&1&1&1&1&1&1&1&1&1&1&1&1&1&1&1&1&1&1&1&1\\ 
\\ 
$\phi_{28,8}$&0&0&0&0&0&0&0&0&0&0&0&0&0&0&0&0&0&0&0&0&0&0&0&0&0\\ 
\\ 
$\phi_{28,68}$&1&1&1&1&1&1&1&0&0&0&0&0&0&0&0&0&6&3&3&1&1&1&1&1&1\\ 
\\ 
$\phi_{35,2}$&0&0&0&0&0&0&0&0&0&0&0&0&0&0&0&0&0&0&0&0&0&0&0&0&0\\ 
\\ 
$\phi_{35,74}$&4&4&4&4&3&3&3&3&2&2&2&2&1&1&1&0&13&10&8&7&5&6&4&5&4\\ 
\\ 
$\phi_{70,32}$&0&0&0&0&0&0&0&0&0&0&0&0&0&0&0&0&1&0&0&0&0&0&0&0&0\\ 
\\ 
$\phi_{50,8}$&0&0&0&0&0&0&0&0&0&0&0&0&0&0&0&0&2&1&0&0&0&0&0&0&0\\ 
\\ 
$\phi_{50,56}$&3&3&2&1&2&1&0&2&2&1&1&0&1&1&0&0&14&10&7&6&4&5&5&3&3\\ 
\\ 
$\phi_{84,4}$&0&0&0&0&0&0&0&0&0&0&0&0&0&0&0&0&1&0&0&0&0&0&0&0&0\\ 
\\ 
$\phi_{84,64}$&7&6&6&5&4&3&3&5&4&3&3&2&2&1&1&0&27&20&15&13&10&11&8&8&7\\ 
\\ 
$\phi_{168,24}$&0&1&0&0&0&0&0&0&0&0&0&0&0&0&0&0&16&10&3&4&1&2&2&1&0\\ 
\\ 
$\phi_{175,12}$&0&0&0&0&0&0&0&0&0&0&0&0&0&0&0&0&6&3&0&1&1&0&0&0&0\\ 
\\ 
$\phi_{175,36}$&3&1&2&0&1&0&0&2&2&0&1&0&1&0&0&0&25&16&7&8&7&5&3&1&1\\ 
\\ 
$\phi_{210,4}$&0&0&0&0&0&0&0&0&0&0&0&0&0&0&0&0&0&0&0&0&0&0&0&0&0\\ 
\\ 
$\phi_{210,52}$&7&7&6&5&4&3&2&4&3&2&2&1&1&1&0&0&44&29&19&16&10&12&7&8&5\\ 
\\ 
$\phi_{420,20}$&1&0&0&0&0&0&0&1&1&0&0&0&0&0&0&0&32&19&4&8&4&3&1&0&0\\ 
\\ 
$\phi_{300,8}$&0&0&0&0&0&0&0&0&0&0&0&0&0&0&0&0&2&0&0&0&0&0&0&0&0\\ 
\\ 
$\phi_{300,44}$&5&5&4&4&1&1&1&4&2&2&1&1&0&0&0&0&52&35&19&19&10&13&5&8&4\\ 
\\ 
$\phi_{350,14}$&0&0&0&0&0&0&0&0&0&0&0&0&0&0&0&0&3&0&0&0&0&0&0&0&0\\ 
\\ 
$\phi_{350,38}$&1&1&1&1&0&0&0&0&0&0&0&0&0&0&0&0&33&15&9&4&2&3&1&2&1\\ 
\\ 
$\phi_{525,12}$&0&0&0&0&0&0&0&0&0&0&0&0&0&0&0&0&13&3&1&0&1&0&0&0&0\\ 
\\ 
$\phi_{525,36}$&4&2&2&1&1&0&1&1&1&0&0&0&0&0&0&0&62&31&19&11&9&8&6&3&4\\ 
\\ 
$\phi_{567,6}$&0&0&0&0&0&0&0&0&0&0&0&0&0&0&0&0&3&0&0&0&0&0&0&0&0\\ 
\\ 
$\phi_{567,46}$&11&9&9&7&5&3&3&5&3&2&2&1&1&0&0&0&96&57&37&27&18&20&11&12&9\\ 
\\ 
$\phi_{1134,20}$&1&1&0&0&0&0&0&0&0&0&0&0&0&0&0&0&63&27&9&6&3&3&1&1&0\\ 
\\ 
$\phi_{700,16}$&0&0&0&0&0&0&0&0&0&0&0&0&0&0&0&0&42&23&5&7&2&3&1&1&0\\ 
\\ 
$\phi_{700,28}$&3&2&1&0&1&0&0&1&1&0&0&0&0&0&0&0&72&38&18&14&7&9&8&2&3\\ 
\\ 
$\phi_{700,6}$&0&0&0&0&0&0&0&0&0&0&0&0&0&0&0&0&9&3&0&0&0&0&0&0&0\\ 
\\ 
$\phi_{700,42}$&15&13&11&7&6&3&1&10&7&4&4&1&2&1&0&0&123&81&45&43&27&30&16&16&10\\ 
\\ 
$\phi_{1400,20}$&0&1&0&0&0&0&0&0&0&0&0&0&0&0&0&0&72&34&9&10&2&4&0&1&0\\ 
\\ 
$\phi_{840,14}$&0&0&0&0&0&0&0&0&0&0&0&0&0&0&0&0&36&15&3&3&1&1&1&0&0\\ 
\\ 
$\phi_{840,26}$&4&4&1&1&0&0&0&3&2&1&0&0&0&0&0&0&88&55&19&25&12&13&5&5&1\\ 
\\ 
$\phi_{1680,22}$&0&0&0&0&0&0&0&0&0&0&0&0&0&0&0&0&76&30&10&6&1&3&0&1&0\\ 
\\ 
$\phi_{972,12}$&0&0&0&0&0&0&0&0&0&0&0&0&0&0&0&0&36&15&3&3&0&1&1&0&0\\ 
\\ 
$\phi_{972,32}$&7&6&4&2&1&1&0&5&3&1&1&0&0&0&0&0&126&78&36&36&18&23&12&10&6\\ 
\\ 
$\phi_{1050,10}$&0&0&0&0&0&0&0&0&0&0&0&0&0&0&0&0&35&16&2&3&0&1&0&0&0\\ 
\\ 
$\phi_{1050,34}$&11&7&7&2&4&1&0&6&4&1&2&0&1&0&0&0&141&83&43&37&23&25&15&9&9\\ 
\\ 
$\phi_{2100,20}$&1&0&0&0&0&0&0&0&0&0&0&0&0&0&0&0&116&49&18&12&5&6&3&1&1\\ 
\\ 
$\phi_{1344,8}$&0&0&0&0&0&0&0&0&0&0&0&0&0&0&0&0&24&6&1&0&0&0&0&0&0\\ 
\\ 
$\phi_{1344,38}$&17&13&12&8&5&2&2&10&6&3&3&1&1&0&0&0&200&122&67&58&34&40&20&20&14\\ 
\\ 
$\phi_{2688,20}$&2&0&1&0&0&0&0&1&0&0&0&0&0&0&0&0&176&88&28&28&10&14&4&2&2\\ 
\\ 
$\phi_{1400,8}$&0&0&0&0&0&0&0&0&0&0&0&0&0&0&0&0&30&12&0&2&0&0&0&0&0\\ 
\\ 
$\phi_{1400,32}$&12&8&8&3&4&1&0&6&4&1&2&0&1&0&0&0&166&94&46&40&25&25&12&9&6\\ 
\\ 
$\phi_{1575,10}$&0&0&0&0&0&0&0&0&0&0&0&0&0&0&0&0&27&9&0&1&0&0&0&0&0\\ 
\\ 
$\phi_{1575,34}$&10&8&7&4&3&1&0&4&2&1&1&0&0&0&0&0&174&93&48&36&21&23&10&10&6\\ 
\\ 
$\phi_{3150,18}$&4&1&1&0&0&0&0&2&1&0&0&0&0&0&0&0&201&102&27&34&18&14&4&1&0\\ 
\\ 
$\phi_{2100,16}$&0&0&0&0&0&0&0&0&0&0&0&0&0&0&0&0&63&21&5&3&0&1&0&0&0\\ 
\\ 
$\phi_{2100,28}$&3&2&2&1&0&0&0&1&0&0&0&0&0&0&0&0&161&77&34&24&10&14&4&5&3\\ 
\\ 
$\phi_{4200,18}$&4&3&1&0&0&0&0&2&1&0&0&0&0&0&0&0&284&148&43&50&19&23&8&4&2\\ 
\\ 
$\phi_{2240,10}$&0&0&0&0&0&0&0&0&0&0&0&0&0&0&0&0&72&30&3&6&2&1&0&0&0\\ 
\\ 
$\phi_{2240,28}$&15&10&8&3&3&0&0&9&6&2&2&0&1&0&0&0&248&146&61&64&36&37&16&12&6\\ 
\\ 
$\phi_{4480,16}$&4&2&1&0&0&0&0&2&1&0&0&0&0&0&0&0&272&136&36&44&18&18&4&2&0\\ 
\\ 
$\phi_{2268,10}$&0&0&0&0&0&0&0&0&0&0&0&0&0&0&0&0&51&18&1&3&0&0&0&0&0\\ 
\\ 
$\phi_{2268,30}$&11&9&7&4&2&1&0&6&3&1&1&0&0&0&0&0&237&132&60&54&27&32&12&13&6\\ 
\\ 
$\phi_{4536,18}$&2&1&1&0&0&0&0&1&0&0&0&0&0&0&0&0&252&120&33&36&12&15&2&2&0\\ 
\\ 
$\phi_{2835,14}$&1&0&0&0&0&0&0&1&0&0&0&0&0&0&0&0&147&75&14&24&9&8&1&0&0\\ 
\\ 
$\phi_{2835,22}$&8&3&4&0&1&0&0&5&3&0&1&0&0&0&0&0&240&132&45&51&27&26&10&4&3\\ 
\\ 
$\phi_{5670,18}$&3&2&1&0&0&0&0&1&0&0&0&0&0&0&0&0&315&147&42&42&15&18&3&3&0\\ 
\\ 
$\phi_{3200,16}$&1&0&0&0&0&0&0&0&0&0&0&0&0&0&0&0&160&68&20&16&7&7&4&0&1\\ 
\\ 
$\phi_{3200,22}$&5&3&2&1&0&0&0&3&1&0&0&0&0&0&0&0&256&140&48&52&21&27&8&8&3\\ 
\\ 
$\phi_{4096,12}$&0&0&0&0&0&0&0&0&0&0&0&0&0&0&0&0&128&48&8&8&2&2&0&0&0\\ 
\\ 
$\phi_{4096,26}$&14&10&8&4&2&0&0&7&3&1&1&0&0&0&0&0&384&208&88&80&38&46&16&16&8\\ 
\\ 
$\phi_{4200,12}$&0&0&0&0&0&0&0&0&0&0&0&0&0&0&0&0&170&77&13&19&5&6&0&0&0\\ 
\\ 
$\phi_{4200,24}$&14&9&7&2&2&0&0&8&4&1&1&0&0&0&0&0&382&211&81&83&42&45&16&12&7\\ 
\\ 
$\phi_{6075,14}$&1&0&0&0&0&0&0&0&0&0&0&0&0&0&0&0&270&120&27&30&9&11&1&1&0\\ 
\\ 
$\phi_{6075,22}$&11&6&5&1&1&0&0&4&2&0&0&0&0&0&0&0&459&228&87&75&36&40&15&9&6\\ 
\\ 
$\phi_{8,1}$&0&0&0&0&0&0&0&0&0&0&0&0&0&0&0&0&0&0&0&0&0&0&0&0&0\\ 
\\ 
$\phi_{8,91}$&2&2&2&2&2&2&2&1&1&1&1&1&1&1&1&0&4&3&3&2&2&2&2&2&2\\ 
\\ 
$\phi_{56,19}$&0&0&0&0&0&0&0&0&0&0&0&0&0&0&0&0&0&0&0&0&0&0&0&0&0\\ 
\\ 
$\phi_{56,49}$&0&0&0&0&0&0&0&0&0&0&0&0&0&0&0&0&4&1&1&0&0&0&0&0&0\\ 
\\ 
$\phi_{112,3}$&0&0&0&0&0&0&0&0&0&0&0&0&0&0&0&0&0&0&0&0&0&0&0&0&0\\ 
\\ 
$\phi_{112,63}$&8&7&7&6&5&4&4&5&4&3&3&2&2&1&1&0&32&22&17&13&11&11&8&8&7\\ 
\\ 
$\phi_{160,7}$&0&0&0&0&0&0&0&0&0&0&0&0&0&0&0&0&0&0&0&0&0&0&0&0&0\\ 
\\ 
$\phi_{160,55}$&4&4&4&4&2&2&2&2&1&1&1&1&0&0&0&0&32&20&14&10&6&8&4&6&4\\ 
\\ 
$\phi_{448,25}$&0&0&0&0&0&0&0&0&0&0&0&0&0&0&0&0&16&4&2&0&0&0&0&0&0\\ 
\\ 
$\phi_{400,7}$&0&0&0&0&0&0&0&0&0&0&0&0&0&0&0&0&8&4&0&1&0&0&0&0&0\\ 
\\ 
$\phi_{400,43}$&10&8&7&3&5&2&0&6&5&2&3&0&2&1&0&0&72&46&27&24&17&17&12&8&7\\ 
\\ 
$\phi_{448,9}$&0&0&0&0&0&0&0&0&0&0&0&0&0&0&0&0&8&2&0&0&1&0&0&0&0\\ 
\\ 
$\phi_{448,39}$&8&5&5&3&2&0&1&6&4&2&2&1&1&0&0&0&72&46&24&24&17&16&8&6&5\\ 
\\ 
$\phi_{560,5}$&0&0&0&0&0&0&0&0&0&0&0&0&0&0&0&0&4&1&0&0&0&0&0&0&0\\ 
\\ 
$\phi_{560,47}$&15&13&12&9&7&4&3&9&6&4&4&2&2&1&0&0&108&69&44&36&24&27&16&16&12\\ 
\\ 
$\phi_{1344,19}$&2&2&0&0&0&0&0&1&1&0&0&0&0&0&0&0&96&52&14&18&8&8&4&2&0\\ 
\\ 
$\phi_{840,13}$&0&0&0&0&0&0&0&0&0&0&0&0&0&0&0&0&20&5&1&0&0&0&0&0&0\\ 
\\ 
$\phi_{840,31}$&4&4&2&2&0&0&0&2&1&1&0&0&0&0&0&0&88&50&22&20&10&12&4&6&2\\ 
\\ 
$\phi_{1008,9}$&0&0&0&0&0&0&0&0&0&0&0&0&0&0&0&0&12&3&0&0&0&0&0&0&0\\ 
\\ 
$\phi_{1008,39}$&9&8&7&5&3&2&1&4&2&1&1&0&0&0&0&0&132&75&42&32&18&22&10&12&7\\ 
\\ 
$\phi_{2016,19}$&3&0&1&0&0&0&0&2&1&0&0&0&0&0&0&0&144&78&21&29&15&13&4&0&1\\ 
\\ 
$\phi_{1296,13}$&0&0&0&0&0&0&0&0&0&0&0&0&0&0&0&0&24&6&1&0&0&0&0&0&0\\ 
\\ 
$\phi_{1296,33}$&4&3&3&2&1&0&0&1&0&0&0&0&0&0&0&0&120&60&30&21&9&13&4&6&3\\ 
\\ 
$\phi_{1400,11}$&0&0&0&0&0&0&0&0&0&0&0&0&0&0&0&0&56&28&3&8&2&2&0&0&0\\ 
\\ 
$\phi_{1400,29}$&11&6&6&1&3&0&0&7&5&1&2&0&1&0&0&0&164&97&42&44&26&26&14&6&6\\ 
\\ 
$\phi_{1400,7}$&0&0&0&0&0&0&0&0&0&0&0&0&0&0&0&0&20&5&0&0&0&0&0&0&0\\ 
\\ 
$\phi_{1400,37}$&17&13&12&8&5&2&2&10&6&3&3&1&1&0&0&0&200&120&65&56&34&38&18&18&12\\ 
\\ 
$\phi_{2400,17}$&0&0&0&0&0&0&0&0&0&0&0&0&0&0&0&0&96&38&10&8&1&3&0&0&0\\ 
\\ 
$\phi_{2400,23}$&2&1&1&0&0&0&0&0&0&0&0&0&0&0&0&0&144&62&24&16&7&8&2&2&1\\ 
\\ 
$\phi_{2800,13}$&0&0&0&0&0&0&0&0&0&0&0&0&0&0&0&0&96&38&7&7&2&2&0&0&0\\ 
\\ 
$\phi_{2800,25}$&6&4&3&1&1&0&0&2&1&0&0&0&0&0&0&0&224&112&46&38&17&21&8&6&3\\ 
\\ 
$\phi_{5600,19}$&3&2&1&0&0&0&0&1&0&0&0&0&0&0&0&0&320&150&45&43&15&19&4&4&1\\ 
\\ 
$\phi_{3240,9}$&0&0&0&0&0&0&0&0&0&0&0&0&0&0&0&0&84&33&3&6&0&1&0&0&0\\ 
\\ 
$\phi_{3240,31}$&24&18&15&7&6&2&0&13&8&3&3&0&1&0&0&0&384&222&106&96&54&60&28&24&15\\ 
\\ 
$\phi_{3360,13}$&0&0&0&0&0&0&0&0&0&0&0&0&0&0&0&0&136&62&11&16&4&5&0&0&0\\ 
\\ 
$\phi_{3360,25}$&10&5&6&1&2&0&0&5&2&0&1&0&0&0&0&0&296&158&63&60&30&33&12&8&6\\ 
\\ 
$\phi_{7168,17}$&6&2&2&0&0&0&0&3&1&0&0&0&0&0&0&0&448&224&64&72&28&32&8&4&2\\ 
\\ 
$\phi_{4096,11}$&0&0&0&0&0&0&0&0&0&0&0&0&0&0&0&0&128&48&8&8&2&2&0&0&0\\ 
\\ 
$\phi_{4096,27}$&14&10&8&4&2&0&0&7&3&1&1&0&0&0&0&0&384&208&88&80&38&46&16&16&8\\ 
\\ 
$\phi_{4200,15}$&1&1&0&0&0&0&0&1&0&0&0&0&0&0&0&0&224&112&25&34&11&13&2&2&0\\ 
\\ 
$\phi_{4200,21}$&8&5&3&0&1&0&0&3&2&0&0&0&0&0&0&0&316&163&56&56&27&28&12&6&3\\ 
\\ 
$\phi_{4536,13}$&1&0&0&0&0&0&0&0&0&0&0&0&0&0&0&0&204&93&18&24&9&8&2&0&0\\ 
\\ 
$\phi_{4536,23}$&14&9&6&2&1&0&0&9&5&1&1&0&0&0&0&0&408&228&83&90&45&48&18&12&6\\ 
\\ 
$\phi_{5600,15}$&1&0&0&0&0&0&0&0&0&0&0&0&0&0&0&0&248&104&26&24&8&9&2&0&0\\ 
\\ 
$\phi_{5600,21}$&7&4&3&1&0&0&0&3&1&0&0&0&0&0&0&0&392&196&68&64&28&32&8&8&3\\ 
\end{longtable} 

\begin{longtable}{c|p{0.275cm}p{0.275cm}p{0.275cm}p{0.275cm}p{0.275cm}p{0.275cm}p{0.275cm}p{0.275cm}p{0.275cm}p{0.275cm}p{0.275cm}p{0.275cm}p{0.275cm}p{0.275cm}p{0.275cm}p{0.275cm}p{0.275cm}|p{0.275cm}p{0.275cm}p{0.275cm}p{0.275cm}p{0.275cm}p{0.275cm}p{0.275cm}p{0.275cm}p{0.275cm}p{0.275cm}} 
&${\widehat{L}}_{10}$&${\widehat{L}}_{11}$&${\widehat{L}}_{12}$&${\widehat{L}}_{13}$&${\widehat{L}}_{14}$&${\widehat{L}}_{15}$&${\widehat{L}}_{16}$&${\widehat{L}}_{17}$&${\widehat{L}}_{18}$&${\widehat{L}}_{19}$&${\widehat{L}}_{20}$&${\widehat{L}}_{21}$&${\widehat{L}}_{22}$&${\widehat{L}}_{23}$&${\widehat{L}}_{24}$&${\widehat{L}}_{25}$&${\widehat{L}}_{26}$&-&-&-&-&-&-&-&-&-&-\\ 
Irrep label&$3A_2+A_1$&$A_3+A_2+2A_1$&$2A_3+A_1$&$D_4+A_3$&$A_5+2A_1$&$A_5+A_2$&$D_5+2A_1$&$D_6+A_1$&$A_7$&$2A_4$&$A_5+A_2+A_1$&$D_5+A_3$&$E_6+A_2$&$A_7+A_1$&$E_7+A_1$&$A_8$&$D_8$&$6A_1$&$7A_1$&$A_3+4A_1$&$D_4+3A_1$&$8A_1$&$4A_2$&$2A_3+2A_1$&$D_4+4A_1$&$2D_4$&$D_6+2A_1$\\\hline 
\endhead 
$\phi_{1,0}$&0&0&0&0&0&0&0&0&0&0&0&0&0&0&0&0&0&0&0&0&0&0&0&0&0&0&0\\ 
\\ 
$\phi_{1,120}$&1&1&1&1&1&1&1&1&1&1&1&1&1&1&1&1&1&1&1&1&1&1&1&1&1&1&1\\ 
\\ 
$\phi_{28,8}$&0&0&0&0&0&0&0&0&0&0&0&0&0&0&0&0&0&0&0&0&0&0&0&0&0&0&0\\ 
\\ 
$\phi_{28,68}$&0&0&0&0&0&0&0&0&0&0&0&0&0&0&0&0&0&1&0&0&0&0&0&0&0&0&0\\ 
\\ 
$\phi_{35,2}$&0&0&0&0&0&0&0&0&0&0&0&0&0&0&0&0&0&0&0&0&0&0&0&0&0&0&0\\ 
\\ 
$\phi_{35,74}$&4&4&3&2&3&2&3&2&1&1&2&1&1&1&1&0&0&8&7&5&4&7&3&3&4&1&2\\ 
\\ 
$\phi_{70,32}$&0&0&0&0&0&0&0&0&0&0&0&0&0&0&0&0&0&0&0&0&0&0&0&0&0&0&0\\ 
\\ 
$\phi_{50,8}$&0&0&0&0&0&0&0&0&0&0&0&0&0&0&0&0&0&1&1&0&0&1&0&0&0&0&0\\ 
\\ 
$\phi_{50,56}$&3&3&3&3&2&1&1&1&2&1&1&1&0&1&0&0&1&8&7&4&3&7&2&2&3&3&1\\ 
\\ 
$\phi_{84,4}$&0&0&0&0&0&0&0&0&0&0&0&0&0&0&0&0&0&0&0&0&0&0&0&0&0&0&0\\ 
\\ 
$\phi_{84,64}$&7&7&6&4&5&4&4&3&3&2&3&2&1&2&1&1&1&16&14&9&7&14&4&5&7&3&3\\ 
\\ 
$\phi_{168,24}$&1&1&1&1&0&0&0&0&0&0&0&0&0&0&0&0&0&8&7&2&1&7&1&1&1&1&0\\ 
\\ 
$\phi_{175,12}$&1&0&0&0&0&0&0&0&0&0&0&0&0&0&0&0&0&1&0&0&0&0&1&0&0&0&0\\ 
\\ 
$\phi_{175,36}$&4&3&2&1&1&2&0&0&0&1&1&0&0&0&0&1&0&10&7&3&1&7&2&1&1&1&0\\ 
\\ 
$\phi_{210,4}$&0&0&0&0&0&0&0&0&0&0&0&0&0&0&0&0&0&0&0&0&0&0&0&0&0&0&0\\ 
\\ 
$\phi_{210,52}$&6&6&4&3&3&2&3&1&0&1&1&1&0&0&0&0&0&19&14&8&5&14&2&3&5&2&1\\ 
\\ 
$\phi_{420,20}$&3&2&1&0&0&0&0&0&0&1&0&0&0&0&0&0&0&11&7&2&0&7&2&1&0&0&0\\ 
\\ 
$\phi_{300,8}$&0&0&0&0&0&0&0&0&0&0&0&0&0&0&0&0&0&0&0&0&0&0&0&0&0&0&0\\ 
\\ 
$\phi_{300,44}$&8&7&4&2&3&2&3&1&0&1&2&1&1&0&0&0&0&25&20&10&6&20&6&4&6&1&1\\ 
\\ 
$\phi_{350,14}$&0&0&0&0&0&0&0&0&0&0&0&0&0&0&0&0&0&0&0&0&0&0&0&0&0&0&0\\ 
\\ 
$\phi_{350,38}$&0&0&0&0&0&0&0&0&0&0&0&0&0&0&0&0&0&5&0&0&0&0&0&0&0&0&0\\ 
\\ 
$\phi_{525,12}$&0&0&0&0&0&0&0&0&0&0&0&0&0&0&0&0&0&1&0&0&0&0&0&0&0&0&0\\ 
\\ 
$\phi_{525,36}$&3&2&2&1&1&1&0&0&1&0&0&0&0&0&0&0&0&15&7&3&1&7&1&0&1&1&0\\ 
\\ 
$\phi_{567,6}$&0&0&0&0&0&0&0&0&0&0&0&0&0&0&0&0&0&0&0&0&0&0&0&0&0&0&0\\ 
\\ 
$\phi_{567,46}$&9&8&5&2&4&3&3&1&1&0&1&0&0&0&0&0&0&33&21&11&6&21&3&3&6&0&1\\ 
\\ 
$\phi_{1134,20}$&0&0&0&0&0&0&0&0&0&0&0&0&0&0&0&0&0&9&0&0&0&0&0&0&0&0&0\\ 
\\ 
$\phi_{700,16}$&1&1&1&0&0&0&0&0&0&0&0&0&0&0&0&0&0&17&14&3&1&14&0&1&1&0&0\\ 
\\ 
$\phi_{700,28}$&3&3&3&2&1&0&0&0&1&0&0&0&0&0&0&0&0&22&14&5&2&14&0&1&2&2&0\\ 
\\ 
$\phi_{700,6}$&0&0&0&0&0&0&0&0&0&0&0&0&0&0&0&0&0&1&0&0&0&0&0&0&0&0&0\\ 
\\ 
$\phi_{700,42}$&18&16&11&6&7&6&5&2&1&3&4&2&1&1&0&1&0&55&42&21&12&42&10&8&12&3&2\\ 
\\ 
$\phi_{1400,20}$&1&1&0&0&0&0&0&0&0&0&0&0&0&0&0&0&0&16&7&2&0&7&1&0&0&0&0\\ 
\\ 
$\phi_{840,14}$&0&0&0&0&0&0&0&0&0&0&0&0&0&0&0&0&0&5&0&0&0&0&0&0&0&0&0\\ 
\\ 
$\phi_{840,26}$&9&7&4&2&1&1&1&0&0&1&1&1&0&0&0&0&0&37&28&10&4&28&6&4&4&1&0\\ 
\\ 
$\phi_{1680,22}$&0&0&0&0&0&0&0&0&0&0&0&0&0&0&0&0&0&10&0&0&0&0&0&0&0&0&0\\ 
\\ 
$\phi_{972,12}$&0&0&0&0&0&0&0&0&0&0&0&0&0&0&0&0&0&9&6&1&0&6&0&0&0&0&0\\ 
\\ 
$\phi_{972,32}$&12&11&8&4&4&2&2&1&1&1&2&1&0&1&0&0&0&54&42&17&9&42&6&6&9&3&1\\ 
\\ 
$\phi_{1050,10}$&0&0&0&0&0&0&0&0&0&0&0&0&0&0&0&0&0&10&7&1&0&7&0&0&0&0&0\\ 
\\ 
$\phi_{1050,34}$&13&11&8&3&5&3&1&1&2&1&2&0&0&1&0&0&0&51&35&15&7&35&6&4&7&1&1\\ 
\\ 
$\phi_{2100,20}$&1&1&1&0&0&0&0&0&0&0&0&0&0&0&0&0&0&21&7&2&0&7&0&0&0&0&0\\ 
\\ 
$\phi_{1344,8}$&0&0&0&0&0&0&0&0&0&0&0&0&0&0&0&0&0&2&0&0&0&0&0&0&0&0&0\\ 
\\ 
$\phi_{1344,38}$&20&18&12&5&8&5&5&2&2&2&3&1&0&1&0&0&0&78&56&26&14&56&8&8&14&2&2\\ 
\\ 
$\phi_{2688,20}$&5&4&2&0&1&0&0&0&0&0&0&0&0&0&0&0&0&48&28&8&2&28&2&1&2&0&0\\ 
\\ 
$\phi_{1400,8}$&0&0&0&0&0&0&0&0&0&0&0&0&0&0&0&0&0&4&0&0&0&0&0&0&0&0&0\\ 
\\ 
$\phi_{1400,32}$&13&10&6&2&3&3&1&0&0&1&1&0&0&0&0&0&0&50&28&12&4&28&6&2&4&1&0\\ 
\\ 
$\phi_{1575,10}$&0&0&0&0&0&0&0&0&0&0&0&0&0&0&0&0&0&3&0&0&0&0&0&0&0&0&0\\ 
\\ 
$\phi_{1575,34}$&9&7&4&1&2&2&1&0&0&0&0&0&0&0&0&0&0&45&21&9&3&21&1&1&3&0&0\\ 
\\ 
$\phi_{3150,18}$&9&5&2&0&0&1&0&0&0&0&0&0&0&0&0&0&0&48&21&6&0&21&4&1&0&0&0\\ 
\\ 
$\phi_{2100,16}$&0&0&0&0&0&0&0&0&0&0&0&0&0&0&0&0&0&7&0&0&0&0&0&0&0&0&0\\ 
\\ 
$\phi_{2100,28}$&4&3&1&0&1&0&0&0&0&0&0&0&0&0&0&0&0&35&14&5&1&14&0&0&1&0&0\\ 
\\ 
$\phi_{4200,18}$&10&8&4&1&1&0&0&0&0&0&0&0&0&0&0&0&0&82&49&14&3&49&3&3&3&0&0\\ 
\\ 
$\phi_{2240,10}$&1&0&0&0&0&0&0&0&0&0&0&0&0&0&0&0&0&10&0&0&0&0&0&0&0&0&0\\ 
\\ 
$\phi_{2240,28}$&21&17&10&4&4&4&2&0&0&2&2&1&0&0&0&0&0&86&56&22&8&56&8&6&8&2&0\\ 
\\ 
$\phi_{4480,16}$&9&6&2&0&0&0&0&0&0&0&0&0&0&0&0&0&0&64&28&8&0&28&2&1&0&0&0\\ 
\\ 
$\phi_{2268,10}$&0&0&0&0&0&0&0&0&0&0&0&0&0&0&0&0&0&6&0&0&0&0&0&0&0&0&0\\ 
\\ 
$\phi_{2268,30}$&15&12&6&2&3&2&2&0&0&1&1&0&0&0&0&0&0&72&42&17&6&42&6&3&6&0&0\\ 
\\ 
$\phi_{4536,18}$&6&4&1&0&0&0&0&0&0&0&0&0&0&0&0&0&0&54&21&6&0&21&3&0&0&0&0\\ 
\\ 
$\phi_{2835,14}$&6&3&1&0&0&0&0&0&0&0&0&0&0&0&0&0&0&39&21&5&0&21&3&1&0&0&0\\ 
\\ 
$\phi_{2835,22}$&15&11&6&1&2&2&0&0&0&1&1&0&0&0&0&0&0&72&42&14&3&42&6&3&3&0&0\\ 
\\ 
$\phi_{5670,18}$&6&4&1&0&0&0&0&0&0&0&0&0&0&0&0&0&0&63&21&6&0&21&0&0&0&0&0\\ 
\\ 
$\phi_{3200,16}$&2&1&1&0&0&0&0&0&0&0&0&0&0&0&0&0&0&28&8&2&0&8&0&0&0&0&0\\ 
\\ 
$\phi_{3200,22}$&12&10&5&1&2&1&1&0&0&0&1&0&0&0&0&0&0&84&56&18&6&56&4&4&6&0&0\\ 
\\ 
$\phi_{4096,12}$&0&0&0&0&0&0&0&0&0&0&0&0&0&0&0&0&0&16&0&0&0&0&0&0&0&0&0\\ 
\\ 
$\phi_{4096,26}$&20&16&8&2&4&2&2&0&0&0&1&0&0&0&0&0&0&112&64&24&8&64&8&4&8&0&0\\ 
\\ 
$\phi_{4200,12}$&2&1&0&0&0&0&0&0&0&0&0&0&0&0&0&0&0&35&14&3&0&14&0&0&0&0&0\\ 
\\ 
$\phi_{4200,24}$&24&18&9&2&4&3&1&0&0&1&2&0&0&0&0&0&0&117&70&25&7&70&10&5&7&0&0\\ 
\\ 
$\phi_{6075,14}$&3&2&0&0&0&0&0&0&0&0&0&0&0&0&0&0&0&54&21&5&0&21&0&0&0&0&0\\ 
\\ 
$\phi_{6075,22}$&15&11&6&1&2&1&0&0&0&0&0&0&0&0&0&0&0&108&48&16&3&48&3&1&3&0&0\\ 
\\ 
$\phi_{8,1}$&0&0&0&0&0&0&0&0&0&0&0&0&0&0&0&0&0&0&0&0&0&0&0&0&0&0&0\\ 
\\ 
$\phi_{8,91}$&1&1&1&1&1&1&1&1&1&0&0&0&0&0&0&0&0&2&1&1&1&0&0&0&0&0&0\\ 
\\ 
$\phi_{56,19}$&0&0&0&0&0&0&0&0&0&0&0&0&0&0&0&0&0&0&0&0&0&0&0&0&0&0&0\\ 
\\ 
$\phi_{56,49}$&0&0&0&0&0&0&0&0&0&0&0&0&0&0&0&0&0&0&0&0&0&0&0&0&0&0&0\\ 
\\ 
$\phi_{112,3}$&0&0&0&0&0&0&0&0&0&0&0&0&0&0&0&0&0&0&0&0&0&0&0&0&0&0&0\\ 
\\ 
$\phi_{112,63}$&7&6&5&3&4&4&3&2&2&2&2&1&1&1&0&1&0&14&7&6&4&0&4&2&0&0&0\\ 
\\ 
$\phi_{160,7}$&0&0&0&0&0&0&0&0&0&0&0&0&0&0&0&0&0&0&0&0&0&0&0&0&0&0&0\\ 
\\ 
$\phi_{160,55}$&3&3&2&1&2&1&2&1&0&0&0&0&0&0&0&0&0&12&6&4&3&0&0&0&0&0&0\\ 
\\ 
$\phi_{448,25}$&0&0&0&0&0&0&0&0&0&0&0&0&0&0&0&0&0&0&0&0&0&0&0&0&0&0&0\\ 
\\ 
$\phi_{400,7}$&0&0&0&0&0&0&0&0&0&0&0&0&0&0&0&0&0&2&1&0&0&0&0&0&0&0&0\\ 
\\ 
$\phi_{400,43}$&10&9&7&4&4&4&2&1&2&2&2&1&0&1&0&1&0&28&14&10&4&0&4&4&0&0&0\\ 
\\ 
$\phi_{448,9}$&0&0&0&0&0&0&0&0&0&0&0&0&0&0&0&0&0&0&0&0&0&0&0&0&0&0&0\\ 
\\ 
$\phi_{448,39}$&12&9&6&2&4&4&2&0&1&2&3&1&1&1&0&1&0&28&14&10&3&0&8&4&0&0&0\\ 
\\ 
$\phi_{560,5}$&0&0&0&0&0&0&0&0&0&0&0&0&0&0&0&0&0&0&0&0&0&0&0&0&0&0&0\\ 
\\ 
$\phi_{560,47}$&15&13&9&5&7&5&5&2&2&2&3&1&1&1&0&0&0&42&21&15&8&0&8&4&0&0&0\\ 
\\ 
$\phi_{1344,19}$&4&3&2&1&0&0&0&0&0&0&0&0&0&0&0&0&0&28&14&4&1&0&0&0&0&0&0\\ 
\\ 
$\phi_{840,13}$&0&0&0&0&0&0&0&0&0&0&0&0&0&0&0&0&0&0&0&0&0&0&0&0&0&0&0\\ 
\\ 
$\phi_{840,31}$&5&4&2&1&1&1&1&0&0&0&0&0&0&0&0&0&0&28&14&6&3&0&0&0&0&0&0\\ 
\\ 
$\phi_{1008,9}$&0&0&0&0&0&0&0&0&0&0&0&0&0&0&0&0&0&0&0&0&0&0&0&0&0&0&0\\ 
\\ 
$\phi_{1008,39}$&9&8&5&2&3&2&2&1&0&0&1&0&0&0&0&0&0&42&21&11&6&0&4&2&0&0&0\\ 
\\ 
$\phi_{2016,19}$&9&6&3&0&1&1&0&0&0&0&1&0&0&0&0&0&0&42&21&8&0&0&4&2&0&0&0\\ 
\\ 
$\phi_{1296,13}$&0&0&0&0&0&0&0&0&0&0&0&0&0&0&0&0&0&0&0&0&0&0&0&0&0&0&0\\ 
\\ 
$\phi_{1296,33}$&3&3&1&0&1&0&1&0&0&0&0&0&0&0&0&0&0&30&15&6&3&0&0&0&0&0&0\\ 
\\ 
$\phi_{1400,11}$&1&1&0&0&0&0&0&0&0&0&0&0&0&0&0&0&0&14&7&1&0&0&0&0&0&0&0\\ 
\\ 
$\phi_{1400,29}$&16&13&9&3&4&3&1&0&1&2&2&1&0&1&0&0&0&56&28&16&3&0&8&6&0&0&0\\ 
\\ 
$\phi_{1400,7}$&0&0&0&0&0&0&0&0&0&0&0&0&0&0&0&0&0&0&0&0&0&0&0&0&0&0&0\\ 
\\ 
$\phi_{1400,37}$&20&17&11&4&7&5&4&1&1&2&3&1&0&1&0&0&0&70&35&21&9&0&8&6&0&0&0\\ 
\\ 
$\phi_{2400,17}$&0&0&0&0&0&0&0&0&0&0&0&0&0&0&0&0&0&14&7&1&0&0&0&0&0&0&0\\ 
\\ 
$\phi_{2400,23}$&1&1&0&0&0&0&0&0&0&0&0&0&0&0&0&0&0&26&13&3&1&0&0&0&0&0&0\\ 
\\ 
$\phi_{2800,13}$&1&0&0&0&0&0&0&0&0&0&0&0&0&0&0&0&0&14&7&0&0&0&0&0&0&0&0\\ 
\\ 
$\phi_{2800,25}$&7&6&3&1&1&0&0&0&0&0&0&0&0&0&0&0&0&56&28&10&3&0&0&2&0&0&0\\ 
\\ 
$\phi_{5600,19}$&6&4&1&0&0&0&0&0&0&0&0&0&0&0&0&0&0&70&35&8&2&0&0&0&0&0&0\\ 
\\ 
$\phi_{3240,9}$&0&0&0&0&0&0&0&0&0&0&0&0&0&0&0&0&0&12&6&0&0&0&0&0&0&0&0\\ 
\\ 
$\phi_{3240,31}$&30&25&15&6&8&6&4&1&1&2&3&1&0&0&0&0&0&126&63&33&12&0&12&8&0&0&0\\ 
\\ 
$\phi_{3360,13}$&2&1&0&0&0&0&0&0&0&0&0&0&0&0&0&0&0&28&14&2&0&0&0&0&0&0&0\\ 
\\ 
$\phi_{3360,25}$&15&12&6&1&3&2&1&0&0&0&1&0&0&0&0&0&0&84&42&18&4&0&4&4&0&0&0\\ 
\\ 
$\phi_{7168,17}$&14&10&4&0&1&0&0&0&0&0&0&0&0&0&0&0&0&112&56&16&2&0&4&2&0&0&0\\ 
\\ 
$\phi_{4096,11}$&0&0&0&0&0&0&0&0&0&0&0&0&0&0&0&0&0&16&8&0&0&0&0&0&0&0&0\\ 
\\ 
$\phi_{4096,27}$&20&16&8&2&4&2&2&0&0&0&1&0&0&0&0&0&0&112&56&24&8&0&8&4&0&0&0\\ 
\\ 
$\phi_{4200,15}$&7&4&1&0&0&0&0&0&0&0&0&0&0&0&0&0&0&56&28&6&1&0&4&0&0&0&0\\ 
\\ 
$\phi_{4200,21}$&12&9&5&2&1&1&0&0&0&0&0&0&0&0&0&0&0&84&42&14&3&0&4&2&0&0&0\\ 
\\ 
$\phi_{4536,13}$&3&2&1&0&0&0&0&0&0&0&0&0&0&0&0&0&0&42&21&3&0&0&0&0&0&0&0\\ 
\\ 
$\phi_{4536,23}$&27&20&11&3&4&3&1&0&0&2&2&0&0&0&0&0&0&126&63&27&6&0&12&6&0&0&0\\ 
\\ 
$\phi_{5600,15}$&2&1&0&0&0&0&0&0&0&0&0&0&0&0&0&0&0&42&21&3&0&0&0&0&0&0&0\\ 
\\ 
$\phi_{5600,21}$&13&9&4&0&1&1&0&0&0&0&0&0&0&0&0&0&0&98&49&15&4&0&4&2&0&0&0\\ 
\end{longtable} 

}
\end{landscape}

\extendedVersion{
\begin{landscape}
\appendix

\section{Characters and irreducible representations}
This appendix contains data from \cite{Carter:FiniteGroupsLieType} and GAP presented here for the reader's convenience. We include tables only for $G_2, F_4$, and $E_6$. The $E_7$ and $E_8$ tables were too large to fit here; we direct an interested reader to the interface GAP3 for access to the $E_7$ and $E_8$ data.

In the tables below, for a given Weyl group $G$, we assume a fixed set of simple roots $\alpha_1, \dots, \alpha_s$. For readability reasons, we abbreviate the reflection $s_{\alpha_i}$ as $s_i$.
\subsection{$G_2$}~
\renewcommand{\arraystretch}{0}%

\begin{longtable}{r|rrrrrr}\caption{\label{tableIrrepCharsG^{1}_2}\\ Irreducible representation characters. Columns are labeled by conjugacy classes.} \\ 
  & \multicolumn{6}{c}{Conjugacy class $\#$}\\ 
Irrep label&$1$&$2$&$3$&$4$&$5$&$6$\\\hline 
$\phi_{1,6}  $&$1$&$-1$&$-1$&$1$&$1$&$1$\\ 
$\phi_{1,3}''$&$1$&$-1$&$1$&$-1$&$1$&$-1$\\ 
$\phi_{1,3}' $&$1$&$1$&$-1$&$-1$&$1$&$-1$\\ 
$\phi_{1,0}  $&$1$&$1$&$1$&$1$&$1$&$1$\\ 
$\phi_{2,1}  $&$2$&$0$&$0$&$1$&$-1$&$-2$\\ 
$\phi_{2,2}  $&$2$&$0$&$0$&$-1$&$-1$&$2$\\ 
\end{longtable}
 
\begin{longtable}{rrcl}\caption{\label{tableConjugacyClassTableG12}}\\ $\#$ & Representative & Class size & Root subsystem label\\ 
$1$ & $id$&$ 1$&$0$\\ 
$2$ & $s_{1}$&$ 3$&$A^{3}_1$\\ 
$3$ & $s_{2}$&$ 3$&$A_1$\\ 
$4$ & $s_{2}s_{1}$&$ 2$&$G_2$\\ 
$5$ & $s_{2}s_{1}s_{2}s_{1}$&$ 2$&$A_2$\\ 
$6$ & $s_{1}s_{2}s_{1}s_{2}s_{1}s_{2}$&$ 1$&$A^{3}_1+A_1$\\ 
\end{longtable}

\subsection{$F_4$}~

\begin{longtable}{r|rrrrrrrrrrrrrrrrrrrrrrrrr}\caption{\label{tableIrrepCharsF^{1}_4}\\ Irreducible representation characters. Columns are labeled by conjugacy classes.} \\ 
  & \multicolumn{25}{c}{Conjugacy class $\#$}\\ 
Irrep label&$1$&$2$&$3$&$4$&$5$&$6$&$7$&$8$&$9$&$10$&$11$&$12$&$13$&$14$&$15$&$16$&$17$&$18$&$19$&$20$&$21$&$22$&$23$&$24$&$25$\\\hline 
$\phi_{1,0}   $&$1$&$1$&$1$&$1$&$1$&$1$&$1$&$1$&$1$&$1$&$1$&$1$&$1$&$1$&$1$&$1$&$1$&$1$&$1$&$1$&$1$&$1$&$1$&$1$&$1$\\ 
$\phi_{1,12}''$&$1$&$1$&$1$&$1$&$1$&$1$&$1$&$1$&$1$&$1$&$1$&$1$&$1$&$1$&$1$&$1$&$-1$&$-1$&$-1$&$-1$&$-1$&$-1$&$-1$&$-1$&$-1$\\ 
$\phi_{1,12}' $&$1$&$1$&$1$&$1$&$1$&$1$&$1$&$1$&$1$&$1$&$1$&$-1$&$-1$&$-1$&$-1$&$-1$&$1$&$1$&$1$&$1$&$1$&$-1$&$-1$&$-1$&$-1$\\ 
$\phi_{1,24}  $&$1$&$1$&$1$&$1$&$1$&$1$&$1$&$1$&$1$&$1$&$1$&$-1$&$-1$&$-1$&$-1$&$-1$&$-1$&$-1$&$-1$&$-1$&$-1$&$1$&$1$&$1$&$1$\\ 
$\phi_{2,4}'' $&$2$&$2$&$2$&$2$&$2$&$2$&$-1$&$-1$&$-1$&$-1$&$-1$&$2$&$2$&$-1$&$-1$&$2$&$0$&$0$&$0$&$0$&$0$&$0$&$0$&$0$&$0$\\
\end{longtable}
 
\begin{longtable}{rrcl}\caption{\label{tableConjugacyClassTableF14}}\\ $\#$ & Representative & Class size & Root subsystem label\\ 
$1$ & $id$&$ 1$&$ $\\ 
$2$ & $s_{1}s_{2}s_{1}s_{3}s_{2}s_{1}s_{3}s_{2}s_{3}s_{4}s_{3}s_{2}s_{1}s_{3}s_{2}s_{3}s_{4}s_{3}s_{2}s_{1}s_{3}s_{2}s_{3}s_{4}$&$ 1$&$4A_1$\\ 
$3$ & $s_{2}s_{3}s_{2}s_{3}$&$ 18$&$2A_1$\\ 
$4$ & $s_{2}s_{1}$&$ 32$&$A_2$\\ 
$5$ & $s_{2}s_{3}s_{2}s_{3}s_{4}s_{3}s_{2}s_{1}s_{3}s_{4}$&$ 32$&$D_4$\\ 
$6$ & $s_{3}s_{2}s_{4}s_{3}s_{2}s_{1}s_{3}s_{2}s_{4}s_{3}s_{2}s_{1}$&$ 12$&$D_4(a_1)$\\ 
$7$ & $s_{4}s_{3}$&$ 32$&$A^2_2$\\ 
$8$ & $s_{1}s_{2}s_{1}s_{4}s_{3}s_{2}s_{1}s_{3}s_{2}s_{3}$&$ 32$&$C_3+A_1$\\ 
$9$ & $s_{3}s_{2}s_{1}s_{4}s_{3}s_{2}s_{1}s_{3}s_{2}s_{3}s_{4}s_{3}s_{2}s_{1}s_{3}s_{2}$&$ 16$&$A_2+A^2_2$\\ 
$10$ & $s_{3}s_{2}s_{4}s_{3}s_{2}s_{1}s_{3}s_{2}$&$ 16$&$F_4(a_1)$\\ 
$11$ & $s_{4}s_{3}s_{2}s_{1}$&$ 96$&$F_4$\\ 
$12$ & $s_{1}$&$ 12$&$A_1$\\ 
$13$ & $s_{2}s_{3}s_{2}s_{3}s_{4}s_{3}s_{2}s_{3}s_{4}$&$ 12$&$3A_1$\\ 
$14$ & $s_{1}s_{4}s_{3}$&$ 96$&$A^2_2+A_1$\\ 
$15$ & $s_{4}s_{3}s_{2}$&$ 96$&$C_3$\\ 
$16$ & $s_{2}s_{3}s_{2}s_{1}s_{3}$&$ 72$&$A_3$\\ 
$17$ & $s_{3}$&$ 12$&$A^2_1$\\ 
$18$ & $s_{1}s_{2}s_{1}s_{3}s_{2}s_{1}s_{3}s_{2}s_{3}$&$ 12$&$2A_1+A^2_1$\\ 
$19$ & $s_{2}s_{1}s_{4}$&$ 96$&$A_2+A^2_1$\\ 
$20$ & $s_{3}s_{2}s_{1}$&$ 96$&$B_3$\\ 
$21$ & $s_{2}s_{4}s_{3}s_{2}s_{3}$&$ 72$&$B_2+A_1$\\ 
$22$ & $s_{1}s_{3}$&$ 72$&$A^2_1+A_1$\\ 
$23$ & $s_{3}s_{2}$&$ 36$&$B_2$\\ 
$24$ & $s_{2}s_{3}s_{2}s_{3}s_{4}s_{3}s_{2}s_{1}s_{3}s_{2}s_{4}s_{3}s_{2}s_{1}$&$ 36$&$A_3+A^2_1$\\ 
$25$ & $s_{2}s_{4}s_{3}s_{2}s_{1}s_{3}$&$ 144$&$B_4$\\ 
\end{longtable}

\subsection{$E_6$}~

\renewcommand{\arraystretch}{0}%

\begin{longtable}{r|rrrrrrrrrrrrrrrrrrrrrrrrr}\caption{\label{tableIrrepCharsE^{1}_6}\\ Irreducible representation characters. Columns are labeled by conjugacy classes.} \\ 
  & \multicolumn{25}{c}{Conjugacy class $\#$}\\ 
Irrep label&$1$&$2$&$3$&$4$&$5$&$6$&$7$&$8$&$9$&$10$&$11$&$12$&$13$&$14$&$15$&$16$&$17$&$18$&$19$&$20$&$21$&$22$&$23$&$24$&$25$\\\hline 
$\phi_{1,0}  $&$1$&$1$&$1$&$1$&$1$&$1$&$1$&$1$&$1$&$1$&$1$&$1$&$1$&$1$&$1$&$1$&$1$&$1$&$1$&$1$&$1$&$1$&$1$&$1$&$1$\\ 
$\phi_{1,36} $&$1$&$1$&$1$&$1$&$1$&$1$&$1$&$1$&$1$&$1$&$1$&$1$&$1$&$1$&$1$&$-1$&$-1$&$-1$&$-1$&$-1$&$-1$&$-1$&$-1$&$-1$&$-1$\\ 
$\phi_{10,9} $&$10$&$-6$&$2$&$1$&$-2$&$4$&$2$&$-2$&$0$&$-3$&$0$&$0$&$2$&$1$&$-1$&$0$&$0$&$0$&$0$&$0$&$0$&$0$&$0$&$0$&$0$\\ 
$\phi_{6,1}  $&$6$&$-2$&$2$&$-3$&$3$&$0$&$2$&$0$&$1$&$1$&$1$&$-2$&$-1$&$0$&$-1$&$4$&$0$&$-2$&$2$&$1$&$-2$&$0$&$0$&$-1$&$1$\\ 
$\phi_{6,25} $&$6$&$-2$&$2$&$-3$&$3$&$0$&$2$&$0$&$1$&$1$&$1$&$-2$&$-1$&$0$&$-1$&$-4$&$0$&$2$&$-2$&$-1$&$2$&$0$&$0$&$1$&$-1$\\
\end{longtable}
 
\begin{longtable}{rrcl}\caption{\label{tableConjugacyClassTableE16}}\\ $\#$ & Representative & Class size & Root subsystem label\\ 
$1$ & $id$&$ 1$&$0$\\ 
$2$ & $s_{2}s_{3}s_{4}s_{2}s_{3}s_{4}s_{5}s_{4}s_{2}s_{3}s_{4}s_{5}$&$ 45$&$4A_1$\\ 
$3$ & $s_{1}s_{4}$&$ 270$&$2A_1$\\ 
$4$ & $s_{1}s_{3}s_{1}s_{4}s_{2}s_{3}s_{1}s_{4}s_{5}s_{4}s_{2}s_{3}s_{1}s_{4}s_{3}s_{5}s_{6}s_{5}s_{4}s_{2}s_{3}s_{4}s_{5}s_{6}$&$ 80$&$3A_2$\\ 
$5$ & $s_{1}s_{3}$&$ 240$&$A_2$\\ 
$6$ & $s_{1}s_{3}s_{5}s_{6}$&$ 480$&$2A_2$\\ 
$7$ & $s_{3}s_{4}s_{2}s_{3}s_{4}s_{5}$&$ 540$&$D_4(a_1)$\\ 
$8$ & $s_{1}s_{4}s_{3}s_{6}$&$ 3240$&$A_3+A_1$\\ 
$9$ & $s_{1}s_{4}s_{2}s_{3}$&$ 5184$&$A_4$\\ 
$10$ & $s_{1}s_{2}s_{3}s_{1}s_{5}s_{4}s_{6}s_{5}s_{4}s_{2}s_{3}s_{4}$&$ 720$&$E_6(a_2)$\\ 
$11$ & $s_{3}s_{4}s_{2}s_{5}$&$ 1440$&$D_4$\\ 
$12$ & $s_{1}s_{2}s_{3}s_{4}s_{2}s_{3}s_{4}s_{6}s_{5}s_{4}s_{2}s_{3}s_{4}s_{5}$&$ 1440$&$A_5+A_1$\\ 
$13$ & $s_{1}s_{2}s_{3}s_{5}$&$ 2160$&$A_2+2A_1$\\ 
$14$ & $s_{1}s_{3}s_{4}s_{2}s_{3}s_{4}s_{5}s_{6}$&$ 5760$&$E_6(a_1)$\\ 
$15$ & $s_{1}s_{4}s_{2}s_{3}s_{6}s_{5}$&$ 4320$&$E_6$\\ 
$16$ & $s_{1}$&$ 36$&$A_1$\\ 
$17$ & $s_{1}s_{4}s_{6}$&$ 540$&$3A_1$\\ 
$18$ & $s_{1}s_{2}s_{3}s_{4}s_{2}s_{3}s_{4}s_{5}s_{4}s_{2}s_{3}s_{4}s_{5}$&$ 540$&$A_3+2A_1$\\ 
$19$ & $s_{1}s_{4}s_{3}$&$ 1620$&$A_3$\\ 
$20$ & $s_{1}s_{2}s_{3}$&$ 1440$&$A_2+A_1$\\ 
$21$ & $s_{1}s_{2}s_{3}s_{5}s_{6}$&$ 1440$&$2A_2+A_1$\\ 
$22$ & $s_{1}s_{4}s_{3}s_{6}s_{5}$&$ 4320$&$A_5$\\ 
$23$ & $s_{1}s_{3}s_{4}s_{2}s_{5}$&$ 6480$&$D_5$\\ 
$24$ & $s_{1}s_{4}s_{2}s_{3}s_{6}$&$ 5184$&$A_4+A_1$\\ 
$25$ & $s_{1}s_{4}s_{2}s_{5}s_{4}s_{2}s_{3}$&$ 4320$&$D_5(a_1)$\\ 
\end{longtable}

\end{landscape}

} 

\hidemeFinal{
\begin{landscape}
\subsection{$E_7$}

\begin{longtable}{rl}\caption{\label{tableIrrepCharsE^{1}_7}\\ Irreducible representations and their characters}\\ 
Irrep label & Character\\ 
$\phi_{1,0}$&$(1, 1, 1, 1, 1, 1, 1, 1, 1, 1, 1, 1, 1, 1, 1, 1, 1, 1, 1, 1, 1, 1, 1, 1, 1, 1, 1, 1, 1, 1, 1, 1, 1, 1, 1, 1, 1, 1, 1, 1, 1, 1, 1, 1, 1, 1, 1, 1, 1, 1, 1, 1, 1, 1, 1, 1, 1, 1, 1, 1)$\\ 
$\phi_{1,63}$&$(1, 1, 1, 1, 1, 1, 1, 1, 1, 1, 1, 1, 1, 1, 1, 1, 1, 1, 1, 1, 1, 1, 1, 1, 1, 1, 1, 1, 1, 1, -1, -1, -1, -1, -1, -1, -1, -1, -1, -1, -1, -1, -1, -1, -1, -1, -1, -1, -1, -1, -1, -1, -1, -1, -1, -1, -1, -1, -1, -1)$\\ 
$\phi_{7,46}$&$(7, -5, -1, 3, -1, 4, -2, 1, 3, 1, -3, -1, 1, 2, -2, 2, 2, 0, 1, -1, -1, 0, -1, 1, 1, 0, -2, 0, 0, -1, 7, -5, -1, 3, -1, 4, -2, 1, 3, 1, -3, -1, 1, 2, -2, 2, 2, 0, 1, -1, -1, 0, -1, 1, 1, 0, -2, 0, 0, -1)$\\ 
$\phi_{7,1}$&$(7, -5, -1, 3, -1, 4, -2, 1, 3, 1, -3, -1, 1, 2, -2, 2, 2, 0, 1, -1, -1, 0, -1, 1, 1, 0, -2, 0, 0, -1, -7, 5, 1, -3, 1, -4, 2, -1, -3, -1, 3, 1, -1, -2, 2, -2, -2, 0, -1, 1, 1, 0, 1, -1, -1, 0, 2, 0, 0, 1)$\\ 
$\phi_{15,28}$&$(15, -5, 7, 3, -1, 0, -3, 3, -1, -3, 1, 3, 1, 0, -2, -2, 1, 0, 1, 1, -1, 1, 1, -1, 0, 0, 0, -2, -1, 0, 15, -5, 7, 3, -1, 0, -3, 3, -1, -3, 1, 3, 1, 0, -2, -2, 1, 0, 1, 1, -1, 1, 1, -1, 0, 0, 0, -2, -1, 0)$\\ 
$\phi_{15,7}$&$(15, -5, 7, 3, -1, 0, -3, 3, -1, -3, 1, 3, 1, 0, -2, -2, 1, 0, 1, 1, -1, 1, 1, -1, 0, 0, 0, -2, -1, 0, -15, 5, -7, -3, 1, 0, 3, -3, 1, 3, -1, -3, -1, 0, 2, 2, -1, 0, -1, -1, 1, -1, -1, 1, 0, 0, 0, 2, 1, 0)$\\
\end{longtable}

\begin{longtable}{rcl}\caption{\label{tableConjugacyClassTableE17}}\\ Representative & Class size & Root subsystem label\\ 
$id$&$ 1$&$ $\\ 
$s_{2}s_{3}s_{4}s_{2}s_{3}s_{4}s_{5}s_{4}s_{2}s_{3}s_{4}s_{5}s_{6}s_{5}s_{4}s_{2}s_{3}s_{4}s_{5}s_{6}s_{7}s_{6}s_{5}s_{4}s_{2}s_{3}s_{4}s_{5}s_{6}s_{7}$&$ 63$&$6A_1$\\ 
$s_{2}s_{3}s_{4}s_{2}s_{3}s_{4}s_{5}s_{4}s_{2}s_{3}s_{4}s_{5}$&$ 315$&$4A_1''$\\ 
$s_{5}s_{7}$&$ 945$&$2A_1$\\ 
$s_{2}s_{3}s_{5}s_{7}$&$ 3780$&$4A_1'$\\ 
$s_{7}s_{6}$&$ 672$&$A_2$\\
\end{longtable}

\subsection{$E_8$}
\begin{longtable}{rl}\caption{\label{tableIrrepCharsE^{1}_8}\\ Irreducible representations and their characters}\\ 
Irrep label & Character\\ 
$\phi_{1,0}$&$(1, 1, 1, 1, 1, 1, 1, 1, 1, 1, 1, 1, 1, 1, 1, 1, 1, 1, 1, 1, 1, 1, 1, 1, 1, 1, 1, 1, 1, 1, 1, 1, 1, 1, 1, 1, 1, 1, 1, 1, 1, 1, 1, 1, 1, 1, 1, 1, 1, 1, 1, 1, 1, 1, 1, 1, 1, 1, 1, 1, 1, 1, 1, 1, 1, 1, 1, 1, 1, 1, 1, 1, 1, 1, 1, 1, 1, 1, 1, 1, 1, 1, 1, 1, 1, 1, 1, 1, 1, 1, 1, 1, 1, 1, 1, 1, 1, 1, 1, 1, 1, 1, 1, 1, 1, 1, 1, 1, 1, 1, 1, 1)$\\ 
$\phi_{1,120}$&$(1, 1, 1, 1, 1, 1, 1, 1, 1, 1, 1, 1, 1, 1, 1, 1, 1, 1, 1, 1, 1, 1, 1, 1, 1, 1, 1, 1, 1, 1, 1, 1, 1, 1, 1, 1, 1, 1, 1, 1, 1, 1, 1, 1, 1, 1, 1, 1, 1, 1, 1, 1, 1, 1, 1, 1, 1, 1, 1, 1, 1, 1, 1, 1, 1, -1, -1, -1, -1, -1, -1, -1, -1, -1, -1, -1, -1, -1, -1, -1, -1, -1, -1, -1, -1, -1, -1, -1, -1, -1, -1, -1, -1, -1, -1, -1, -1, -1, -1, -1, -1, -1, -1, -1, -1, -1, -1, -1, -1, -1, -1, -1)$\\ 
$\phi_{28,8}$&$(28, 28, -4, 4, 4, 4, -4, 10, 10, 10, 10, 1, 1, 1, 1, 8, 8, 0, 0, 0, 0, 0, 3, 3, 3, 3, 2, 2, 2, -2, -2, -2, 5, 5, -1, -1, 1, 1, 1, -1, 0, 0, -2, 2, 2, 1, 1, 1, 1, -1, -1, -1, 2, 2, 2, 2, -1, -1, 0, 0, 0, 0, 0, 0, 0, 14, 14, -2, -2, 6, 6, -2, 2, 2, 2, -2, 2, 2, 5, 5, -2, -2, -1, -1, 1, 1, 1, 1, 0, 0, 0, 0, -1, -1, 4, 4, 0, 0, 1, -1, -1, -1, 0, 0, -1, -1, 1, 1, 0, 0, 2, 2)$\\ 
$\phi_{28,68}$&$(28, 28, -4, 4, 4, 4, -4, 10, 10, 10, 10, 1, 1, 1, 1, 8, 8, 0, 0, 0, 0, 0, 3, 3, 3, 3, 2, 2, 2, -2, -2, -2, 5, 5, -1, -1, 1, 1, 1, -1, 0, 0, -2, 2, 2, 1, 1, 1, 1, -1, -1, -1, 2, 2, 2, 2, -1, -1, 0, 0, 0, 0, 0, 0, 0, -14, -14, 2, 2, -6, -6, 2, -2, -2, -2, 2, -2, -2, -5, -5, 2, 2, 1, 1, -1, -1, -1, -1, 0, 0, 0, 0, 1, 1, -4, -4, 0, 0, -1, 1, 1, 1, 0, 0, 1, 1, -1, -1, 0, 0, -2, -2)$\\
\end{longtable} 
\begin{longtable}{rcl}\caption{\label{tableConjugacyClassTableE18}}\\ Representative & Class size & Root subsystem label\\ 
$id$&$ 1$&$ $\\ 
$s_{1}s_{2}s_{3}s_{1}s_{4}s_{2}s_{3}s_{1}s_{4}s_{3}s_{5}s_{4}s_{2}s_{3}s_{1}s_{4}s_{3}s_{5}s_{4}s_{2}s_{6}s_{5}s_{4}s_{2}s_{3}s_{1}s_{4}s_{3}s_{5}s_{4}s_{2}s_{6}s_{5}s_{4}s_{3}s_{1}s_{7}s_{6}s_{5}s_{4}s_{2}s_{3}s_{1}s_{4}s_{3}s_{5}s_{4}s_{2}s_{6}s_{5}s_{4}s_{3}s_{1}s_{7}s_{6}s_{5}s_{4}s_{2}s_{3}s_{4}s_{5}s_{6}s_{7}s_{8}s_{7}s_{6}s_{5}s_{4}s_{2}s_{3}s_{1}s_{4}s_{3}s_{5}s_{4}s_{2}s_{6}s_{5}s_{4}s_{3}s_{1}s_{7}s_{6}s_{5}s_{4}s_{2}s_{3}s_{4}s_{5}s_{6}s_{7}s_{8}s_{7}s_{6}s_{5}s_{4}s_{2}s_{3}s_{1}s_{4}s_{3}s_{5}s_{4}s_{2}s_{6}s_{5}s_{4}s_{3}s_{1}s_{7}s_{6}s_{5}s_{4}s_{2}s_{3}s_{4}s_{5}s_{6}s_{7}s_{8}$&$ 1$&$8A_1$\\ 
$s_{2}s_{3}s_{4}s_{2}s_{3}s_{4}s_{5}s_{4}s_{2}s_{3}s_{4}s_{5}$&$ 3150$&$4A_1'$\\ 
$s_{1}s_{6}$&$ 3780$&$2A_1$\\ 
$s_{2}s_{3}s_{4}s_{2}s_{3}s_{4}s_{5}s_{4}s_{2}s_{3}s_{4}s_{5}s_{6}s_{5}s_{4}s_{2}s_{3}s_{4}s_{5}s_{6}s_{7}s_{6}s_{5}s_{4}s_{2}s_{3}s_{4}s_{5}s_{6}s_{7}$&$ 3780$&$6A_1$\\ 
$s_{4}s_{2}s_{3}s_{4}s_{5}s_{4}s_{3}s_{1}s_{6}s_{5}s_{4}s_{2}s_{3}s_{1}s_{4}s_{3}s_{5}s_{4}s_{2}s_{6}s_{5}s_{7}s_{8}s_{7}s_{6}s_{5}s_{4}s_{2}s_{3}s_{1}s_{4}s_{3}s_{5}s_{4}s_{2}s_{6}s_{5}s_{4}s_{3}s_{1}s_{7}s_{6}s_{5}s_{4}s_{2}s_{3}s_{4}s_{5}s_{6}s_{8}s_{7}s_{6}s_{5}s_{4}s_{2}s_{3}s_{1}s_{4}s_{3}s_{5}$&$ 15120$&$2D_4(a_1)$\\
\end{longtable} 
 
\end{landscape}
}
}

\end{document}